%% file: hf+.tex
\documentclass [11pt]{amsart}
\usepackage{amsmath, amssymb, amscd, mathrsfs, stmaryrd, color, url,epsf,enumerate, 
  mathtools, ifpdf}
\usepackage[all,knot,arc,dvips]{xy}
\ifpdf
  \usepackage[pdftex]{graphicx}
  \usepackage[bookmarks=true, bookmarksopen=true,%
    bookmarksdepth=3,bookmarksopenlevel=2,%
    colorlinks=true,%
    linkcolor=blue,%
    citecolor=blue,%
    filecolor=blue,%
    menucolor=blue,%
    pagecolor=blue,%
    urlcolor=blue]{hyperref}
\else
  \usepackage[dvips]{graphicx}
  \usepackage[draft]{hyperref}
\fi
\usepackage[lmargin=1in,rmargin=1in,tmargin=1in,bmargin=1in]{geometry}
\usepackage[T1]{fontenc}
\usepackage[utf8]{inputenc}

\newtheorem {theorem}{Theorem}[section]
\newtheorem {lemma}[theorem]{Lemma}
\newtheorem {proposition}[theorem]{Proposition}

\newtheorem {conjecture}[theorem]{Conjecture}
\newtheorem {definition}[theorem]{Definition}

\theoremstyle{remark}
\newtheorem {remark}[theorem]{Remark}
\newtheorem {example}[theorem]{Example}

\newcommand\eps{\varepsilon}

\renewcommand\th{^{\text{th}}}

\newcommand\Z{\mathbb{Z}}
\newcommand\R{\mathbb{R}}
\newcommand\Q{\mathbb{Q}}

\def\H{\mathbb{H}}
\def\dd {{\mathbf {d}}}
\newcommand \bH {\overline{\H}}
\newcommand\T{\mathbb{T}}
\newcommand\Ta{\mathbb{T}_\alpha}
\newcommand\Tb{\mathbb{T}_\beta}

\newcommand\Xs{\mathbb{X}}
\newcommand\Os{\mathbb{O}}

\newcommand\Rect{\mathrm{Rect}}

\def\Sym{\mathrm{Sym}}

\def\CP{\mathit{CP}}
\def\HP{\mathit{HP}}
\def\Os{\mathbb O}

\def\zz {{\mathbb{Z}}}
\def\rr {{\mathbb{R}}}

\def\TT {{\mathcal{T}}}
\def\del {{\partial}}
\def\tt {{\mathfrak{t}}}
\def\ss {{\mathfrak{s}}}
\def\C {{\mathcal{C}}}
\def\zed {{\mathcal{Z}}}
\def\lk {{\operatorname{lk}}}
\def\spc {{\operatorname{Spin^c}}}

\def\fin\qedhere
\def\pr {{\text{pr}}}
\DeclareMathOperator{\id}{id}

\DeclareMathOperator{\Hom}{Hom}
\def\precdot {\prec}
\def\from {{\leftarrow}}
\def\hey {\! \sslash \!}

\def\CFK{\textit{CFK}}

\def\CFKi{\CFK^{\infty}}
\def\M {\mathcal {M}}

\def\s{\mathbf s}
\def\t{\mathbf t}
\def\x{\mathbf x}
\def\y{\mathbf y}

\def\Chain {\mathfrak A}
\def\Ring {\mathcal R}
\def\K {\mathcal K}
\newcommand\orL{\vec{L}}
\newcommand\orN{\vec{N}}
\newcommand\orM{\vec{M}}
\newcommand\Mas{\mu}
\def\S{\mathbf S}
\newcommand\alphas{\mbox{\boldmath$\alpha$}}
\newcommand\betas{\mbox{\boldmath$\beta$}}

\def\Hyper{\mathcal{H}}

\def\E{\mathbb{E}}

\def\Field{\mathbb F}

\newcommand\EmptyRect{\Rect^\circ}

\def\cx {\mathfrak{c}}

\def\ff {{\mathbb{F}}}
\def\F {{\mathcal{F}}}
\def\G {{\mathcal{G}}}
\def\S {{\mathbf{S}}}
\def\ES {{\mathbf{ES}}}
\def\p {{\mathbf{p}}}
\def\grcp {{\operatorname{gr}_{\F}\CP^*(G, \zed) }}
\def\I {{\mathcal{I}}}
\def\ux {{\mathfrak{u}}}
\def\can {{\operatorname{can}}}

\def\delt{{\mathfrak{d}}}
\def\D {{\mathcal{D}}}
\def\De {{D}}
\def\sh {{\mathit{Sh}}}
\def\snail {{\operatorname{sn}}}
\def\mix {{\operatorname{mix}}}

\def\ker {{\operatorname{Ker}}}
\def\im {{\operatorname{Im}}}
\def\coker{{\operatorname{Coker}}}

\def\Tors{{\operatorname{Tors}}}
\def\fin\qedhere
\def\from {{\leftarrow}}
\newcommand\HFm{\mathbf{HF}^-}
\newcommand\CFm{\mathbf{CF}^-}
\newcommand\HFinf{\mathbf{HF}^\infty}
\newcommand\CFinf{\mathbf{CF}^\infty}

\newcommand{\HE}{\mathit{HE}}
\newcommand{\CE}{\mathit{CE}}
\newcommand{\EC}{\mathit{EC}}
\newcommand{\iHF}{\mathit{HF}}
\newcommand{\iCF}{\mathit{CF}}
\def\Am {\mathbf A^{\! \!-}}
\def\Cc {\mathfrak{C}}

\def\Ccint {\Cc_{\operatorname{int}}}

\def\Ringbig{\mathfrak{R}}

\def\tU{\widetilde{U}}

\def\zero{\mathbf{0}}

\def\FF{\mathfrak{F}}

\def\U {{\mathbf{U}}}
\newcommand\Tbp{\mathbb{T}_{\beta'}}
\def\CFLc{\mathcal{CFL}}
\def\ECFLc{\mathcal{ECFL}}
\def\prelim{\operatorname{prelim}}
\def\tR{\widetilde{\Ring}}
\def\tN{\widetilde{N}}
\def\ms{\mathbf{m}}

\newcommand{\leftexp}[2]{{\vphantom{#2}}^{#1}{#2}}
\begin{document}

\title{Grid diagrams and Heegaard Floer invariants}

\author[Ciprian Manolescu]{Ciprian Manolescu}
\thanks {CM was supported by NSF grants DMS-0852439, DMS-1104406, DMS-1402914, and a Royal Society University Research Fellowship.}
\address {Department of Mathematics, Stanford University\\ 
Stanford, CA 94305}
\email {cm5@stanford.edu}

\author [Peter S. Ozsv\'ath]{Peter S. Ozsv\'ath}
\thanks{PSO was supported by NSF grant numbers DMS-0804121, DMS-1258274 and a Guggenheim Fellowship.}
\address {Department of Mathematics, Princeton University\\ 
Princeton, NJ 08544}
\email {petero@math.princeton.edu}

\author [Dylan P. Thurston]{Dylan P. Thurston}
\thanks{DPT was supported by a Sloan Research Fellowship.}
\address {Department of Mathematics, Indiana University\\ 
Bloomington, IN 47405}
\email {dpthurst@indiana.edu}

\begin {abstract}
  We give combinatorial descriptions of the Heegaard Floer homology
  groups for arbitrary three-manifolds (with coefficients in
  $\Z/2\Z$).  The descriptions are based on presenting the
  three-manifold as an integer surgery on a link in the three-sphere,
  and then using a grid diagram for the link. We also give
  combinatorial descriptions of the mod 2 Ozsv\'ath-Szab\'o mixed invariants
  of closed four-manifolds, also in terms of grid diagrams.
\end {abstract}

\maketitle

\section {Introduction}

Starting with the seminal work of Donaldson \cite{Donaldson}, gauge
theory has found numerous applications in low-dimensional
topology. Its role is most important in dimension four, where the
Donaldson invariants \cite{DonaldsonPolynomials}, and later the
Seiberg-Witten invariants~\cite{SW1,SW2,Witten}, were used to
distinguish between homeomorphic four-manifolds that are not
diffeomorphic. More recently, Ozsv\'ath and Szab\'o introduced
Heegaard Floer theory~\cite{HolDisk,HolDiskTwo}, an invariant for
low-dimensional manifolds inspired by gauge theory, but defined using
methods from symplectic geometry. Heegaard Floer invariants in
dimension three are known to detect the Thurston norm \cite{GenusBounds}
and fiberedness \cite{Ni3}.  Heegaard Floer homology can be used to
construct various four-dimensional invariants as
well~\cite{HolDiskFour, Zemke}. Notable are the so-called mixed invariants,
which are conjecturally the
same as (the mod two reduction of)
the Seiberg-Witten invariants and, further, are known to
share many of their properties: in particular, they are able to
distinguish homeomorphic four-manifolds with different smooth
structures. In a different direction, there also exist Heegaard Floer
invariants for null-homologous knots and links in three-manifolds
(see~\cite{Knots,RasmussenThesis,Links}); these have applications to
knot theory.

One feature shared by the Donaldson, Seiberg-Witten, and Heegaard
Floer invariants is that their original definitions are based on
counting solutions to some nonlinear partial differential
equations. This makes it difficult to exhibit algorithms which, given
a combinatorial presentation of the topological object (for example, a
triangulation of the manifold, or a diagram of the knot), calculate
the respective invariant. The first such general algorithms appeared
in 2006, in the setting of Heegaard Floer theory.  Sarkar and
Wang~\cite{SarkarWang} gave an algorithm for calculating the hat
version of Heegaard Floer homology of three-manifolds.  The
corresponding maps induced by simply connected cobordisms were
calculated in~\cite{LMW}.  In a different direction, all versions of
the Heegaard Floer homology for links in $S^3$ were found to be
algorithmically computable using {\em grid diagrams}, see
\cite{MOS}. See also~\cite{CubeResolutions,OSS, LOThf} for other
developments. This progress notwithstanding, a combinatorial
description of the smooth, closed four-manifold invariants remained
elusive.

Our aim in this paper is to present combinatorial descriptions of all
(completed) versions of Heegaard Floer homology for three-manifolds,
as well as of the mixed invariants of closed
four-manifolds. It should be noted that we only
work with invariants defined over the field $\Field = \zz/2\zz$.
However, we expect a sign refinement of our descriptions to be
possible.

Our strategy is to use (toroidal) grid diagrams to represent links,
and to represent
three- and four-manifolds in terms of surgeries on those links. Grid
diagrams were previously used in \cite{MOS} to give a combinatorial
description of link Floer homology. From here one automatically
obtains a combinatorial description of the Heegaard Floer homology of
Dehn surgeries on knots in $S^3$, since these are known to be
determined by the knot Floer complex,
cf.~\cite{IntSurg,RatSurg}. Since every three-manifold can be
expressed as surgery on a link in $S^3$, it is natural to pursue a
similar approach for links.

In ~\cite{LinkSurg}, the Heegaard Floer homology of
an integral surgery on a link is described in terms of some data
associated to Heegaard diagrams for the link and its sublinks. In particular, we can
consider this description in the case where the Heegaard diagrams come
from a grid diagram for the link. For technical reasons, we need to consider slightly different grid diagrams than the ones used in \cite{MOS} and \cite{MOST}. Whereas the grids in \cite{MOS} and \cite{MOST} had one $O$ marking and one $X$ marking in each row and column, here we will use grids with at least one extra free $O$ marking, that is, a marking such that there are no other markings in the same row or column. (An example of a grid diagram with four free markings is shown on the right hand side of Figure~\ref{fig:sparse}.)

In the setting of grids with at least one free marking, the
problem of computing the Heegaard Floer homology of surgeries boils
down to counting isolated pseudo-holomorphic polygons in the symmetric
product of the grid torus. Pseudo-holomorphic bigons of index one are
easy to count, as they correspond bijectively to empty rectangles on
the grid, cf.~\cite{MOS}. In general, one needs to count
$(k+2)$-gons of index $1-k$ that relate the Floer complex of the grid
to those of its destabilizations at $k$ points, where $k \geq 0$.
 
Just as in the case of bigons, one can associate to each polygon a
certain object on the grid, which we call an {\em enhanced
  domain}. Roughly, an enhanced domain consists of an ordinary domain
on the grid plus some numerical data at each destabilization
point. (See Definition~\ref{def:EnhancedDomain} below.) Further, when
the enhanced domain is associated to a pseudo-holomorphic polygon, it
satisfies certain positivity conditions. (See
Definition~\ref{def:PositiveEnhancedDomain} below.) Thus, the problem
of counting pseudo-holomorphic polygons reduces to finding the
positive enhanced domains of the appropriate index, and counting the
number of their pseudo-holomorphic representatives.

When $k=1$, this problem is almost as simple as in the case $k=0$.
Indeed, the only positive enhanced domains for $k=1$ are the
``snail-like'' ones used to construct the destabilization map in
\cite[Section 3.2]{MOST}, see Figure~\ref{fig:ULeft} below. It is not
hard to check that each such domain has exactly one pseudo-holomorphic
representative, modulo $2$.

The key fact that underlies the $k=1$ calculation is that in that case
there exist no positive enhanced domains of index $-k$ corresponding
to destabilization at $k$ points. Hence, the positive domains of index
$1-k$ are what is called \emph{indecomposable}: in particular, their
counts of pseudo-holomorphic representatives depend only on the
topology of the enhanced domain (i.e., they are independent of its
conformal structure). Unfortunately, this fails to be true for $k \geq
2$: there exist positive enhanced domains of index $-k$, so the counts
in index $1-k$ depend on the almost complex structure on the symmetric
product of the grid.

Nevertheless, we know that different almost complex structures give
rise to chain homotopy equivalent Floer complexes for the surgeries on
the link, though they may give different counts for 
particular domains. Let us think of an almost complex structure as a
way of assigning a count (zero or one, modulo $2$) to each positive
enhanced domain on the grid, in a compatible way. Of course, there may
exist other such assignments: we refer to an assignment that satisfies
certain compatibility conditions as a {\em formal complex
  structure}. (See Definition~\ref{def:FormalComplexStructure} below.)
Each formal complex structure produces a model for the Floer complex
of the link surgery. If two formal complex structures are homotopic
(in a suitable sense, see Definition~\ref{def:homot}), the resulting
Floer complexes are chain homotopy equivalent. Suppose now we could
show that all formal complex structures on a grid are homotopic. Then,
we would easily arrive at a combinatorial formulation for the Floer
homology of surgeries. Indeed, one could pick an arbitrary formal
complex structure on the grid, which is a combinatorial object; and
the homology of the resulting complex would be the right answer,
whether or not the formal structure came from an actual almost complex
structure on the symmetric product.

The question of whether or not any two formal complex structures on a
grid $G$ are homotopic basically reduces to whether or not a
certain cohomology group $\HE^d(G)$ vanishes in degrees $d < \min\{0, 2-k\}$.
Here, $\HE^*(G)$ is the cohomology of a complex generated by
positive enhanced domains, modulo periodic domains. (See  Definition~\ref{def:FormalComplexStructure}.) Computer
experimentation suggests the following:
\begin {conjecture}
\label {conj:extended}
Let $G$ be a toroidal grid diagram (with at least one free $O$ marking) representing a link $L \subset S^3$. Then $\HE^d(G) = 0$ for any $d < 0$. 
\end {conjecture}

We can prove only a weaker version of this conjecture, but one that is
sufficient for our purposes. This version applies only to grid diagrams that are {\em sparse}, in the sense that if a row contains an $X$ marking, then the adjacent rows do not contain $X$ markings (i.e., they each contain a free $O$ marking); also if a column contains an $X$ marking, then the adjacent columms do not contain $X$ markings. One way to construct a sparse diagram is to start with any grid, and then double its size by interspersing free $O$ markings along a diagonal. The result is called the  {\em sparse double} of the original grid: see Figure~\ref{fig:sparse}. 

\begin{figure}
\begin{center}
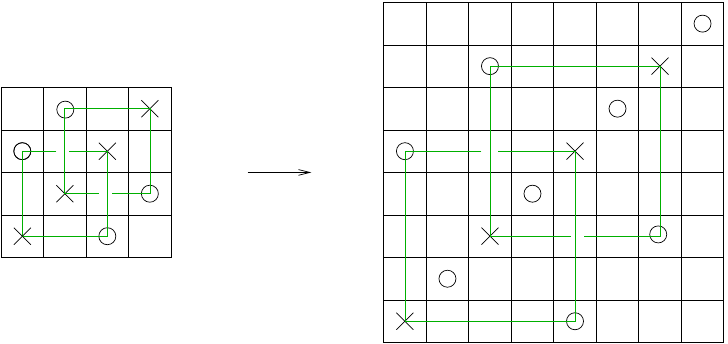
\end{center}
\caption {{\bf The sparse double.} On the left we show an ordinary grid diagram for the Hopf link, with no free markings. On the right is the corresponding sparse double.}
\label{fig:sparse}
\end{figure}

We can show that if $G$ is a sparse grid diagram, then  
$\HE^d(G) = 0$ for $d< \min\{0, 2-k\}$. As a consequence, we obtain:

\begin {theorem}
\label {thm:Three}
Let $\orL \subset S^3$ be an oriented link with a framing $\Lambda$.
Let $S^3_\Lambda(L)$ be the manifold obtained by surgery on $S^3$
along $(L, \Lambda)$, and let $\ux$ be a $\spc$ structure on
$S^3_\Lambda(L)$. Then, given as input a sparse grid diagram for $L$, the
Heegaard Floer homology groups $\HFm(S^3_\Lambda(L), \ux)$,
$\HFinf(S^3_\Lambda(L), \ux)$, $\iHF^+(S^3_\Lambda(L), \ux)$ with
coefficients in $\ff =\zz/2\zz$ are algorithmically computable.
\end {theorem}

The algorithm alluded to in the statement of Theorem~\ref{thm:Three}
is given in the proof.  In the above statement, $\HFm$ and $\HFinf$
are the completed versions of the usual Heegaard Floer homology groups
$\iHF^-$ and $\iHF^{\infty}$, respectively. The groups $\HFm$ and $\HFinf$
are constructed from complexes defined over the power series rings
$\Field[[U]]$ and $\Field[[U, U^{-1}]$, respectively,
cf.\ \cite[Section 2]{LinkSurg}, compare also \cite{KMBook}. More
precisely, the respective chain complexes are
$$ \CFm = \iCF^- \otimes_{\Field[U]} \Field [[U]], \ \ \CFinf = \iCF^{\infty} \otimes_{\Field[U, U^{-1}]} \Field [[U, U^{-1}].$$
One could also define a completed version of $\iHF^+$ in a similar way; but in that case, since multiplication by $U$ is nilpotent  on $\iCF^+$, the completed version is the same as the original one.

The three completed variants of Heegaard Floer homology are related by a long exact sequence:
$$ \cdots \longrightarrow \HFm \longrightarrow \HFinf \longrightarrow \iHF^+ \longrightarrow \cdots.$$

Moving to four dimensions, in \cite{LinkSurg} it is shown that every
closed four-manifold with $b_2^+ > 1$ admits a description as a
three-colored link with a framing. This description is called a {\em cut link
presentation}; see Definition~\ref{def:lp} below.

\begin {theorem}
\label {thm:Four}
Let $X$ be a closed four-manifold with $b_2^+(X) > 1$, and $\ss$ a
$\spc$ structure on $X$. Given as input a cut link presentation $(\orL =
L_1 \cup L_2 \cup L_3, \Lambda)$ for $X$ and a sparse grid diagram
for $L$, one can algorithmically compute the mixed invariant $\Phi_{X,
  \ss}$ with coefficients in $\ff =\zz/2\zz$.
\end {theorem}

Although Theorems~\ref{thm:Three} and \ref{thm:Four} give
combinatorial procedures for calculating the respective invariants, in
practice our algorithms are very far from being effective. This is
especially true since we need to double the size of an ordinary grid in order to
arrive at a sparse one. However, one could just assume the truth of
Conjecture~\ref{conj:extended}, and try to do calculations using any grid with at least one free marking. This is still not effective, and it remains an interesting challenge
to modify the algorithm in such a way as to make it more suitable for
actual calculations. We remark that, in the case of the hat version of
knot Floer homology, the algorithm from \cite{MOS} has been
implemented on computers, see~\cite{BaldwinGillam, Beliakova, Droz};
at present, one can calculate these groups for links with grid number
up to $13$.

Another application of the methods in this paper is to describe combinatorially the pages of the link surgeries spectral sequence from \cite[Theorem 4.1]{BrDCov}; see Section~\ref{sec:spectral} and Theorem~\ref{thm:SpectralComb} below for the relevant discussion. Recall that this spectral sequence can be used to give a relationship between Khovanov homology to the Heegaard Floer homology of branched double covers. In the setting of $\widehat{\iHF}$, the spectral sequence can also be described combinatorially using bordered Floer homology; see \cite{LOTbrcov}, \cite{LOTbrcov2}.

This paper is organized as follows. In Section~\ref{sec:Surgery} we
state the results from \cite{LinkSurg} about how $\HFm$ of an integral
surgery on a link $L$ (and all the related invariants) can be expressed in terms of counts of
holomorphic polygon counts on the symmetric product of a grid (with at least one free marking).  In
Section~\ref{sec:enhanced} we define enhanced domains, and associate
one to each homotopy class of polygons on the symmetric product. We
also introduce the positivity condition on enhanced domains, which is
necessary in order for the domain to have holomorphic
representatives. In Section~\ref{sec:comps} we associate to a grid $G$  
the complex $\CE^*(G)$ whose generators are positive enhanced domains
modulo periodic domains. We show that, if
Conjecture~\ref{conj:extended} is true, then all possible ways of
assigning holomorphic polygon counts to enhanced domains are
homotopic. We explain how this would lead to a combinatorial
description of the Heegaard Floer invariants. In
Section~\ref{sec:sparse} we prove the weaker (but sufficient for our
purposes) version of the conjecture, which is applicable to sparse grid diagrams.

Throughout the paper we work over the field $\Field = \zz/2\zz$.

\medskip \textbf {Acknowledgements.} The first author would like to
thank the mathematics department at the University of Cambridge for
its hospitality during his stay there in the Spring of 2009.  We would
also like to thank Robert Lipshitz and Zolt{\'a}n Szab{\'o} for
interesting conversations, and the referees for comments on the paper.

\section{The surgery theorem applied to grid diagrams}
\label {sec:Surgery}

Our main goal here is to present the statement of
Theorem~\ref{thm:Surgery} below, which expresses the Heegaard Floer
homology $\HFm$ of an integral surgery on a link in terms of a grid
diagram for the link (or, more precisely, in terms of holomorphic
polygon counts on a symmetric product of the grid). We also state
similar results for the cobordism maps on $\HFm$ induced by two-handle
additions (Theorem~\ref{thm:Cobordisms}), for the other completed
versions of Heegaard Floer homology (Theorem~\ref{thm:AllVersions}),
and for the mixed invariants of four-manifolds
(Proposition~\ref{prop:mixed}). 

The proofs of all the results from this section are given in \cite{LinkSurg}. 

\subsection {Hyperboxes of chain complexes}
\label {sec:hyper}
We start by summarizing some homological algebra from \cite[Section 5]{LinkSurg}.

When $f$ is a function, we denote its $n\th$ iterate by $ f^{\circ n}$, i.e., $f^{\circ 0} = id, \ f^{\circ 1} = f, \ f^{\circ (n+1)} = f^{\circ n} \circ f$.  

For $\dd = (d_1, \dots, d_n) \in (\Z_{\ge 0})^n$ a collection of nonnegative integers, we set
$$  \E(\dd) =  \{ \eps = (\eps_1, \dots, \eps_n) \ | \  \eps_i \in \{0,1, \dots, d_i\}, \ i=1, \dots, n \}. $$

In particular, $ \E_n =  \E(1, \dots, 1) = \{0,1\}^n$ is the set of vertices of the $n$-dimensional unit hypercube. 

For $\eps = (\eps_1, \dots, \eps_n)\in \E(\dd)$, set
$$ \| \eps \| = \eps_1 + \dots + \eps_n.$$

We can view the elements of $\E(\dd)$ as vectors in $\R^n$. There is a partial ordering on $\E(\dd)$, given by $\eps' \leq \eps \iff \forall i, \ \eps'_i \leq \eps_i$. We write $\eps' < \eps$ if $\eps' \leq \eps$ and $\eps' \neq \eps$. We say that two multi-indices $\eps, \eps'$ with $\eps \leq \eps'$ are {\em neighbors} if $\eps' - \eps \in \E_n$, i.e., none of their coordinates differ by more than one.

We define an {\em $n$-dimensional hyperbox of chain complexes} $H
= \bigl( (C^{\eps})_{\eps \in \E(\dd)}, (\De^{\eps})_{\eps \in \E_n}
\bigr)$ of size $\dd \in (\Z_{\ge 0})^n$ to consist of a collection of
$\Z$-graded vector spaces 
$$ (C^{\eps})_{\eps \in \E(\dd)}, \ \  C^{\eps} = \bigoplus_{* \in \Z} C^{\eps}_*,$$ 
together with a collection of linear maps
$$ \De^{\eps}  : C_*^{\eps^0} \to C_{*+\| \eps \|-1 }^{\eps^0 + \eps},$$
defined for all $\eps^0 \in \E(\dd)$ and $\eps \in \E_n $ such that $\eps^0 + \eps \in \E(\dd)$ (i.e., the multi-indices of the domain and the target are neighbors). The maps $\De^\eps$ are required to satisfy the relations
\begin {equation}
\label {eq:d2}
 \sum_{\eps' \leq \eps}  \De^{\eps - \eps'} \circ \De^{\eps'}   = 0,
 \end {equation}
for all $\eps \in \E_n$.  If $\dd = (1, \dots, 1)$, we say that $H$ is a {\em hypercube of chain complexes}.

Note that $\De^{\eps}$ in principle also depends on $\eps^0$, but we omit that from notation for simplicity. Further, if we consider the total complex
$$ C_* =  \bigoplus_{\eps \in \E(\dd)} C^{\eps}_{* + \|\eps\|}, $$
we can think of $ \De^{\eps}$ as a map from $C_*$ to itself, by extending it to be zero when is not defined. Observe that $\De = \sum \De^{\eps}: C_* \to C_{*-1}$ is a differential.

Let $H= \bigl( (C^{\eps})_{\eps \in \E(\dd)}, (\De^{\eps})_{\eps \in \E_n} \bigr)$ be an $n$-dimensional hyperbox of chain complexes. Let us imagine the hyperbox $[0, d_1] \times \dots \times [0, d_n]$ to be split into $d_1d_2 \cdots d_n$ unit hypercubes. At each vertex $\eps$ we see a chain complex $(C^{\eps}, D^{(0,\dots, 0)})$. Associated to each edge of one of the unit hypercubes is a chain map. Along the two-dimensional faces we have chain homotopies between the two possible ways of composing the edge maps, and along higher-dimensional faces we have higher homotopies. 

There is a natural way of turning the hyperbox $H$ into an $n$-dimensional hypercube $\hat H = (\hat C^{\eps}, \hat \De^{\eps})_{\eps \in \E_n}$, which we called the {\em compressed hypercube} of $H$. The compressed hypercube has the property that along its $i\th$ edge we see the composition of all the $d_i$ edge maps on the $i\th$ axis of the hyperbox.

In particular, if $n=1$, then $H$ is a string of chain complexes and chain maps: 
$$ C^{(0)} \xrightarrow{\De^{(1)}} C^{(1)}  \xrightarrow{\De^{(1)}}  \cdots  \xrightarrow{\De^{(1)}} C^{(d)}, $$
and the compressed hypercube $\hat H$ is
$$  C^{(0)}  \xrightarrow{\left(\De^{(1)}\right)^{\circ d}} C^{(d)}.$$

For general $n$ and $\dd = (d_1, \dots, d_n)$, the compressed hypercube $\hat H$ has at its vertices the same complexes as those at the vertices of the original hyperbox $H$:
$$ \hat C^{(\eps_1, \dots, \eps_n)} = C^{(\eps_1d_1, \dots, \eps_n d_n)}, \ \eps = (\eps_1, \dots, \eps_n) \in \E_n.$$
    
If along the $i\th$ coordinate axis in $H$ we have the edge maps $f_i=\De^{(0, \dots, 0, 1, 0, \dots, 0)}$, then along the respective axis in $\hat H$ we see $f_i^{\circ d_i}$. Given a two-dimensional face of $H$ with edge maps $f_i$ and $f_j$ and chain homotopies $f_{\{i, j\}}$ between $f_i \circ f_j$ and $f_j \circ f_i$, to the respective compressed face in $\hat H$ we assign the map
$$\sum_{k_i=1}^{d_i} \sum_{k_j=1}^{d_j} f_i^{\circ(k_i-1)} \circ f_j^{\circ (k_j-1)} \circ f_{\{i,j\}} \circ f_j^{\circ (d_j - k_j)} \circ f_i^{\circ (d_i - k_i)},$$
which is  a chain homotopy between $f_i^{\circ d_i} \circ f_j^{\circ d_j}$ and $f_j^{\circ d_j} \circ f_i^{\circ d_i}$. The formulas for what we assign to the higher-dimensional faces in $\hat H$ are more complicated, but they always involve sums of compositions of maps in $H$. 

Let  $\leftexp{0}{H}$ and $\leftexp{1}{H}$ be two hyperboxes
of chain complexes, of the same size $\dd \in (\Z_{\ge 0})^n$. A {\em
  chain map} $F : \leftexp{0}{H} \to \leftexp{1}{H}$  is
defined to be a collection of linear maps 
$$F_{\eps^0}^\eps: \leftexp{0}{C}^{\eps^0}_* \to \leftexp{1}{C}^{\eps^0 + \eps} _{*+ \| \eps \|},$$
satisfying 
$$\sum_{\eps' \leq \eps} \bigl( \De^{\eps - \eps'}_{\eps^0 + \eps'} \circ F^{\eps'}_{\eps^0}  +   F^{\eps - \eps'}_{\eps^0 + \eps'} \circ \De^{\eps'}_{\eps^0} \bigr)  = 0,$$
for all $\eps^0 \in \E(\dd), \eps \in \E_n$ such that $\eps^0 + \eps \in \E(\dd)$. In particular, $F$ gives an ordinary chain map between the total complexes $\leftexp{0}{C}$ and $\leftexp{1}{C}$. 

Starting with from here, we can define chain homotopies and chain homotopy equivalences between hyperboxes by analogy with the usual notions between chain complexes.

The construction of $\hat H$ from $H$ is natural in the following sense. Given a chain map $F : \leftexp{0}{H} \to \leftexp{1}{H}$, there is an induced chain map $\hat F : {}^{0}{\hat H} \to {}^{1}{\hat H}$ between the respective compressed hypercubes. Moreover, if $F$ is a chain homotopy equivalence, then so is $\hat F$.

\subsection {Chain complexes from grid diagrams}
\label {subsec:ags}
Consider an oriented, $\ell$-component link $\orL \subset S^3$. We denote the components of $\orL$ by
$\{ L_i\}^{\ell}_{i=1}$. Let
$$ \H(L)_i = \frac{\lk(L_i, L - L_i)}{2} + \Z \subset \Q, \ \ \H(L) = \bigtimes_{i=1}^\ell \H(L)_i,$$  
where $\lk$ denotes linking number. Further, let
$$ \bH(L)_i = \H(L)_i \cup \{-\infty, + \infty\}, \ \ \bH(L) = \bigtimes_{i=1}^{\ell} \bH(L)_i.$$

Let $G$ be a toroidal grid diagram representing the link $\orL$ and having at least one free marking, as in \cite[Section 15.1]{LinkSurg}. Precisely, $G$
consists of a torus $\TT$, viewed as a square in the plane with the
opposite sides identified, and split into $n$ annuli (called rows) by
$n$ horizontal circles $\alpha_1, \dots, \alpha_n$, and into $n$ other
annuli (called columns) by $n$ vertical circles $\beta_1, \dots,
\beta_n$. Further, we are given some markings on the torus, of two
types: $X$ and $O$, such that:
\begin {itemize}
\item each row and each column contains exactly one $O$ marking;
\item each row and each column contains at most one $X$ marking;
 \item if the row of an $O$ marking contains no $X$ markings, then the column of that $O$ marking  contains no $X$ markings either. An $O$ marking of this kind is called a {\em free marking}. We assume that the number $q$ of free markings is at least $1$.
  \end {itemize}

 Observe that $G$ contains exactly $n$ $O$ markings and exactly $n-q$ $X$ markings. A marking that is not free is called {\em linked}. The number $n$ is called  the {\em grid number} or the {\em size} of $G$. 

  We draw horizontal 
arcs between the linked markings in the same row (oriented to go from the $O$
to the $X$), and vertical arcs between the linked markings in the same column
(oriented to go from the $X$ to the $O$). Letting the vertical arcs be
overpasses whenever they intersect the horizontal arcs, we obtain a
planar diagram for a link in $S^3$, which we ask to be the given link $\orL$. 

We let $\S = \S(G)$ be the set of matchings between the horizontal and vertical circles in $G$.  Any $\x \in \S$ admits a Maslov grading $M(\x) \in \Z$ and an Alexander multi-grading
$$A_i(\x) \in \H(L)_i, \ i \in \{1, \dots, \ell \}.$$ 
(For the precise formulas for $M$ and $A_i$, see \cite{MOST}, where they were written in the context of grid diagrams without free markings. However, the same formulas also apply to the present setting.)

For $\x, \y \in \S$, we let $\EmptyRect(\x, \y)$ be the space of empty
rectangles between $\x$ and $\y$, as in \cite{MOST}.  Specifically, a
{\em rectangle} from $\x$ to $\y$ is an embedded rectangle $r$ in the
torus, whose lower left and upper right corners are coordinates of
$\x$, and whose lower right and upper left corners are coordinates of
$\y$; and moreover, if all other coordinates of $\x$ coincide with all
other coordinates of $\y$. We say that the rectangle is {\em empty} if
its interior contains none of the coordinates of $\x$ (or $\y$).

For $r \in
\EmptyRect(\x, \y)$, we denote by $O_j(r)$ and $X_j (r) \in \{0,1\}$
the number of times $O_j$ (resp.\ $X_j$) appears in the interior of
$r$. Let $\Os_i$ and $\Xs_i$ be the set of $O$'s (resp.\ $X$'s)
belonging to the component $L_i$ of the link. The $i^{th}$ coordinate of the Alexander multi-grading is characterized uniquely up to an overall additive constant by the property that
$$ A_i(\x) - A_i(\y) = \sum_{j \in \Xs_i} X_j(r) - \sum_{j \in \Os_i} O_j(r),$$
where here $r$ is any rectangle from $\x$ to $\y$.

The {\em Floer complex} $C^-(G)$ associated to the grid $G$  is the free module over $$\Ring = \Field[[U_1, \dots, U_n]]$$ generated by the $n!$ elements of $\S$, and endowed with the differential:
\begin{equation}
\label {eq:du}
\partial \x = \sum_{\y \in \S} \sum_{r\in \EmptyRect(\x, \y)} U_1^{O_1(r)} \cdots  U_n^{O_n(r)} \y.
\end {equation}

The Maslov grading $M$ gives the homological grading on $C^-(G)$, such that each variable $U_i$ decreases it by $2$. The Alexander functions $A_i$ give filtrations on $C^-(G)$, such that $U_i$ decreases the filtration level $A_i$ by one, and preserves all filtrations $A_j$ for $j\neq i$.

Note that we can view $C^-(G)$ as a Lagrangian Floer
chain complex $\CFm(\Ta, \Tb)$ in a symplectic manifold, the symmetric product $\Sym^n(\TT)$. The two Lagrangian submanifolds are the tori $\Ta =
\alpha_1 \times \dots \times \alpha_n$, the product of horizontal
circles, and $\Tb = \beta_1 \times \dots \times \beta_n$, the product
of vertical circles.  Empty rectangles are the same as holomorphic strips
of index one in $\Sym^n(\TT)$, with boundaries on $\Ta$ and $\Tb$; compare \cite{MOS}.

Given $\s = (s_1, \dots, s_\ell) \in \bH(L)$, we denote by
$$ \Am(G, \s) \subseteq C^-(G)$$
the subcomplex given by $A_i \leq s_i$, $i=1, \dots, \ell$. 

\begin {remark}
When $L=K$ is a knot, the complex $\Am(G, s)$ is a multi-basepoint, completed version of the
subcomplex $A_s^- = C\{\max(i, j-s) \leq 0\}$ of the knot Floer
complex $\CFKi(Y, K)$, in the notation of \cite{Knots}. A similar complex 
$A_s^+ = C\{\max(i, j-s) \geq 0\}$ appeared in the integer surgery formula in \cite{IntSurg}, stated there in terms of $\iHF^+$ rather than $\HFm$.
\end {remark}

\begin{remark}
\label{rmk:notfree}
Observe that $\Am(G, \s)$ is typically not a free $\Ring$-complex. This is due to the presence of several $O$ markings on the same link component. For example, suppose $G$ represents a knot (that is, $\ell =1$) and has size at least $2$. If $\x \in G$ satisfies $A_1(\x) = s_1 +1$, then $\x$ does not appear in the subcomplex $\Am(G, (s_1))$, but both $U_1\x$ and $U_2 \x$ do, and are related by $U_2(U_1 \x) = U_1(U_2\x)$.
\end{remark}

Following \cite[Section 13]{LinkSurg}, we will construct a resolution of $\Am(G, \s)$, that is, a free $\Ring$-module $\Chain^-(G, \s)$ together with a surjective quasi-isomorphism $\tilde{P}: \Chain^-(G, \s) \to \Am(G, \s)$.

First, we introduce some more notation. For each $i$, we re-label the
elements of $\Os_i$ and $\Xs_i$ as
$$ O_{i,1}, X_{i,1}, O_{i,2}, X_{i,2}, \dots, O_{i, p_i}, X_{i, p_i},$$
in the cyclic order in which they appear on $L_i$. We denote by $O_{i,j}(r), X_{i,j}(r)$ the multiplicities of a rectangle at the given marking, and we will sometimes use the alternate name $U_{i,j}$ for the $U$ variable associated to $O_{i,j}$. Further, we arrange so that the free markings are labeled $O_{n-q+1}, \dots, O_n$. 

Consider the differential graded algebra
$$  \Ringbig^{Y} = \Field[[(U_{i,j})_{\substack{1\leq i \leq \ell, \\ 1 \leq j \leq p_i}} (U_{j})_{n-q+1\leq j \leq p}]][(V_{i,j}, Y_{i,j})_{\substack{1\leq i \leq \ell, \\ 2 \leq j \leq p_i}}]'/(Y_{i,j}^2=0, \del Y_{i,j} = U_{i,j} + U_{i,1}V_{i,j}).$$
Here, the prime signifies that the elements of the algebra are infinite sums of monomials $\ms$ in the $U$, $V$ and $Y$ variables, with the property that, for each $i=1, \dots, \ell$,
\begin{equation}
\label{eq:bddness}
- e_{\ms}(U_{i, 1})+ \sum_{j=2}^{p_i} e_{\ms}(V_{i,j}) 
\text{  is bounded above,} 
\end{equation}
where $e_{\ms}(Z)$ denotes the exponent of a variable $Z$ in the monomial $\ms$.

We define an intermediate complex $\Ccint(G)$, which is the dg module over $\Ringbig^Y$ generated by $\x \in \S$, with differential 
$$
 \del \x = \sum_{\y \in \S}  \sum_{r\in \EmptyRect(\x, \y)}  \prod_{i=1}^\ell  U_{i,1}^{X_{i,1}(r) + \ldots + X_{i,p_i}(r)} V_{i,2}^{O_{i,2}(r)} \dots  V_{i,p_i}^{O_{i,p_i}(r)} \cdot \prod_{j=n-q+1}^n U_{j}^{O_{ j}(r)} \y.
$$

This complex admits filtrations $\FF_i$, for $1 \leq i \leq \ell$, such that
$$ \FF_i(\x) = -A_i(\x), \ \FF_i(U_{i,1})=-1, \ \FF(V_{i,j})=1, \ j \geq 2,$$
and $\FF_i$ takes all the other $U$, $V$ and $Y$ variables to zero (i.e, the action of those variables preserves $\FF_i$).

For $\s=(s_1, \dots, s_\ell) \in \H(L)$, we define $\Chain^-(G, \s)$ to be the filtered part of $\Ccint(G)$ given by $ \FF_i \leq s_i$, $i=1, \dots, \ell.$ The complex $\Chain^-(G, \s)$ is a resolution of $\Am(G, \s)$, via the $\Ring$-linear projection 
\begin{equation}
\label{eq:tildeP}
\tilde{P}: \Chain^-(G, \s) \to \Am(G, \s),
\end{equation} 
$$ \tilde{P}\Bigl(\prod_{i,j} V_{i,j}^{n_{i,j}}  Y_{i,j}^{a_{i,j}} \x \Bigr) = 
\begin{cases}
\prod_{i=1}^{\ell} U_{i,1}^{A_i(\x) - s_i - \sum_{j=2}^{p_i} n_{i,j}} U_{i,2}^{n_{i,2}} \dots U_{i,p_i}^{n_{i,p_i}} \x  & \text{if all } a_{i,j} = 0, \\
 0 & \text{otherwise.}\end{cases}$$
 
Note that $\Chain^-(G, \s)$, unlike $\Am(G, \s)$, is a free $\Ring$-module. Indeed, since only one $U$ variable  (namely, $U_{i,1}$) changes the $\FF_i$-filtration level, for any given $ \prod V_{i,j} \prod Y_{i,j} \cdot \x$ there exist minimal exponents $n_{i,1}$ of the $U_{i,1}$ such  that $\prod U_{i,1}^{n_{i,1}} \prod V_{i,j} \prod Y_{i,j} \cdot \x$ are in $\Chain^-(G, \s)$; these are the free generators. For example, in the situation described in Remark~\ref{rmk:notfree}, in filtration level $s_1$ we have the generators $\x$ and $V\x$ (for some $V= V_{1,j}$), with $\tilde{P} (\x)= U_1 \x$ and $\tilde{P} (V\x)= U_2 \x$.

We will refer to $\Am(G, \s)$, $\Ccint(G)$ and $\Chain^-(G, \s)$ as various {\em generalized Floer complexes} associated to the grid $G$. Further variations of these constructions will be explored in Sections~\ref{sec:inclusions} and \ref{sec:descent2}.

Finally, we mention the {\em curved link Floer complex} that appeared in \cite{ZemkeHFL}. Recall that a {\em curved complex} (or a {\em matrix factorization}) $(C, \del)$ over a ring $\tR$ is a $\tR$-module $C$ with a endomorphism $\del: C \to C$ such that $\del^2=k \cdot \id,$ for some $k$ in the base ring $\tR$. In our case, we take
$$ \tR= \Ring[[(U_{i,j}, \tU_{i,j})_{\substack{1\leq i \leq \ell, \\ 1 \leq j \leq p_i}} (U_{j})_{n-q+1\leq j \leq p}]].$$
We define the curved link Floer complex $\CFLc(G)$ to be the $\tR$-module freely generated by $\x \in \S$, with the endomorphism
$$ \del \x = \sum_{\y \in \S}  \sum_{r\in \EmptyRect(\x, \y)}  \prod_{i=1}^\ell  \prod_{j=1}^{p_i} U_{i,j}^{O_{i,j}(r)} \tU_{i,j}^{X_{i,j}(r)} \cdot \prod_{j=n-q+1}^n U_{j}^{O_{ j}(r)} \y.
$$
We have
\begin{equation}
\label{eq:delcurved}
 \del^2 = \sum_{i=1}^{\ell} (U_{i,1} \tU_{i,1} + \tU_{i,1} U_{i,2} + U_{i,2} \tU_{i,2} + \dots + U_{i,p_i} \tU_{i,p_i} + \tU_{i, p_i} U_{i,1}).
 \end{equation}
See for example \cite[Lemma 2.1]{ZemkeQuasi}; compare also \cite{FOOO1}, \cite{KR1}. The contributions to $\del^2$ come from the index two periodic domains, that is, the rows and columns of the grid.

There is a homological (Maslov) grading on $\CFLc(G)$, just as in the case of $C^-(G)$, which is preserved by the action of the $\tU$ variables.

\subsection {Summary of the construction}
\label {sec:summary}
For the benefit of the reader, before moving further we include here a
short summary of Sections~\ref{sec:subred}--\ref{sec:surgery}
below. The aim of these sections is to be able to state the Surgery
Theorem~\ref{thm:Surgery}, which expresses $\HFm$ of an integral
surgery on a link $\orL \subset S^3$ (with framing $\Lambda$) as the
homology of a certain chain complex $\C^-(G, \Lambda)$ associated to a
grid diagram $G$ for~$\orL$ (with at least one free marking). Given a sublink $M \subseteq L$, we let
$G^{L-M}$ be the grid diagram for $L-M$ obtained from $G$ by deleting
all the rows and columns that support components of $M$. Roughly, the
complex $\C^-(G, \Lambda)$ is built as an $\ell$-dimensional hypercube
of complexes. Each vertex corresponds to a sublink $ M \subseteq L$,
and the chain complex at that vertex is the direct product of generalized 
 Floer complexes $\Chain^-(G^{L-M}, \s)$, over all possible
values $\s$; the reader is encouraged to peek ahead at the
expression~\eqref{eq:cgl}.

The differential on $\C^-(G, \Lambda)$ is a sum of some maps denoted
$\Phi^{\orM}_\s$, which are associated to oriented sublinks $\orM$:
given a sublink $M$, we need to consider all its possible
orientations, not just the one induced from the orientation of
$\orL$. When $M$ has only one component, the maps $\Phi^{\orM}_\s$ are
chain maps going from a generalized Floer complex associated to $G^{L'}$ (for
$L'$ containing $M$) to one associated to $G^{L'-M}$. When $M$ has two
components, $\Phi^{\orM}_\s$ are chain homotopies between different
ways of composing the respective chain maps removing one component at
a time; for more components of
$M$ we get higher homotopies. The maps fit together in a hypercube of
chain complexes, as in Section~\ref{sec:hyper}.

Each map $\Phi^{\orM}_\s$ will be a composition of three kinds of maps, cf. Equation~\eqref{eq:phims} below: a ``projection-inclusion map'' $ \I^{\orM}_\s$, a ``descent map'' $ \hat\De_{p^{\orM}(\s)}^{\orM}$, and an isomorphism $ \Psi_{\psi^{\orM}(\s)}^{\orM}$. (Here, $p^{\orM}$ refers to a natural projection from the set $\H(L)$ to itself, and $\psi^{\orM}$ to a projection from $\H(L)$ to $\H(L-M)$; see Section~\ref{sec:subred}.) The  maps $ \I^{\orM}_\s$ are defined in Section~\ref{sec:inclusions}, and go between different generalized Floer complexes associated to the same grid. The descent maps $\hat \De_{p^{\orM}(\s)}^{\orM}$ are defined in Sections~\ref{sec:descent1}-\ref{sec:descent3}, and go from a generalized Floer complex for a grid $G^{L'}$ to one associated to another diagram, which is obtained from the grid by handlesliding some beta curves over others. Finally, the isomorphisms $\Psi_{\psi^{\orM}(\s)}^{\orM}$ relate these latter complexes to generalized Floer complexes for the smaller grid $G^{L'-M}$; these isomorphisms are defined in Section~\ref{subsec:ks}.

The main difficulty lies in correctly defining the descent maps. When the link $M$ has one component, they are constructed by composing some transition maps (involving just counts of empty rectangles on $G^{L'}$) with some maps that count pseudo-holomorphic triangles. For example, there is a triangle map that corresponds to a single handleslide, done so that the new beta curve encircles a marking $Z \in M$. The type ($O$ or~$X$) of the marking $Z$ is determined by the chosen orientation $\orM$ for $M$. More generally, by counting pseudo-holomorphic polygons with more sides, we define maps corresponding to a whole set $\zed$ of markings; see Section~\ref{sec:PolygonMaps} below. When $\zed$ has two markings, the respective maps count quadrilaterals, and are chain homotopies between triangle maps. In general, the polygon maps fit into a hypercube of chain complexes. To construct the descent maps for a sublink (which we do in Sections~\ref{sec:descent1}-\ref{sec:descent3}), we build a hyperbox out of both transition maps and polygon maps (as well as hybrids of these), and apply the compression procedure mentioned in Section~\ref{sec:hyper}.

\subsection{Sublinks and reduction}
\label {sec:subred}
Suppose that $M\subseteq L$ is a sublink. We
choose an orientation on $M$ (possibly different from the one induced
from $\orL$), and denote the corresponding oriented link by $\orM$. We
let $I_+(\orL, \orM)$ (resp.\ $I_-(\orL, \orM)$) be the set of
indices $i$ such that the component $L_i$ is in $M$ and its
orientation induced from $\orL$ is the same as (resp.\ opposite to) the
one induced from $\orM$. We let 
$$ M_{\pm} = \bigcup_{i \in I_{\pm}(\orL, \orM)} L_i,$$
so that $M$ is the disjoint union of $M_+$ and $M_-$. We also denote by
$$ I(\orL, \orM) = I_+(\orL, \orM) \cup I_-(\orL, \orM)$$
the set of all indices $i$ with $L_i \subseteq M$.

For $i \in \{1, \dots, \ell\}$, we define a map $p_i^{\orM} : \bH(L)_i
\to \bH(L)_i $ by
$$ p_i^{\orM}(s) = 
\begin{cases}
+\infty & \text{ if } i \in I_+(\orL, \orM), \\
-\infty & \text{ if } i \in I_-(\orL, \orM), \\
s & \text{ otherwise.}
\end {cases}$$

Then, for $\s = (s_1, \dots, s_\ell) \in \bH(L)$, we set
 $$p^{\orM} (\s)= \bigl(p_1^{\orM}(s_1), \dots, p_\ell^{\orM}(s_\ell)\bigr).$$
 
 Let $N$ be the complement of the sublink $M$ in $L$. We define the {\em reduction} $r_{\orM}(G)$ of the grid diagram $G$ with respect to the sublink $\orM$ to be obtained from $G$ by deleting the basepoints in $\Xs_i$ for $i \in I_+(\orL, \orM)$, deleting the basepoints in $\Os_i$ for $i\in I_-(\orL, \orM)$, and viewing the basepoints in $\Xs_i$ as free $O$ basepoints, for $i\in I_-(\orL, \orM)$. Note that $r_{\orM}(G)$ is a grid diagram for the link $N$.
  
Thus, if we let $\Xs^{\orM} \subseteq \Xs$ be the subset consisting of the $X$ basepoints on $L - M$, this is exactly the set of remaining $X$ basepoints on $r_{\orM}(G)$. The set of $O$ basepoints on $r_{\orM}(G)$ is 
\begin{equation}
\label{eq:newXO}
 \Os^{\orM} := \Bigl( \Os - \bigcup_{i \in I_-(\orL, \orM)} \Os_i \Bigr) \cup  \bigcup_{i \in I_-(\orL, \orM)} \Xs_i.
 \end{equation}

We define a reduction map 
\begin {equation}
\label {eq:psic}
 \psi^{\orM} : \bH(L) \longrightarrow \bH(N)
 \end {equation}
as follows. The map $\psi^{\orM}$ depends only on the summands $\bH(L)_i$ of $\bH(L)$ corresponding to $L_i \subseteq N$. Each of these $L_i$'s appears in $N$ with a (possibly different) index $j_i$, so there is a corresponding summand $\bH(N)_{j_i}$ of $\bH(N)$.  
We then set
$$ \psi^{\orM}_i : \bH(L)_i \to \bH(N)_{j_i}, \ \ s_i \mapsto s_i - \frac{\lk(L_i, \orM)}{2},$$
where $L_i$ is considered with the orientation induced from $\orL$, while $\orM$ is with its own orientation. We then define $\psi^{\orM}$ to be the direct sum of the maps $\psi_i^{\orM}$, pre-composed with projection  to the relevant factors. Note that $\psi^{\orM} = \psi^{\orM} \circ p^{\orM}$.

The Alexander filtrations behave nicely with respect to reduction maps. If a generator $\x \in \S(G)$ has filtration level
$$ \s = (A_1(\x), \dots, A_\ell(\x)),$$
then the same generator has multi-grading $\psi^{\orM}(\s)$ when viewed as a generator for the reduced grid $r_{\orM}(G)$. Hence, for example, we have an identification of generalized Floer complexes
$$
\begin {CD}
\Am(G, p^{\orM}(\s)) @>{\cong}>> \Am(r_{\orM}(G), \psi^{\orM}(\s)).
\end {CD}
 $$
 
 \subsection{Projection-inclusion maps}
\label {sec:inclusions}
Let $\orM$ be an oriented sublink of $\orL$, as in the previous subsection. We refine the construction of the resolution $\Chain^-(G, \s)$ from Section~\ref{subsec:ags} by introducing a relative version, denoted $\Chain^-(G, \orM, p^{\orM} (\s))$. This version looks like
\begin{itemize}
\item the Floer complex $\CFm$ with respect to $M_+$, and 
\item the intermediate complex $\Ccint$ with respect to $M_-$, and
\item the complex $\Chain^- \subset \Ccint$ with respect to $L-M$.
\end{itemize}
Precisely, we first consider $\Ccint(r_{M_+}(G))$, the complex similar to that denoted $\Ccint(G)$ in Section~\ref{subsec:ags}, but using the diagram $r_{M_+}(G)$. Thus, compared to $\Ccint(G)$, in the definition of $\Ccint(r_{M_+}(G))$ we do not introduce variables $V_{i,j}$ and $Y_{i,j}$ for $i \in I_+(\orL, \orM)$; instead, for such $i$, we use  $O_{i,j}(r)$ as exponents for $U_{i,j}$ in the formula for the differential. The complex $\Ccint(r_{M_+}(G))$ has filtrations $\FF_i$ for $i \not \in I_+(\orL, \orM)$, where we use the convention that the Alexander filtration levels $A_i$ of $\x \in \Ta \cap \Tb$ are the ones from the diagram $G$ (rather than those from $r_{M_+}(G)$, where $\s$ becomes $\psi^{M_+}(\s)$).

We let $\Chain^-(G,\orM, p^{\orM}(\s))$ be the subcomplex of $\Ccint(r_{M_+}(G))$ in filtration degrees
$$ \FF_i \leq s_i \text{ for } L_i \not\subseteq M.$$

We now define a {\em projection-inclusion map}
\begin {equation}
\label {eq:Proj}
 \I^{\orM}_\s : \Chain^-(G, \s) \to \Chain^- (G, \orM, p^{\orM}(\s)).
\end {equation}
This will be the first ingredient in the construction of the maps $\Phi^{\orM}_\s$ advertised in Section~\ref{sec:summary}. The map  $\I^{\orM}_{\s}$ is constructed as the composition of a projection similar to $\tilde P$ from \eqref{eq:tildeP}, but taken only with respect to the indices $i \in I_+(\orL, \orM)$, and the natural inclusion into $\Chain^-(G,\orM, p^{\orM}(\s))$. Specifically, for $\x \in \S$, we set
\begin{multline}
\label{eq:Ims}
\I^{\orM}_{\s}\Bigl(\prod_{i\in I_+(\orL, \orM)} \prod_j V_{i,j}^{n_{i,j}}  Y_{i,j}^{a_{i,j}} \x \Bigr) = \\
\begin{cases}
\prod_{i \in I_+(\orL, \orM)} U_{i,1}^{A_i(\x) - s_i - \sum_{j=2}^{p_i} n_{i,j}} U_{i,2}^{n_{i,2}} \dots U_{i,p_i}^{n_{i,p_i}} \x  & \text{if all } a_{i,j} = 0 \text{ for } i \in I_+(\orL, \orM), \\
 0 & \text{otherwise.}\end{cases}
 \end{multline}
We then extend $\I^{\orM}_{\s}$ to be equivariant with respect to the action of the variables $V_{i,j}$ and $Y_{i,j}$ for $i \not \in  I_+(\orL, \orM)$, as well as to that of all the $U$ variables.

\subsection {Handleslides over a set of markings}
\label {sec:DestabSeveral}
Consider a subset $\zed= \{Z_1, \dots, Z_k \} \subseteq \Xs \cup \Os$ consisting only of linked markings. We say that $\zed$ is {\em consistent} if, for any $i$, at most one of the sets $\zed \cap \Os_i$ and $\zed \cap \Xs_i$ is nonempty. From now on we shall assume that $\zed$ is consistent.

We let $L(\zed) \subseteq L$ be the sublink consisting of those components $L_i$ such that at least one of the markings on $L_i$ is in $\zed$. We orient $L(\zed)$ as $\orL(\zed)$, such that a component $L_i$ is given the orientation coming from $\orL$ when $\zed \cap \Os_i \neq \emptyset$, and is given the opposite orientation when  $\zed \cap \Xs_i \neq \emptyset$.

Let us define a new set of curves $\betas^\zed = \{\beta_j^{\zed}|  j=1, \dots, n\}$ on
the torus $\TT$. Let $j_i$ be the index corresponding to the vertical
circle $\beta_{j_i}$ just to the left of a marking $Z_i \in \zed$. We let
$\beta_{j_i}^{\zed}$ be a circle encircling $Z_i$ and intersecting
$\beta_{j_i}$, as well as the $\alpha$ curve just below $Z_i$, in two
points each; in other words, $\beta_{j_i}^{\zed}$ is obtained from
$\beta_j$ by handlesliding it over the vertical curve just to the
right of $Z_i$. For those $j$ that are not $j_i$ for any $Z_i \in \zed$, we 
let $\beta_j^{\zed}$ be a curve isotopic to $\beta_j$ and intersecting it in two points.

\begin {remark}
Our assumption on the existence of a free marking ensures that $\betas^\zed$ contains at least one vertical beta curve.  
\end {remark}

For any consistent collection $\zed$, we denote 
$$\mathbb{T}_{\beta}^{\zed} = \beta_1^\zed \times \dots \times \beta_n^\zed \subset \Sym^n(\TT).$$

Observe that
\begin{equation}
\label{eq:HGZ}
 \Hyper_G(\zed):=(\TT, \alphas, \betas^\zed, \Os^{\orL(\zed)}, \Xs^{\orL(\zed)})
 \end{equation}
is a multi-pointed Heegaard diagram representing the link $\orL - L(\zed)$. More generally, let $\orM$ be any sublink of $L$ containing $\orL(\zed)$, such that the restriction to $L(\zed)$ of the orientation on $\orM$ coincides with $\orL(\zed)$. We can then consider a diagram\footnote{The notation $\Hyper^{\orM}_G(\zed)$ is specific to this paper. In \cite[Section 15.3]{LinkSurg}, this diagram was part of a ``hyperbox of Heegaard diagrams,'' denoted $\Hyper^{\orL,\orM}_G$.}
$$\Hyper^{\orM}_G(\zed) := (\TT, \alphas, \betas^\zed, \Os^{\orM}, \Xs^{\orM}),$$
representing the link $\orL-M$. (In particular, note that $\Hyper^{\orM}_G(\emptyset))$ is the reduction $r_{\orM}(G)$.) Recall from Equation~\eqref{eq:newXO} that $\Os^{\orM}$ contains some previous $X$ markings that are now re-labelled as $O$'s.

We can define a completed Floer complex $\CFm(\Hyper^{\orM}_G(\zed))$, just as we defined $C^-(G)$ in Section~\ref{subsec:ags}. More precisely, $\CFm(\Hyper^{\orM}_G(\zed))$ is generated
over $\Ring$ by $\Ta\cap\Tb^{\zed}$, with differential given by
\begin{multline}
 \partial \x = \sum_{\y \in \Ta\cap\Tb^{\zed}} \sum_{\{\phi\in \pi_2(\x, \y) | \Mas(\phi)=1\}} \#(\M(\phi)/\R)\cdot \prod_{i \not \in I_-(\orL, \orM)} U_{i,1}^{O_{i,1}(\phi)} \dots U_{i, p_i}^{O_{i, p_i}(\phi)} \\
 \cdot \prod_{i  \in I_-(\orL, \orM)} U_{i,1}^{X_{i,1}(\phi)} \dots U_{i, p_i}^{X_{i, p_i}(\phi)} \cdot \prod_{j=n-q+1}^n U_j^{O_j(\phi)}\cdot  \y.
 \end{multline}
Here $\pi_2(\x, \y)$ is the set of homology classes of Whitney disks from $\x$ to $\y$, and $\M(\phi)$ is the moduli space of pseudo-holomorphic disks in the class $\phi$, of Maslov index $\mu(\phi)=1$; compare \cite{HolDisk}. The functions $O_{i,j}, X_{i,j}, O_j$ are as in Section~\ref{subsec:ags}, except we apply them to the domain of $\phi$ (a two-chain on the torus), rather than to a rectangle.

For each $\s \in \bH(L)$, we can also define a generalized Floer chain complex 
$$\Chain^-(\Hyper^{\orM}_G(\zed), \psi^{\orM}(\s)).$$ 
This is constructed just as $\Chain^-(G, \s)$ from Section~\ref{subsec:ags}, but using pseudo-holomorphic disks in $\Sym^n(\TT)$ relative to $\Ta$ and $\mathbb{T}_{\beta}^{\zed}$, instead of rectangles. Further, the basepoints in $\Os^{\orM}$ and $\Xs^{\orM}$ play the roles of those in $\Os$ resp. $\Xs$. 

Explicitly, we first consider the dga generated (in the sense of infinite sums with boundedness condition \eqref{eq:bddness}) over the base ring $\Ring$ by variables 
$$V_{i,j}, \ Y_{i,j}, \text{ for } i \not\in I(\orL, \orM), \ 2 \leq j \leq p_i$$
with relations
$$ Y_{i,j}^2=0, \ \ \del Y_{i,j} = U_{i,j} + U_{i,1}V_{i,j}.$$

We define a complex $\Ccint(\Hyper^{\orM}_G(\zed))$, as the dg module over the dga above, generated by $\x \in \Ta\cap\Tb^{\zed}$, with differential 
\begin{multline}
\label{eq:ley}
\del \x=
\sum_{\y \in \Ta\cap\Tb^{\zed}} \sum_{\{\phi\in \pi_2(\x, \y) | \Mas(\phi)=1\}} \#(\M(\phi)/\R)\cdot \prod_{i \not \in I_(\orL, \orM)} U_{i,1}^{X_{i,1}(\phi) + \ldots + X_{i,p_i}(\phi)} V_{i,2}^{O_{i,2}(\phi)} \dots  V_{i,p_i}^{O_{i,p_i}(\phi)}\\ \cdot
 \prod_{i  \in I_+(\orL, \orM)}
U_{i,1}^{O_{i,1}(\phi)} \dots U_{i, p_i}^{O_{i, p_i}(\phi)} 
 \cdot \prod_{i  \in I_-(\orL, \orM)} U_{i,1}^{X_{i,1}(\phi)} \dots U_{i, p_i}^{X_{i, p_i}(\phi)} \cdot \prod_{j=n-q+1}^n U_j^{O_j(\phi)}\cdot  \y.
\end{multline}

This complex admits filtrations $\FF_i$, for $i \not \in I(\orL, \orM)$, such that
$$ \FF_i(\x) = -A_i(\x), \ \FF_i(U_{i,1})=-1, \ \FF(V_{i,j})=1, \ j \geq 2,$$
and $\FF_i$ takes all the other $U$, $V$ and $Y$ variables to zero. We then define $\Chain^-(\Hyper^{\orM}_G(\zed), \psi^{\orM}(\s))$ as the subcomplex of $ \Ccint(\Hyper^{\orM}_G(\zed))$ in filtration levels $\FF_i \leq \psi_i^{\orM}(\s).$

\medskip
When we have two collections of markings $\zed, \zed'$ such that $\zed \cup \zed'$ is consistent, we will require that $\beta_j^\zed$ and $\beta_j^{\zed'}$ intersect in exactly two points. Hence, there is always a unique maximal degree intersection point $$\Theta^\can_{\zed, \zed'} \in \mathbb{T}_{\beta}^{\zed} \cap \mathbb{T}_{\beta}^{\zed'}.$$ See Figure~\ref{fig:betaZ}. (The homological grading on generators is computed by viewing the markings in both $\zed$ and $\zed'$ as $O$-markings.) 

\begin{figure}
\begin{center}
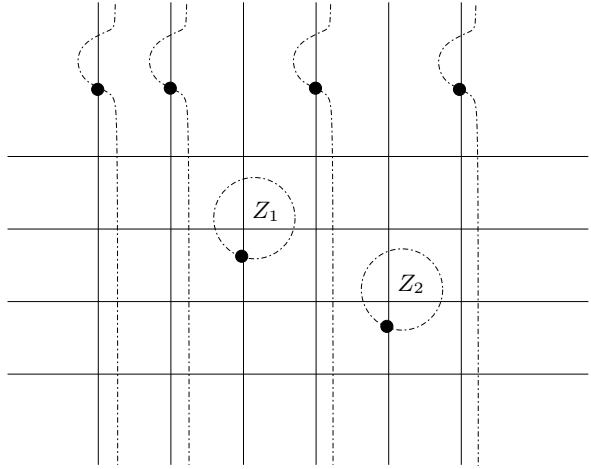
\end{center}
\caption {{\bf A new collection of curves.} We show here a part of a grid diagram, with the horizontal segments lying on  curves in $\alphas$ and the straight vertical segments lying on curves in $\betas$. The interrupted curves (including the two circles) represent curves in $\betas^\zed$, where $\zed$ consists of the two markings $Z_1$ and $Z_2$. The maximal degree intersection point $\Theta^\can_{\emptyset, \zed}$ is represented by the black dots.}
\label{fig:betaZ}
\end{figure}

\subsection {Polygon maps}
\label {sec:PolygonMaps}
Fix a consistent collection of linked markings $\zed= \{Z_1, \dots, Z_m \}$, and a sublink $\orM$ containing $\orL(\zed)$, as in the previous subsection. We define an $m$-dimensional hypercube of chain complexes 
$$H^\zed = \bigl( C^{\zed, \eps}, \De^{\zed, \eps} \bigr)_{\eps \in \E_m}$$ as follows. For $\eps \in \E_m = \{0,1\}^m$, we let 
$$\zed^\eps = \{Z_i \in \zed \mid \eps_i = 1\}$$
and set
$$ C^{\zed, \eps}_{*} = \CFm(\Hyper^{\orM}_G(\zed^\eps)).$$ 

For simplicity, we denote $\Theta^{\can}_{\zed^\eps, \zed^{\eps'}}$ by $\Theta^\zed_{\eps, \eps'}$. We write $\eps \precdot \eps'$ whenever $\eps, \eps' \in \E_m$ are immediate successors, i.e., $\eps < \eps'$ and $\|\eps' - \eps\|=1$. For a string of immediate successors $\eps= \eps^0 \precdot \eps^1 \precdot \cdots \precdot \eps^k = \eps'$, we let
 \begin {equation}
 \label {eq:order}
 \De^{\zed, \eps^0 \precdot \eps^1 \precdot \cdots \precdot \eps^k } : C^{\zed, \eps}_* \to C^{\zed, \eps'}_{* + k-1},
 \end {equation}
\begin {multline*}
  \De^{\zed, \eps^0 \precdot \eps^1 \precdot \cdots \precdot \eps^k } (\x) = 
   \sum_{\y \in \Ta \cap \Tb^{\zed^{\eps^k}}} \sum_{\{\phi \in \pi_2(\x, \Theta^{\zed}_{\eps^0, \eps^1}, \dots,  \Theta^{\zed}_{\eps^{k-1}, \eps^k}, \y)| \mu(\phi) =1-k\}} \bigl(\# \M(\phi)\bigr) \\
   \cdot \prod_{i \not \in I_-(\orL, \orM)} U_{i,1}^{O_{i,1}(\phi)} \dots U_{i, p_i}^{O_{i, p_i}(\phi)} 
 \cdot \prod_{i  \in I_-(\orL, \orM)} U_{i,1}^{X_{i,1}(\phi)} \dots U_{i, p_i}^{X_{i, p_i}(\phi)} \cdot \prod_{j=n-q+1}^n U_j^{O_j(\phi)} \cdot \y
\end {multline*}
be the map defined by counting isolated pseudo-holomorphic polygons in the symmetric product $\Sym^n(\TT)$. Here, 
$$\pi_2(\x, \Theta^{\zed}_{\eps^0, \eps^1}, \dots,  \Theta^{\zed}_{\eps^{k-1}, \eps^k}, \y)$$ denotes the set of  homotopy classes of polygons with edges on $\Ta, \Tb^{\zed^{\eps^0}}, \dots, \Tb^{\zed^{\eps^k}}$, in this cyclic order, and with the specified vertices. The number $\mu(\phi) \in \zz$ is the Maslov index, and $\M(\phi)$ is the moduli space of holomorphic polygons in the class
$\phi$. The Maslov index has to be $1-k$ for the expected dimension of
the moduli space of polygons to be zero. This is because the moduli space of conformal structures on 
a disk with $k+2$ marked points along its boundary has dimension
$(k+2)-3=k-1$.
Note that this definition of $\mu$ is different from the one in
\cite[Section 4.2]{BrDCov}, where it denoted expected dimension.

In the special case $k=0$, we need to divide $\M(\phi)$ by the action of $\rr$ by translations; the resulting $\De^{\zed, \eps^0}$ is the usual differential $\del$. 

We refer to \cite[Section 8]{HolDisk}, \cite[Section 4.2]{BrDCov} and \cite[Section 3.5]{LinkSurg} for further details about polygon counts in Heegaard Floer theory. See also \cite[Lemma 15.2]{LinkSurg} for the proof that the Heegaard multi-diagrams that we consider in our case are admissible (so that the polygon counts are finite).

Define
\begin {equation}
\label {eq:dezede}
 \De^{\zed, \eps} : C^{\zed, \eps^0}_* \to C^{\zed, \eps^0+\eps}_{* + k-1}, \ \ \  \De^{\zed, \eps} = \sum_{\eps^0 \precdot \eps^1 \precdot \cdots \precdot \eps^k = \eps^0 + \eps}
\De^{\zed, \eps^0 \precdot \eps^1 \precdot \cdots \precdot \eps^k } .
\end {equation}

The following is a particular case of \cite[Lemma 8.12]{LinkSurg}; compare also \cite[Lemma 9.7]{HolDisk} and \cite[Lemma 4.3]{BrDCov}:

\begin {lemma}
\label {lemma:d2}
For any consistent collection $\zed= \{Z_1, \dots, Z_m \}$, the resulting $H^\zed= (C^{\zed, \eps}, \De^{\zed, \eps})_{\eps \in \E_m}$ is a hypercube of chain complexes.
\end {lemma}

\begin {remark}
\label {rem:note}
For future reference, it is helpful to introduce a different notation for the maps  \eqref{eq:order} and \eqref{eq:dezede}. First, let us look at \eqref{eq:order} in the case $k=m, \ \eps^0 = (0, \dots, 0)$ and $\eps^m = (1, \dots, 1)$. A string of immediate successors $\eps^0 \prec \cdots \prec \eps^m$ is the same as a  re-ordering $(Z_{\sigma(1)}, \dots, Z_{\sigma(m)})$ of $\zed=\{Z_1, \dots, Z_m\}$, according to the permutation $\sigma$ in the symmetric group $S_m$ such that 
$$ \zed^{\eps^i} = \zed^{\eps^{i-1}} \cup \{Z_{\sigma(i)}\}.$$
We then write:
$$ \De^{(Z_{\sigma(1)}, \dots, Z_{\sigma(m)})} = \De^{\zed, \eps^0 \precdot \eps^1 \precdot \cdots \precdot \eps^k }.$$
 In particular, we let $\De^{(\zed)} = \De^{(Z_1,\dots, Z_m)}$ be the
 map corresponding to the identity permutation. Further, we write $\De^\zed$ for the longest map $\De^{\zed, (1, \dots, 1)}$ in
 the hypercube $H^\zed$, that is,
\begin {equation}
\label {eq:dezeds}
 \De^\zed = \sum_{\sigma \in S_m} \De^{(Z_{\sigma(1)}, \dots, Z_{\sigma(m)})}.
 \end {equation}
Observe that $\De^\zed$, unlike $\De^{(\zed)}$, is independent of the ordering of $\zed$. Observe also that an arbitrary map $\De^{\zed, \eps}$ from $H^\zed$ is the same as the longest map $\De^{\zed^\eps}$ in a sub-hypercube of $H^\zed$. In this notation, the result of Lemma~\ref{lemma:d2} can be written as:
$$ \sum_{\zed' \subseteq \zed} \De^{\zed \setminus \zed'} \circ \De^{\zed'} = 0.$$

See Figure~\ref{fig:destabs} for a picture of the hypercube corresponding to handleslides at a set $\zed = \{Z_1, Z_2\}$ of two linked markings.

\begin{figure}
\begin{center}
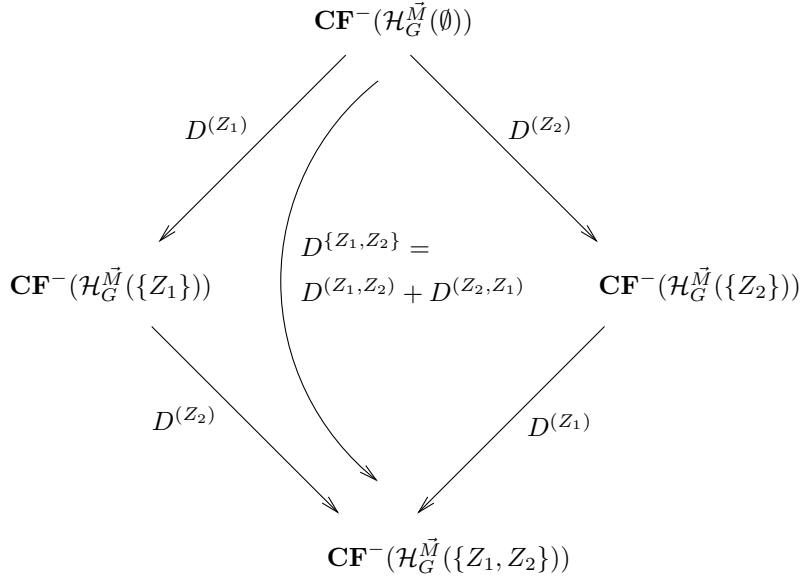
\end{center}
\caption {{\bf Polygon maps for two linked markings.} The straight lines represent chain maps corresponding to handleslides over one marking, and the curved map is a chain homotopy between the two compositions. Each chain map $\De^{(Z_i)}$ could also have been written as $\De^{\{Z_i\}}$.}
\label{fig:destabs}
\end{figure}

\end {remark}

The constructions of polygon maps generalize to other types of complexes. For example, instead of $\CFm(\Hyper^{\orM}_G(\zed^\eps))$ we could consider complexes of the form $\Chain^-(\Hyper^{\orM}_G(\zed), \psi^{\orM}(\s))$. Then, the polygon maps are defined by formulas similar to the ones above, except that in \eqref{eq:order}, the coefficients coming from indices $i \not \in I(\orL, \orM)$ should be
$$U_{i,1}^{X_{i,1}(\phi) + \ldots + X_{i,p_i}(\phi)} V_{i,2}^{O_{i,2}(\phi)} \dots  V_{i,p_i}^{O_{i,p_i}(\phi)}.$$
Compare the formula \eqref{eq:ley} for the differential on $\Chain^-(\Hyper^{\orM}_G(\zed), \psi^{\orM}(\s))$. 

\subsection {Descent maps, Part 1: Overview}
\label{sec:descent1}
 Let $M \subseteq \orL$ be a sublink, endowed with an arbitrary orientation $\orM$, as in Section~\ref{sec:subred}. Set
$$\zed(\orM) = \bigcup_{i \in I_+(\orL, \orM)} \Os_i \cup \bigcup_{i \in I_-(\orL, \orM)} \Xs_i. $$ 
Note that $\orL(\zed(\orM)) = \orM$. We will write $\Hyper_G^{\orM}$ for the diagram $\Hyper_G^{\orM}(\zed(\orM))$, obtained by handlesliding over all the markings in $\Os^{\orM}$.

In this subsection we will use holomorphic polygon counts to construct maps  \begin {equation}
\label {eq:descent}
\hat \De^{\orM}_{p^{\orM}(\s)} : \Chain^-(G, \orM, p^{\orM}(\s)) \to \Chain^-(\Hyper_G^{\orM}, \psi^{\orM}(\s)),
\end {equation}
for $\s \in \bH(L)$. These are called {\em descent maps}, and are the second ingredient in the definition of the maps $\Phi^{\orM}_{\s}$, as advertised in Section~\ref{sec:summary}.

Let $m$ be the number of components of $M$. In the construction of the descent maps we will use some $m$-dimensional hyperboxes of chain complexes. In Section~\ref{sec:hyper}, the sides of  such a hyperbox were labeled by $i=1, \dots, m$, so that each complex had an index $\eps= (\eps_1, \dots, \eps_m)$. In this section, for notational convenience, we will label the sides of hyperboxes by the indices $i \in I(\orL, \orM)$ instead, so that the index of a complex is $\eps = (\eps_i)_{i \in I(\orL, \orM)}$. Of course, one can go back and forth between an ordinary hyperbox from one labeled by $I(\orL, \orM)$, by using the order-preserving bijection from $I(\orL, \orM)$ to $\{1, \dots, m\}$. Hence, we can easily transfer notions from Section~\ref{sec:hyper} (such as compression)  into similar notions for the hyperboxes labeled by $I(\orL, \orM)$.

For $i\in I(\orL, \orM)$, equip $L_i$ with the orientation $\orL_i$ induced from $\orM$. Then $\zed(\orL_i)$ is either $\Os_{i}$ or $\Xs_{i}$. In either case, we have an ordering of its elements, so we can write 
$$ \zed(\orM_j) = \{Z_{i,1}, \dots, Z_{i, p_i} \}.$$
(Recall from Section~\ref{subsec:ags} that, in general, $p_i$ denotes the number of basepoints of one type on the component $L_i$.)

For each multi-index $\eps = (\eps_i)_{i \in I(\orL, \orM)}$ with $0 \leq \eps_i \leq p_i$, we let $\zed(\orM)^\eps \subseteq \zed(\orM)$ be the collection of markings 
$$ \zed(\orM)^{\eps} = \bigcup_{i \in I(\orL, \orM)} \{Z_{i,1}, \dots, Z_{i, \eps_i} \}.  $$

We then let 
$$\betas^{\eps} = \betas^{\zed(\orM)^\eps}$$
be the collection of beta curves handleslid at the points of $\zed(\orM)^\eps$. We will also write $\T_{\beta}^{\eps} \subset \Sym^n(\T)$ for the product of the curves in the collection $\betas^{\eps}$.

For each $\eps$, consider the Heegaard diagram $$\Hyper^{\orM}_{G, \eps} := \Hyper_G^{\orM}(\zed(\orM)^{\eps}) = (\TT, \alphas, \betas^\eps, \Os^{\orM}, \Xs^{\orM}).$$ This diagram represents the link $\orL - M$. In particular, for $\eps=\zero=(0,\dots, 0)$ the diagram $\Hyper^{\orM}_{G, \zero}$ is the reduction $r_{\orM}(G)$. At the other extreme, for $\eps =(p_i)_{i \in I(\orL, \orM)}$, we obtain the diagram $\Hyper_G^{\orM}$.

For $\eps < \eps'$ with $\|\eps' - \eps\| =1$, we also denote
$$ \Theta_{\eps, \eps'} =\Theta^\can_{\zed(\orM)^{\eps}, \zed(\orM)^{\eps'}}.$$

Using this information, we can construct a preliminary hyperbox $H_{\prelim}$ of generalized Floer complexes, of size $\dd= (p_i)_{i \in I(\orL, \orM)}$, such that at position $\eps$ we have the complex
$$ \Chain^-(\Hyper^{\orM}_{G, \eps}, \psi^{\orM}(\s))$$
and along the faces we have maps given by counting polygons accordingly (as noted at the end of Section~\ref{sec:PolygonMaps}). After compressing this hyperbox, the longest map in the resulting hypercube would be of the form 
$$ \Chain^-(r_{\orM}(G),  \psi^{\orM}(\s)) \to \Chain^-(\Hyper^{\orM}_{G}, \psi^{\orM}(\s)).$$
This has the same target as the desired descent map \eqref{eq:descent}, but the domain is different. 
In fact, the two domains $\Chain^-(G, \orM, p^{\orM}(\s))$ and $\Chain^-(r_{\orM}(G),  \psi^{\orM}(\s))$ are constructed in rather similar ways; they both look like $\Chain^-$ with respect to $L-M$ and like $\CFm$ with respect to $M_+$ (that is, their differential keeps track of the basepoints on $L-M$ and $M_+$ in the same manner). The only difference is with regard to the components $L_i \subseteq M_-$. Indeed, in $\Chain^-(G, \orM, p^{\orM}(\s))$ terms in the differential contain factors of the form
$$U_{i,1}^{X_{i,1}(r) + \ldots + X_{i,p_i}(r)} V_{i,2}^{O_{i,2}(r)} \dots  V_{i,p_i}^{O_{i,p_i}(r)}, $$
whereas for $\Chain^-(r_{\orM}(G),  \psi^{\orM}(\s))$ we have factors
$$ U_{i,1}^{X_{i,1}(r)} U_{i,2}^{X_{i,2}(r)} \dots U_{i,p_i}^{X_{i,p_i}(r)}.$$

To obtain the desired descent map $\hat \De^{\orM}_{p^{\orM}(\s)}$ starting from $\Chain^-(G, \orM, p^{\orM}(\s))$, we need to relate the two ways of keeping track of the basepoints on $M_-$. Specifically, $\hat \De^{\orM}_{p^{\orM}(\s)}$ will be the longest map in the compression of a hyperbox $H$, of the same dimension $m$ as $H_{\prelim}$ but of larger size. The hyperbox $H$ is obtained from $H_{\prelim}$ by adding certain complexes and ``transition'' maps at the beginning of the sides corresponding to $M_-$. 

In Sections~\ref{sec:descent2}-\ref{sec:descent3} below we give the detailed definition of the hyperbox $H$, by adapting the construction in \cite[Section 13.6]{LinkSurg} to the case of grids. The hyperbox $H$ will be of size $\dd':=(d'_i)_{i \in I(\orL, \orM)}$, where
\begin{equation}
\label{eq:di}
 d'_i = \begin{cases}
p_i &\text{if } i\in I_+(\orL, \orM),\\
2p_i - 1 &\text{if } i\in I_-(\orL, \orM)
\end{cases}.
\end{equation}

\begin{remark}
In the formula for the size of the hyperbox in \cite[Equation (234)]{LinkSurg}, we had $d_i$ for $i\in I_+(\orL, \orM)$ and $d_i + p_i-1$ for $i\in I_-(\orL, \orM)$, where $d_i$ are the side lengths of the preliminary hyperbox $H_{\prelim}$. For grids, we have $d_i = p_i$, because we use exactly $p_i$ moves (handleslides) to go from a grid diagram $G$ to $\Hyper^{\orL_i}_G$.  
\end{remark}

We warn the reader that the algebra involved in the construction of the hyperbox $H$ is rather complicated. We refer to \cite[Section 13]{LinkSurg} for a friendlier treatment, with the motivation built up from simpler particular cases. For example, the case of a knot with four basepoints is discussed in \cite[Section 13.1]{LinkSurg}, and already contains many of the features of the general setting. Also, a good picture of $H$ to keep in mind is \cite[Figure 38]{LinkSurg}, where the preliminary hyperbox $H_{\prelim}$ is the rectangle bounded by dashed lines.

\subsection{Descent maps, Part 2: Complexes}
\label{sec:descent2}
In this section we write down the complexes assigned to vertices of the hyperbox $H$. In the next section, we will describe the differentials. 

At a vertex $\eps=(\eps_i)_{i\in I(\orL, \orM)}$ in $H$, we place a complex $C^{\eps}$ defined as follows. Let
$$ I(\eps) = \{ i \in I_-(\orL, \orM) \mid \eps_i < p_i - 1 \}$$
and
$$ M_{\eps}= \bigcup_{i \in I(\eps)} L_i \subseteq M_-.$$

We split the information in $\eps$ into two parts, $\eps^<$ and $\eps^>$, corresponding to $i \in I(\eps)$ and $i \not \in I(\eps)$, respectively. Specifically, $\eps^<$ is a vector whose entries are indexed by $i \in I(\eps)$, such that
$$ \eps^<_i = \eps_i.$$

The second vector, $\eps^>$, has entries indexed by $i \in I(\orL, \orM) - I(\eps)$, such that
$$ \eps^>_i = \begin{cases}\eps_i & \text{if } i \in I_+(\orL, \orM), \\
\eps_i - (p_i -1) & \text{if } i \in I_-(\orL, \orM) - I(\eps).
\end{cases}$$
In the second case, we subtract $p_i-1$ so that we are in agreement with the indexing of the entries in the hyperbox $H_{\prelim}$. Indeed, note that the values of $\eps^>_i$ vary between $0$ and $d_i$. In fact, the sub-hyperbox of $H$ corresponding to indices $\eps$ such that $I(\eps)=\emptyset$  will be exactly $H_{\prelim}$.

For arbitrary $\eps$, equip the sublink $\orM-M_{\eps}$ with the orientation induced from $\orM$. The complex at position $\eps$ in the hyperbox $H$ is associated to the Heegaard diagram 
$$\Hyper^{\orM-M_{\eps}}_{G, \eps^>}= (\TT, \alphas, \betas^{\eps^>}, \Os^{\orM-M_{\eps}}, \Xs^{\orM-M_{\eps}}).$$ We denote this complex by
\begin{equation}
\label{eq:CCmix}
 C^{\eps} = \Cc_{\eps^<}(\Hyper^{\orM-M_{\eps}}_{G, \eps^>}, M_{\eps}, \psi^{\orM-M_{\eps}}(\s)).
 \end{equation}
This is a generalized Floer complex that looks like
\begin{itemize}
\item the resolution $\Chain^-$ in the directions $i \not \in I(\orL, \orM)$,
\item a Floer complex $\CFm$ in the directions $i \in I(\orL, \orM) - I(\eps)$, using the markings in $\Os^{\orM}$, 
\item a complex interpolating between $\Ccint$ and $\CFm$, in a way depending on $\eps_i$, in the directions $i \in I(\eps)$.
\end{itemize}
We will construct $C^{\eps}$ as a subcomplex of a complex denoted
$$ \Cc_{\eps^<}(\Hyper^{\orM-M_{\eps}}_{G, \eps^>}, M_{\eps}),$$
which is similar to $C^{\eps}$, except it looks like $\Ccint$ in the directions $i\not \in I(\orL, \orM)$.

Let us make these definitions fully rigorous. First, recall that we have a base ring
$$ \Ring = \Field[[(U_{i,j})_{\substack{1\leq i \leq \ell, \\ 1 \leq j \leq p_i}} (U_{j})_{n-q+1\leq j \leq n}]].$$
We introduce the dga
$$ \Ringbig^Y_{\eps^<}(\orL, \orM, M_{\eps})$$
generated over $\Ring$ (in the sense of infinite sums with boundedness condition) by new variables $V_{i,j}, Y_{i,j}$, for indices $i, j$ such that
\begin{equation}
\label{eq:condo}
\bigl(i \not\in I(\orL, \orM), \ 2 \leq j \leq p_i \bigr ) \ \text{ or } \ \bigl(i \in I(\eps), \eps_i +2 \leq j \leq p_i \bigr)
\end{equation}
with relations
$$ Y_{i,j}^2=0, \ \ \del Y_{i,j} = U_{i,j} + U_{i,1}V_{i,j}.$$
We let $ \Cc_{\eps^<}(\Hyper^{\orM-M_{\eps}}_{G,\eps^>}, M_{\eps})$ be the complex  generated over $\Ringbig^Y_{\eps^<}(\orL, \orM, M_{\eps})$ by $\x \in \Ta \cap \T_{\beta}^{\eps^>}$, with differential
\begin{multline}
\label{eq:multidel}
 \del \x = \sum_{\y \in \Ta \cap \T_{\beta}^{\eps^>}}  \sum_{\substack{\phi \in \pi_2(\x, \y) \\ \mu(\phi)=1} }  \# (\M(\phi)/\R) \cdot \prod_{i \not \in I(\orL, \orM)}  U_{i,1}^{X_{i,1}(\phi) + \ldots + X_{i,p_i}(\phi)} V_{i,2}^{O_{i,2}(\phi)} \dots  V_{i,p_i}^{O_{i,p_i}(\phi)} \\
\cdot \prod_{i \in I(\eps)}  U_{i,1}^{X_{i,1}(\phi)} \cdots U_{i,\eps_i}^{X_{i,\eps_i}(\phi)}
U_{i,\eps_i+1}^{X_{i,\eps_i+1}(\phi) + \ldots + X_{i,p_i}(\phi)} V_{i,\eps_i+2}^{O_{i,\eps_i+2}(\phi)} \dots  V_{i,p_i}^{O_{i,p_i}(\phi)}\\
\cdot  \prod_{i \in I_-(\orL, \orM) - I(\eps)} U_{i,1}^{X_{i,1}(\phi)} \cdots U_{i,p_i}^{X_{i,p_i}(\phi)} 
   \prod_{i \in I_+(\orL, \orM)} U_{i,1}^{O_{i,1}(\phi)} \cdots U_{i,p_i}^{O_{i,p_i}(\phi)} 
  \prod_{j=n-q+1}^n U_{j}^{O_{j}(\phi)} \y.
 \end{multline}

This complex admits filtrations $\FF_i$, for $i \not\in I(\orL, \orM)$, such that
$$ \FF_i(\x) = -A_i(\x), \ \FF_i(U_{i,1})=-1, \ \FF_i(V_{i,j})=1, \ j \geq 2,$$
and $\FF_i$ takes all the other $U$, $V$ and $Y$ variables to zero. 

We then define $C^{\eps} = \Cc_{\eps^<}(\Hyper^{\orM-M_{\eps}}_{G,\eps^>}, M_{\eps}, \psi^{\orM-M_{\eps}}(\s))$ to be the subcomplex of $\Cc_{\eps^<}(\Hyper^{\orL, \orM-M_{\eps}}_{G,\eps^>}, M_{\eps})$ given by
$$ \FF_i \leq \psi^{\orM-M_{\eps}}_i(\s), \ i \not\in I(\orL, \orM).$$

\begin{example} When $\eps = \zero$, observe that $M_{\zero} = M_-$, $I(\zero)=I_-(\orL, \orM)$ and  $\Hyper^{\orM-M_{\zero}}_{G,\zero}=r_{M_+}(G).$ Therefore, we have identifications
$$ \Cc_{\zero}(\Hyper^{\orM-M_{\zero}}_{G,\zero}, M_{\zero}) \cong \Ccint(r_{M_+}(G))$$
and
\begin{equation}
\label{eq:Cinitial}
C^{\zero}=\Cc_{\zero}(\Hyper^{\orM-M_{\zero}}_{G, \zero}, M_{\zero}, \psi^{\orM-M_{\zero}}(\s)) \cong \Chain^-(G, \orM, p^{\orM}(\s)).
\end{equation}
(Compare the construction of $\Chain^-(G, \orM, p^{\orM}(\s))$ in Section~\ref{sec:inclusions}.)
\end{example}

\begin{example}
At the other extreme, for $\eps$ such that $I(\eps)=\emptyset$, we have $M_{\eps} = \emptyset$ and
$$
 C^\eps \cong \Chain^-(\Hyper^{\orM}_{G, \eps^>}, \psi^{\orM}(\s)),
 $$
 which is part of the preliminary hyperbox $H_{\prelim}$. Thus, the final complex in the hyperbox $H$, with $\eps= \dd'$, is
\begin{equation}
\label{eq:Cfinal}
C^{\dd'} \cong\Chain^-(\Hyper^{\orM}_G, \psi^{\orM}(\s)).
 \end{equation}
\end{example}

\subsection{Descent maps, Part 3: Differentials}
\label{sec:descent3}
For general $\eps, \eps'$ such that $\eps \leq \eps'$ and $\eps, \eps'$ are neighbors, we now proceed to define the maps
$$ D^{\eps' - \eps}_{\eps}: C^{\eps} \to C^{\eps'}$$
that are part of the hyperbox $H$. 

Since $\eps$ and $\eps'$ are neighbors, for any $i$ we must have $\eps'_i = \eps_i$ or $\eps'_i=\eps_i+1$. Observe that $I(\eps') \subseteq I(\eps)$. When an index $i$ belongs to $I(\eps) - I(\eps')$, it means that $\eps_i = p_i-2$ and $\eps'_i = p_i -1$, so $(\eps')^>$ has $0$ in position $i$. Thus, while $(\eps')^>$ may have more entries that $\eps^>$, all these additional entries are zero; let us denote by $(\eps^>)'$ the vector obtained from $(\eps')^>$ by deleting these entries. In the same manner, $(\eps^<)'$ may have fewer entries than $\eps^<$, in which case the additional entries of $\eps^<$ are $p_i-2$; we let $(\eps^<)'$ be the vector obtained from $(\eps^<)'$ by introducing new entries of $p_i-1$ in those positions. Thus, $(\eps^>)'$ has the same length as $\eps^>$, and $(\eps^<)'$ has the same length as $\eps^<$.

Note that $\Hyper^{\orM-M_{\eps'}}_{G, (\eps')^>}$, the diagram used to define $C^{\eps'}$, can be viewed as a reduction:
$$ \Hyper^{\orM-M_{\eps'}}_{G,(\eps')^>} = r_{M_{\eps} - M_{\eps'}} (\Hyper^{ \orM-M_{\eps}}_{G,(\eps^>)'}).$$

Thus, the complex $C^{\eps'}$ can be interpreted as being associated to the diagram $\Hyper^{\orM-M_{\eps}}_{G, (\eps^>)'}.$ Indeed, from the definitions we see that
\begin{align*}
 C^{\eps'} &=  \Cc_{(\eps')^<}(\Hyper^{\orM-M_{\eps'}}_{G,(\eps')^>}, M_{\eps'}, \psi^{\orM-M_{\eps'}}(\s)) \\
 &= \Cc_{(\eps')^<}(r_{M_{\eps} - M_{\eps'}} (\Hyper^{\orM-M_{\eps}}_{G,(\eps^>)'}), M_{\eps'}, \psi^{\orM-M_{\eps}}(\s)) \\
 &=  \Cc_{(\eps^<)'}(\Hyper^{\orM-M_{\eps}}_{G,(\eps^>)'}, M_{\eps}, \psi^{\orM-M_{\eps}}(\s)).
 \end{align*}

Similar remarks apply to the intermediate indices between $\eps$ and $\eps'$. Specifically, we can write the corresponding complexes as
$$ \Cc_{\nu}(\Hyper^{\orM-M_{\eps}}_{G,\eta}, M_{\eps}, \psi^{\orM-M_{\eps}}(\s))$$
for all $\nu, \eta$ with
$$ \eps^< \leq \nu \leq (\eps^<)', \ \ \ \eps^> \leq \eta \leq (\eps^>)'.$$
These complexes have almost the same definition as the one for index $\eps$, except that changing $\eps^>$ into $\eta$ results in a change in the underlying Heegaard diagram, whereas changing $\eps^<$ into $\nu$ results in a change in the exponents of $U_{i,j}$ for some $i\in I(\eps)$ in the formula \eqref{eq:multidel}; namely, $\eps_i$ gets replaced by $\nu_i$, which may equal $\eps_i + 1$.

In light of this discussion, to define $ D^{\eps' - \eps}_{\eps}$ we will focus on diagrams of the form
$$\Hyper^{\orM-M_{\eps}}_{G, \eta},$$
for indices $\eta$ such that
$$ \eps^> \leq \eta \leq (\eps^>)'.$$
These diagrams are naturally arranged in a $v$-dimensional hypercube, with 
$$v=\| (\eps^>)' - \eps^{>} \|.$$ 

The maps $ D^{\eps' - \eps}_{\eps}$ will be given by counts of holomorphic $k$-gons for $k \leq v+2$, coming from the above hypercube.

For any intermediate $\zeta$ between (and including) $\eps^>$ and $(\eps^>)'$, as part of the Heegaard diagram $\Hyper^{\orM-M_{\eps}}_{G, \eta}$ we have a collection of curves $\betas^\zeta$, obtained from $\betas$ by handleslides.  We let $\betas'$ denote $\betas^\zeta$ for $\zeta = (\eps^>)'.$ 

For any two intermediate indices $\zeta, \zeta'$ with
$$  (\eps^{>}) \leq \zeta \leq \zeta' \leq (\eps^>)'$$
we have an intersection point
$$ \Theta_{\zeta, \zeta'} =\Theta^{\can}_{\zed(\orM)^{\zeta}, \zed(\orM)^{\zeta'}} \in \T_{\beta^{\zeta}} \cap \T_{\beta^{\zeta'}}.$$

The map $D_{\eps}^{\eps'-\eps}: C^{\eps} \to C^{\eps'}$ is defined as follows. For an intersection point $ \x \in \Ta \cap  \Tb$, we set
\begin{equation}
\label{eq:bigformula}
D_{\eps}^{\eps'-\eps}(\x) = \sum_{E} \#\M(\phi) \cdot  \U^{u(E)} \y,
\end{equation}
where the summation is over data $E$ consisting of:
\begin{itemize}
\item$q \geq 0, $\\
\item $ \eps^> = \zeta^0< \dots < \zeta^q = (\eps^>)',$\\
\item $ \y \in  \Ta \cap  \Tbp$,\\
\item $ \phi \in \pi_2(\x, \Theta_{\zeta^0\zeta^1}, \dots, \Theta_{\zeta^{q-1}\zeta^q}, \y) \text{ with } \mu(\phi) = 1-q.$
\end{itemize}
In the formula \eqref{eq:bigformula}, by $\# \M(\phi)$ we mean the count of holomorphic polygons in the class $\phi$ (where we divide by the action of $\rr$ if we count bigons). Moreover, the factor $\U^{u(E)}$ is
\begin{multline}
\label{eq:multiU}
\U^{u(E)} :=  \prod_{i \not \in I(\orL, \orM)}  U_{i,1}^{X_{i,1}(\phi) + \ldots + X_{i,p_i}(\phi) } V_{i,2}^{O_{i,2}(\phi)} \dots  V_{i,p_i}^{O_{i,p_i}(\phi)} \\
\cdot \prod_{\{i \in I(\eps)\mid \nu_i = \eps_i \}}  U_{i,1}^{X_{i,1}(\phi)} \cdots U_{i,\eps_i}^{X_{i,\eps_i}(\phi)}
U_{i,\eps_i+1}^{X_{i,\eps_i+1}(\phi) + \ldots + X_{i,p_i}(\phi)} V_{i,\eps_i+2}^{O_{i,\eps_i+2}(\phi)} \dots  V_{i,p_i}^{O_{i,p_i}(\phi)}\\
\cdot \prod_{\{i \in I(\eps)\mid \nu_i = \eps_i+1 \}}  U_{i,1}^{X_{i,1}(\phi)} \cdots U_{i,\eps_i+1}^{X_{i,\eps_i+1}(\phi)}
\frac{U_{i,\eps_i+1}^{X_{i,\eps_i+2}(\phi) + \ldots + X_{i,p_i}(\phi)} - U_{i,\eps_i+2}^{X_{i,\eps_i+2}(\phi) + \ldots + X_{i,p_i}(\phi)} }{U_{i, \eps_i + 1} - U_{i, \eps_i + 2}} V_{i,\eps_i+3}^{O_{i,\eps_i+3}(\phi)} \dots  V_{i,p_i}^{O_{i,p_i}(\phi)}
\\
\cdot  \prod_{i \in I_-(\orL, \orM) - I(\eps)} U_{i,1}^{X_{i,1}(\phi)} \cdots U_{i,p_i}^{X_{i,p_i}(\phi)} 
\cdot   \prod_{i \in I_+(\orL, \orM)} U_{i,1}^{O_{i,1}(\phi)} \cdots U_{i,p_i}^{O_{i,p_i}(\phi)} 
 \cdot  \prod_{j=n-q+1}^q U_{j}^{O_j(\phi)}.
\end{multline}

\begin{remark}
Most of the factors in the expression \eqref{eq:multiU} are similar to those in the formula \eqref{eq:multidel} for the differential on the complex $C^{\eps}$. The new ingredient, which is the key factor responsible for the transition from $C^\eps$ to $C^{\eps'}$, is
$$ \frac{U_{i,\eps_i+1}^{X_{i,\eps_i+2}(\phi) + \ldots + X_{i,p_i}(\phi)} - U_{i,\eps_i+2}^{X_{i,\eps_i+2}(\phi) + \ldots + X_{i,p_i}(\phi)} }{U_{i, \eps_i + 1} - U_{i, \eps_i + 2}}.$$
The simplest setting where a factor of this type appears is that of a knot with four basepoints, discussed in \cite[Section 13.1]{LinkSurg}; see Equation (185) in the definition of the transition map there. Similar formulas appeared previously in the work of Sarkar \cite{SarkarMoving}. 
\end{remark} 

So far we have described the values of $D^{\eps'-\eps}_{\eps}$ at intersection points $\x$. To define the map in full, we extend it to be equivariant with respect to the action of all the $U_{i,j}, V_{i,j}, Y_{i,j}$ and $U_{j}$ variables that appear in the construction of the target $C^{\eps'}$. In view of the conditions \eqref{eq:condo}, note that the domain $C^{\eps}$ may have a few more variables, namely 
$$V_{i,\eps_i+2}, Y_{i, \eps_i + 2} \text{ for } i \in I(\eps) \text{ such that } \eps'_i = \eps_i + 1.$$

We let $D^{\eps'-\eps}_{\eps}$ be equivariant with respect to $V_{i, \eps_i+2}$, where we let this variable act on the target by $1$. With regard to $Y_{i, \eps_i + 2}$, we set 
$$D^{\eps'-\eps}_{\eps}(Y_{i, \eps_i + 2} \x) = 0\  \text{ when } \|\eps' - \eps\| > 1,$$
whereas when $\|\eps' - \eps\| =1$ there is a unique $i$ with $\eps'_i = \eps_i + 1$, and (for the extra variable $Y$ to exist) we must have $i \in I(\eps)$; that is,
$$ (\eps^<)'_i = \eps^<_i + 1, \ \  (\eps^>)'_i = \eps^>_i.$$

In this particular case and for that particular $i$, the formula for $D^{\eps'-\eps}_{\eps}(Y_{i, \eps_i + 2} \x)$ differs from the one for $D^{\eps'-\eps}_{\eps}( \x)$ by replacing 
$$U_{i,1}^{X_{i,1}(\phi)} \cdots U_{i,\eps_i+1}^{X_{i,\eps_i+1}(\phi)}
\frac{U_{i,\eps_i+1}^{X_{i,\eps_i+2}(\phi) + \ldots + X_{i,p_i}(\phi)} - U_{i,\eps_i+2}^{X_{i,\eps_i+2}(\phi) + \ldots + X_{i,p_i}(\phi)} }{U_{i, \eps_i + 1} - U_{i, \eps_i + 2}} V_{i,\eps_i+3}^{O_{i,\eps_i+3}(\phi)} \dots  V_{i,p_i}^{O_{i,p_i}(\phi)}$$
on the third line of the formula \eqref{eq:multiU}, with $1$.

We have now completed the description of the hyperbox $H$. We compress this hyperbox as in  Section~\ref{sec:hyper}, and define the descent map
\begin {equation}
\label {eq:dezed}
\hat \De^{\orM}_{p^{\orM}(\s)} : \Chain^-(G, \orM, p^{\orM}(\s))  \to \Chain^-(\Hyper^{\orM}_G, \psi^{\orM}(\s))
\end {equation}
to be the longest diagonal map in the resulting hypercube. Note that the domain $C^{\zero}$ and the target $C^{\dd'}$ of this map were identified in Equations~\eqref{eq:Cinitial} and \eqref{eq:Cfinal}.

\subsection {Destabilized complexes}
\label {subsec:ks}
Let $\zed = \{Z_1, \dots, Z_m\}$ be a consistent set of linked markings. Recall that $L(\zed) \subseteq L$ is the sublink consisting of those components $L_i$ such that at least one of the markings on $L_i$ is in $\zed$. Let $\orM$ a sublink containing $\orL(\zed)$, with a compatible orientation, as in Section~\ref{sec:DestabSeveral}. In that section we introduced Floer complexes $\CFm( \Hyper^{\orM}_G(\zed))$ and
$\Chain^-(\Hyper^{\orM}_G(\zed), \psi^{\orM}(\s))$, based on counting holomorphic
curves. Our next goal is to describe these complexes combinatorially.

Consider the grid diagram $G^{\zed}$ obtained from $G$ by 
deleting the markings in $\Xs_i$ when $\zed \cap \Os_i \neq
\emptyset$, deleting the markings in $\Os_i$ when $\zed \cap \Xs_i
\neq \emptyset$, and finally deleting the
rows and columns containing the markings in $\zed$.
The former free $O$ markings in $G$ remain as free
markings in $G^{\zed}$. However, in $G^\zed$ there may be some
additional free markings, coming from linked $O$ or $X$ markings in
$G$ that were not in $\zed$, but were on the same link component as a
marking in $\zed$. To be consistent with our previous conventions, we
relabel all the newly free $X$ markings as $O$'s.

Thus, $G^{\zed}$ becomes a grid diagram (with free markings) for the link $L - L(\zed)$, with the orientation induced from $\orL$. We can form a Floer complex
$$ C^-(G^{\zed}) = \CFm(G^{\zed})$$
as in Section~\ref{subsec:ags}, but defined over a ring with fewer $U$ variables. (We no longer have the $U$ variables corresponding to markings in $\zed$.) Note that holomorphic disks of index one in the symmetric product of $G^\zed$ are in one-to-one correspondence with rectangles on $G^\zed$.

If we delete and relabel more basepoints, according to the sublink $\orM - L(\zed) \subseteq L-L(\zed)$, we get a diagram
$$ \Hyper^{\orM -L(\zed)}_{G^{\zed}} (\emptyset) = r_{\orM - L(\zed)} (G^{\zed}).$$
We define complexes $\CFm(r_{\orM - L(\zed)} (G^{\zed}))$ and $\Chain^-(r_{\orM - L(\zed)} (G^{\zed}), \psi^{\orM}(\s))$ for $\s \in \bH(G)$, as in Section~\ref{subsec:ags}.

For $Z \in \zed$, let $U_Z$ be the variable corresponding to $Z$, and $U'_Z$ the variable corresponding to the row exactly under the row through $Z$ (in the grid $G$). We define a complex
$$ \K(\zed) =  \bigotimes_{Z \in \zed} \Bigl( \Ring \xrightarrow{U_{Z} - U'_{Z} } \Ring \Bigr).$$

Using the argument in \cite[Proposition 15.3]{LinkSurg}, for a suitable choice of almost complex structure on the symmetric product $\Sym^n (\TT)$, we have isomorphisms
\begin {equation}
\label {eq:psiz0}
 \Psi^{\orM, \zed} : \CFm(\Hyper^{\orM}_G(\zed)) \to  \CFm(r_{\orM - L(\zed)} (G^{\zed}))[[\{U_{Z}\}_{Z \in  \zed } ]] \otimes_{\Ring} \K(\zed).
\end {equation}
and
\begin {equation}
\label {eq:psiz}
 \Psi^{\orM, \zed}_{ \psi^{\orM}(\s)} : \Chain^-(\Hyper^{\orM}_G(\zed), \psi^{\orM}(\s)) \to  \Chain^-(r_{\orM - L(\zed)} (G^{\zed}), \psi^{\orM}(\s))[[\{U_{Z}\}_{Z \in  \zed } ]] \otimes_{\Ring} \K(\zed).
\end {equation}
This gives a combinatorial description of the complexes associated to the diagram $\Hyper^{\orM}_G(\zed)$.

To get an understanding of Equations~\eqref{eq:psiz0} and \eqref{eq:psiz}, note that there
is a clear one-to-one correspondence between the generators of each
side; see Figure~\ref{fig:annuli}. The differentials correspond likewise: as explained in \cite[proof of Proposition 15.3]{LinkSurg}, there is a filtration on the complex on the left hand side, such that the associated graded consists of copies of $ \K(\zed)$, one for each generator of $\CFm(r_{\orM - L(\zed)} (G^{\zed}))$. The differential on the associated graded involves bigons that pick up $U_Z$, as well as more complicated domains that pick up $U_{Z}'$, such as for example the annulus shown in Figure~\ref{fig:annuli}. Together, they give rise to the mapping cone of $U_Z - U'_Z$, which is a factor in $\K(\zed)$. There are also contributions to the left hand side differential that decrease the filtration degree; these exactly correspond to empty rectangles in the destabilized diagram $G^{\zed}$.

\begin{figure}
\begin{center}
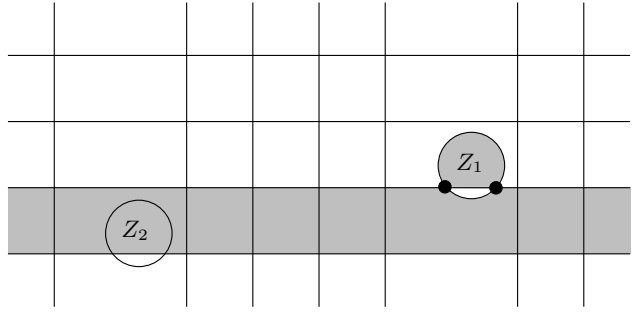
\end{center}
\caption {{\bf The complex $\K(\zed)$.}  The figure shows
  part of a grid diagram with some arcs on the $\alpha$ and
  $\beta^{\zed}$ curves drawn. There are two intersection
  points (marked as bullets) between the alpha curve below the marking
  $Z_1 \in \zed$, and the corresponding beta curve. There are two differentials
  going from the left to the right generator: a bigon containing $Z_1$
  and an annulus containing $Z_2$, both drawn shaded in the
  diagram. This produces one of the factors in the definition of the
  complex $\K(\zed)$: the left bullet corresponds to the domain of the map $U_{Z_1} - U_{Z_2}$, and the right bullet to the target.  }
\label{fig:annuli}
\end{figure}

Note that the isomorphisms in ~\eqref{eq:psiz0} and \eqref{eq:psiz} are given by the above correspondence between generators. Their definition is thus purely combinatorial.

For a generic almost complex structure on the symmetric product, we still have maps of the form $\Psi^{\orM, \zed}$ and $\Psi^{\orM, \zed}_{ \psi^{\orM}(\s)}$, but they are chain homotopy equivalences rather than isomorphisms.

A particular instance of the discussion above appears when $\zed =\zed(\orM)$. In this setting, recall that we have $\orM=\orL(\zed)$, and $\Hyper^{\orM}_G(\zed)$ is denoted $\Hyper^{\orM}_G.$ The destabilized grid diagram $$G^{L - M}:=G^{\zed(\orM)}$$ is obtained from $G$ by
eliminating all rows and columns on which $M$ is supported. It represents the link $L - M$. 

For simplicity, we denote 
$$\K(M) = \K(\zed(\orM)).$$
The isomorphism \eqref{eq:psiz} becomes
\begin {equation}
\label {eq:psim}
\Psi^{\orM}_{ \psi^{\orM}(\s)} = \Psi^{\zed(\orM)}_{ \psi^{\orM}(\s)}: \Chain^-(\Hyper^{\orM}_G, \psi^{\orM}(\s)) \to  \Chain^-( G^{L-M}, \psi^{\orM}(\s))[[\{U_{i,j}\}_{L_i \subseteq M} ]] \otimes_{\Ring} \K(M).
\end {equation}

By composing the maps
\eqref{eq:Proj}, \eqref{eq:dezed} and \eqref{eq:psim}, we obtain a
map
\begin {equation}
  \begin{aligned}
\Phi_\s^{\orM} &: \Chain^- (G, \s) \longrightarrow \Chain^-(G^{L-M}, \psi^{\orM}(\s))[[\{U_{i,j}\}_{L_i \subseteq M }]] \otimes_{\Ring} \K(M),\\ 
\Phi_\s^{\orM} &= \Psi_{\psi^{\orM}(\s)}^{\orM} \circ \hat\De_{p^{\orM}(\s)}^{\orM} \circ \I^{\orM}_\s ,
  \end{aligned}
\label {eq:phims}\end {equation}
defined for any $\s \in  \bH(L)$. This was the map advertised in Section~\ref{sec:summary}.

\subsection {The surgery theorem}
\label {sec:surgery}

Let us fix a framing $\Lambda$ for the link $\orL$. For a component
$L_i$ of $L$, we let $\Lambda_i$ be its induced framing, thought of as
an element in $H_1(Y-L)$. This last group can be identified with
$\zz^\ell$ using the basis of oriented meridians for the
components. Under this identification, for $i \neq j$, the $j\th$
component of the vector $\Lambda_i$ is the linking number between
$L_i$ and $L_j$. The $i\th$ component of $\Lambda_i$ is its
homological framing coefficient of $L_i$ as a knot.

Given a sublink $N \subseteq L$, we let $\Omega(N)$ be the set of all possible orientations on $N$. For $\orN \in \Omega(N)$, we let
$$ \Lambda_{\orL, \orN} = \sum_{i \in I_-(\orL, \orN)} \Lambda_i \in \zz^\ell.$$
We consider the $\Ring$-module
\begin {equation}
\label {eq:cgl}
 \C^-(G, \Lambda) = \bigoplus_{M \subseteq L}\,\, \prod_{\s \in \H(L)} \Bigl( \Chain^-(G^{L - M}, \psi^{M}(\s))[[\{U_{i,j}\}_{ L_i \subseteq M}]] \Bigr)  \otimes_\Ring \K(M),
 \end {equation}
where $\psi^{M}$ simply means $\psi^{\orM}$ with $\orM$ being the orientation induced from the one on $\orL$.

We equip $\C^-(G, \Lambda)$ with a boundary operator $\D^-$ as follows. 

For $\s \in \H(L)$ and $\x \in \bigl( \Chain^-(G^{L - M}, \psi^{M}(\s))[[\{U_{i,j}\}_{ L_i
  \subseteq M}]] \bigr) \otimes_\Ring \K(M)$, we will denote by $(\s, \x)$ the element in the direct product from \eqref{eq:cgl} whose $\s$ coordinate is $\x$ and all of whose other coordinates are zero. 
 We set
 \begin {align*} 
\D^-(\s, \x) &= \sum_{\vphantom{\orN}N \subseteq L - M}\,\, \sum_{\orN \in \Omega(N)} (\s + \Lambda_{\orL, \orN}, \Phi^{\orN}_\s(\x)) \\
&\in  \bigoplus_{\vphantom{\orN}N \subseteq L - M} \,\,\bigoplus_{\orN \in \Omega(N)} \Bigl( \s+ \Lambda_{\orL, \orN}, \bigl( \Chain^-(G^{L-M-N}, \psi^{M \cup \orN} (\s)) [[\{U_{i,j}\}_{ L_i \subseteq M \cup N}]]\bigr)  \otimes_\Ring \K(M \cup N) \Bigr) \\
 &\subset \C^-(G, \Lambda).
\end {align*}

One can show that $\D^-$ squares to zero, so that $\C^-(G, \Lambda)$ is a chain complex; see \cite[Proposition 9.4 and Section 13.6]{LinkSurg}. The complex $\C^-(G, \Lambda)$ is the analogue of the mapping cone complex $\mathbb{X}^+(n)$ from \cite{IntSurg}, which computes the Heegaard Floer homology of integral surgery on a knot. The main new aspects are that we use the minus rather than the plus version, we have several link components (which produces the hypercube structure on the complex), and we have several basepoints per link component (which forces us to use the resolutions $\Chain^-$ of $\Am$, and thus makes the formulas more complicated).

Let $H(L, \Lambda) \subseteq \zz^\ell$ be the lattice generated by $\Lambda_i, i=1, \dots, \ell$. The complex  $\C^-(G, \Lambda)$ splits into a direct product of complexes $\C^-(G, \Lambda, \ux)$, according to equivalence classes $\ux \in \H(L)/H(L, \Lambda)$. (Note that $H(L, \Lambda)$ is not a subspace of $\H(L)$, but it still acts on $\H(L)$ naturally by addition.) Further, the space of equivalence classes $\H(L)/H(L, \Lambda)$ can be canonically identified with the space of $\spc$ structures on the three-manifold $S^3_\Lambda(L)$ obtained from $S^3$ by surgery along the framed link $(L, \Lambda)$.

Given a $\spc$ structure $\ux$ on $Y_\Lambda(L)$, we set 
$$ \delt(\ux) = \gcd_{\xi \in H_2(Y_\Lambda(L); \zz)} \langle c_1(\ux), \xi\rangle,$$
where $c_1(\ux)$ is the first Chern class of the $\spc$ structure. Thinking of $\ux$ as an equivalence class in $\H(L)$, we can find a function $\nu: \ux \to \zz/(\delt(\ux)\zz)$ with the property that
$$ \nu(\s + \Lambda_i) \equiv \nu(\s) + 2s_i,$$
for any $i=1, \dots, \ell$ and $\s = (s_1, \dots, s_\ell) \in \ux$. The function $\nu$ is unique up to the addition of a constant. For $\s \in \H(L)$ and $\x \in \bigl( \Chain^-(G^{L - M}, \psi^{M}(\s))[[\{U_{i,j}\}_{ L_i
  \subseteq M}]] \bigr) \otimes_\Ring \K(M)$, let
  $$ \mu(\s, \x) = M(\x) +\nu(\s) - m,$$
where $M(\x)$ is the Maslov grading of $\x$, and $m$ denotes the number of components of the sublink $M$. Then $\mu$ gives a relative $\zz/(\delt(\ux)\zz)$-grading on the complex $\C^-(G, \Lambda, \ux)$. The differential 
$\D^-$ decreases $\mu$ by one modulo $\delt(\ux)$.

The following is \cite[Theorem 15.8]{LinkSurg}:
\begin {theorem}
\label {thm:Surgery}
Fix a grid diagram $G$ with at least one free marking, such that $G$ represents an oriented link $\orL$ in $S^3$. Fix also a framing $\Lambda$ of $L$. Then, for every $\ux \in \spc(S^3_\Lambda(L))$, we have an isomorphism of relatively graded $\Field[[U]]$-modules: 
\begin {equation}
\label {eq:surgery}
 H_*(\C^-(G, \Lambda, \ux), \D^-) \cong  \HFm_{*}(S^3_\Lambda(L), \ux).
 \end {equation}
 \end {theorem}

Observe that the left hand side of Equation~\eqref{eq:surgery} is a priori a module over $\Ring = \ff[[U_1, \dots, U_{n}]]$. Part of the claim of Theorem~\ref{thm:Surgery} is that all the $U_i$'s act the same way, so that we have a $\Field[[U]]$-module.

\subsection {Maps induced by surgery} \label {sec:Maps} Let $L' \subseteq L$ be a sublink, with the orientation $\orL'$ induced from~$\orL$. 

Denote by $W=W_{\Lambda}(L', L)$ the  cobordism from $S^3_{\Lambda|_{L'}}(L')$ to $S^3_{\Lambda}(L)$ given by surgery on $L - L'$, framed with the restriction of $\Lambda$.  Let  $H(L, \Lambda|_{L'}) \subseteq \zz^\ell$ be the sublattice generated by the framings $\Lambda_i$, for $L_i \subseteq L'$. There is an identification:
$$\spc ( W_{\Lambda}(L', L)) \cong \H(L)/H(L, \Lambda|_{L'})$$
under which the natural projection 
$$\pi^{L, L'} : \bigl( \H(L)/H(L, \Lambda|_{L'})\bigr) \longrightarrow \bigl(\H(L)/H(L, \Lambda)\bigr)$$
corresponds to restricting the $\spc$ structures to $S^3_{\Lambda}(L)$, while the map
$$ \psi^{L-L'}: \bigl(\H(L)/H(L, \Lambda|_{L'})\bigr) \to \bigl(\H(L')/H(L', \Lambda|_{L'})\bigr) $$
corresponds to restricting them to $S^3_{\Lambda|_{L'}}(L')$.

Observe that, for every equivalence class $\t \in \H(L)/H(L, \Lambda|_{L'})$,
$$ \C^-(G, \Lambda)^{L', \t} =
  \bigoplus_{L - L' \subseteq M \subseteq L}\,\prod_{\{\s \in \H(L)|[\s] = \t\}} \Bigl( \Chain^-(G^{L - M}, \psi^{M}(\s))[[\{U_{i,j}\}_{ L_i \subseteq M}]] \Bigr)  \otimes_\Ring \K(M)$$
is a subcomplex of $\C^-(G, \Lambda, \pi^{L, L'}(\t)) \subseteq \C^-(G, \Lambda)$. 
This subcomplex is chain homotopy equivalent to
 $$\C^-(G^{L'}, \Lambda|_{L'},  \psi^{L-L'}(\t)),$$ where
\begin{multline*}
\C^-(G^{L'}, \Lambda|_{L'},  \psi^{L-L'}(\t)) =\\
\bigoplus_{M' \subseteq L'} \prod_{\{\s' \in \H(L')| [\s'] = \psi^{L-L'}(\t)\}}  \Bigl( \Chain^-(G^{L' - M'}, \psi^{M'}(\s'))[[\{U_{i,j}\}_{ L_i \subseteq M'}]] \Bigr)  \otimes_{\Ring'} \K(M')
\end{multline*}
and $\Ring'$ is the power series ring in the $U_i$ variables from $L'$. The chain homotopy equivalence is induced by taking $M$ to $M' = M - (L-L')$, $\s$ to $\s'=\psi^{\orL - \orL'}(\s)$, and getting rid of the $U_i$ variables from $L - L'$ via relations coming from $\K(L-L')$.

Theorem~\ref{thm:Surgery} implies that the homology of $\C^-(G, \Lambda)^{L', \t}$ is isomorphic to 
$$ \HFm_{*}(S^3_{\Lambda'}(L'), \t|_{S^3_{\Lambda|_{L'}}(L')}).$$

In \cite{HolDiskFour}, Ozsv\'ath and Szab\'o associated a map $F^-_{W,\t}$ to any cobordism $W$ between connected three-manifolds, and $\spc$ structure $\t$ on that cobordism. In the case when the cobordism $W$ consists only of two-handles (i.e., is given by integral surgery on a link), the following theorem gives a way of understanding the map $F^-_{W, \t}$ in terms of grid diagrams:

\begin {theorem}
\label {thm:Cobordisms}
Let $\orL \subset S^3$ be an oriented link, $L' \subseteq L$ a sublink,  $G$ a grid diagram for $\orL$ with at least one free marking, and $\Lambda$ a framing of $L$. Set $W=W_{\Lambda}(L', L)$. Then, for any $\t \in \spc ( W) \cong \H(L)/H(L, \Lambda|_{L'})$, the following diagram commutes:
$$\begin {CD}
H_*(\C^-(G, \Lambda)^{L', \t}) @>{\phantom{F^-_{W, \t}}}>> H_*(\C^-(G, \Lambda, \pi^{L, L'}(\t))) \\
@V{\cong}VV @VV{\cong}V \\
\HFm_{*}(S^3_{\Lambda'}(L'), \t|_{S^3_{\Lambda|_{L'}}(L')}) @>{F^-_{W, \t}}>> \HFm_{*}(S^3_{\Lambda}(L), \t|_{S^3_{\Lambda|_{L}}(L)}).
\end {CD}$$
Here, the top horizontal map is induced from the inclusion of chain complexes, while the two vertical isomorphisms are the ones from Theorem~\ref{thm:Surgery}. 
\end {theorem}
Theorem~\ref{thm:Cobordisms} is basically \cite[Theorem 14.3]{LinkSurg}, but stated here in the particular setting of grid diagrams. See \cite[Remark 15.9]{LinkSurg}.

\subsection {Other versions} The chain complex $\C^-(G, \Lambda, \ux)$ was constructed so that the version of Heegaard Floer homology appearing in Theorem~\ref{thm:Surgery} is $\HFm$. We now explain how one can construct similar chain complexes $\hat \C(G, \Lambda, \ux), \C^+(G, \Lambda, \ux)$ and $\C^{\infty}(G, \Lambda, \ux)$, corresponding to the theories $\widehat{\iHF},  \iHF^+$ and $\HFinf$. 

The chain complex $\hat \C(G, \Lambda, \ux)$ is simply obtained from $\C^-(G, \Lambda)$ by setting one of the variables $U_i$ equal to zero. Its homology computes $\widehat{\iHF}(S^3_{\Lambda}(L), \ux)$.

The chain complex $\C^\infty(G, \Lambda, \ux)$ is obtained from
$\C^-(G, \Lambda, \ux)$ by inverting all the $U_i$ variables. It is a
vector space over the field of semi-infinite Laurent polynomials
$$\Ring^{\infty} = \Field[[U_1, \dots, U_{n}; U_1^{-1},\dots, U_{n}^{-1}]. $$
In other words, $\Ring^{\infty}$ consists of those power series in
$U_i$'s that are sums of monomials with degrees bounded from below.

Note that $C^-(G, \Lambda, \ux)$ is a subcomplex of $\C^\infty(G,
\Lambda, \ux)$. We denote the quotient complex by $\C^+(G,
\Lambda, \ux)$. Theorems~\ref{thm:Surgery} and \ref{thm:Cobordisms}
admit the following extension:

\begin {theorem}
\label {thm:AllVersions}
Fix a grid diagram $G$, with at least one free marking, representing an oriented link $\orL$
in $S^3$. Fix also a framing $\Lambda$ of $L$. Then, for every $\ux \in \spc(S^3_\Lambda(L)) \cong \H(L)/H(L, \Lambda)$, there are vertical isomorphisms and horizontal long exact sequences making the following diagram commute:
$$\begin {CD}
\cdots \to  H_*(\C^-(G, \Lambda, \ux)) @>>> H_*(\C^{\infty}(G, \Lambda, \ux))  @>>> H_*(\C^+(G, \Lambda, \ux)) \to \cdots\\
@VV{\cong}V  @VV{\cong}V @VV{\cong}V \\
\cdots \to \HFm_{*}(S^3_\Lambda(L), \ux)  @>>>  \HFinf_{*}(S^3_\Lambda(L), \ux)  @>>>  \iHF^+_{*}(S^3_\Lambda(L), \ux)  \to \cdots
 \end {CD}$$

Furthermore, the maps in these diagrams behave naturally with respect to cobordisms, in the sense that there are commutative diagrams analogous to those in Theorem~\ref{thm:Cobordisms}, involving the cobordism maps $F^-_{W, \tt}, F^{\infty}_{W, \tt}, F^+_{W, \tt}$.
\end {theorem}
Compare \cite[Theorem 14.4]{LinkSurg}.

\subsection {Mixed invariants of closed four-manifolds}

We recall the definition of the closed four-manifold invariant from \cite{HolDiskFour}. Let $X$ be a closed, oriented four-manifold with $b_2^+(X) \geq 2$. By puncturing $X$ in two points we obtain a cobordism $W$ from $S^3$ to $S^3$. We can cut $W$ along a three-manifold $N$ so as to obtain two cobordisms $W_1, W_2$ with $b_2^+(W_i) > 0$; further, the manifold $N$  can be chosen such that $\delta H^1(N; \Z) \subset H^2(W; \Z)$ is trivial. (If this is the case, $N$ is called an {\em admissible cut}.) Let $\tt$ be a $\spc$ structure on $X$ and $\tt_1, \tt_2$ its restrictions to $W_1, W_2$. In this situation, the cobordism maps
$$ F^-_{W_1, \tt_1} : \HFm(S^3) \to \HFm(N, \tt|_N)$$
and
$$ F^+_{W_2, \tt_2}: \iHF^+(N, \tt|_N) \to \iHF^+(S^3)$$
factor through $\iHF_{\operatorname{red}}(N, \tt|_N)$, where
$$ \iHF_{\operatorname{red}} = \coker(\HFinf \to \iHF^+) \cong \ker (\HFm \to \HFinf).$$
By composing the maps to and from $\iHF_{\operatorname{red}}$ we
obtain the mixed map
$$ F^{\operatorname{mix}}_{W, \tt}: \HFm(S^3) \to \iHF^+(S^3),$$
which changes degree by the quantity
$$ d(\tt) = \frac{c_1(\tt)^2 - 2\chi(X) - 3\sigma(X)}{4}.$$

Let $\Theta_-$ be the maximal degree generator in $\HFm(S^3)$.  Clearly the map $ F^{\operatorname{mix}}_{W, \tt}$ can be nonzero only when $d(\tt)$ is even and nonnegative. If this is the case, the value 
\begin {equation}
\label {eq:mixedOS}
\Phi_{X, \tt} = U^{d(\tt)/2} \cdot F^{\operatorname{mix}}_{W, \tt}(\Theta_-) \in \iHF^+_{0}(S^3) \cong \ff 
\end {equation}
is an invariant of the four-manifold $X$ and the $\spc$ structure
$\tt$. It is conjecturally the same as the Seiberg-Witten invariant of
$(X, \tt)$.

\begin {definition}
\label {def:lp}
Let  $X$ be a closed, oriented four-manifold with $b_2^+(X) \geq 2$. A {\em 
cut link presentation  for $X$} consists of a link $L \subset S^3$, a decomposition of $L$ as a disjoint union
$$ L = L_1 \amalg L_2 \amalg L_3,$$
and a framing $\Lambda$ for $L$ (with restrictions $\Lambda_i$ to $L_i, i=1, \dots, 4$)  
with the following properties:
\begin {itemize}
\item $S^3_{\Lambda_1}(L_1)$ is a connected sum of $m$ copies of $S^1
  \times S^2$, for some $m \geq 0$. We denote by $W_1$ the cobordism
  from $S^3$ to $\#^m (S^1 \times S^2)$ given by $m$ one-handle
  attachments.
\item $S^3_{\Lambda_1 \cup \Lambda_2 \cup \Lambda_3} (L_1 \cup L_2
  \cup L_3)$ is a connected sum of $m'$ copies of $S^1 \times S^2$,
  for some $m' \geq 0$. We denote by $W_4$ the cobordism from $\#^{m'}
  (S^1 \times S^2)$ to $S^3$ given by $m'$ three-handle attachments.
\item If we denote by $W_2$ resp.\ $W_3$ the cobordisms from $S^3_{\Lambda_1}(L_1)$ to $S^3_{\Lambda_1\cup \Lambda_2}(L_1 \cup L_2)$, resp.\ from $S^3_{\Lambda_1\cup \Lambda_2}(L_1 \cup L_2)$ to $S^3_{\Lambda_1 \cup \Lambda_2 \cup \Lambda_3} (L_1 \cup L_2 \cup L_3)$, given by surgery on $L_2$ resp.\ $L_3$ (i.e., consisting of two-handle additions), then 
$$ W = W_1 \cup W_2 \cup W_3 \cup W_4$$
is the cobordism from $S^3$ to $S^3$ obtained from $X$ by deleting two
copies of $B^4$.
\item The manifold $N=S^3_{\Lambda_1\cup \Lambda_2}(L_1 \cup L_2)$ is an admissible cut for $W$, i.e., $b_2^+(W_1 \cup W_2) > 0, b_2^+(W_3 \cup W_4) > 0$, and $\delta H^1(N) =0$ in $H^2(W)$.
\end {itemize}
\end {definition}

It is proved in \cite[Lemma 14.8]{LinkSurg} that any closed, oriented four-manifold $X$ with $b_2^+(X) \geq 2$ admits a cut link presentation.

\begin {definition}
Let  $X$ be a closed, oriented four-manifold with $b_2^+(X) \geq 2$. A {\em grid presentation} $\Gamma$ for $X$ consists of a cut link presentation $(L = L_1 \cup L_2 \cup L_3, \Lambda)$ for $X$, together with a grid presentation for $L$.
\end {definition}

The four-manifold invariant $\Phi_{X, \tt}$ can be expressed in terms of a grid presentation $\Gamma$ for $X$ as follows. By combining the maps $F^-_{W_2, \tt|_{W_2}}$ and $F^+_{W_3, \tt|_{W_3}}$ using their factorization through $\iHF_{\operatorname{red}}$, we obtain a mixed map
$$ F^{\mix}_{W_2 \cup W_3, \tt|_{W_2 \cup W_3}} : \HFm(\#^m (S^1 \times S^2)) \to \iHF^+(\#^{m'} (S^1 \times S^2)).$$

 Using Theorem~\ref{thm:AllVersions}, we can express the maps $F^-_{W_2, \tt|_{W_2}}$ and $F^+_{W_3, \tt|_{W_3}}$ in terms of counts of holomorphic polygons on a symmetric product of the grid. Combining these polygon counts, we get a mixed map 
$$ F^{\mix}_{\Gamma, \tt}: H_*(\C^-(G, \Lambda)^{L_1, \tt|_{\#^m(S^1 \times S^2)}}) \to H_*(\C^+(G, \Lambda)^{L_1 \cup L_2 \cup L_3, \tt|_{\#^{m'} (S^1 \times S^2)}}).$$
 We conclude that $F^{\mix}_{\Gamma, \tt}$ is the same as 
$ F^{\mix}_{W_2 \cup W_3, \tt|_{\#^m(S^1 \times S^2)}}$,
up to compositions with isomorphisms on both the domain and the target. Note, however, that at this point we do not know how to identify elements in the domains (or targets) of the two maps in a canonical way. For example, we know that there is an isomorphism
\begin {equation}
\label {eq:isoV}
H_*(\C^-(G, \Lambda)^{L_1, \tt|_{\#^m(S^1 \times S^2)}}) \cong \HFm(\#^m (S^1 \times S^2)),
\end {equation}
but it is difficult to say exactly what the isomorphism is. Nevertheless, $ \HFm(\#^m (S^1 \times S^2))$ has a unique maximal degree elements $\Theta_{\max}^m$. We can identify what $\Theta_{\max}^m$ corresponds to on the left hand side of ~\eqref{eq:isoV} by simply computing degrees. Let us denote the respective element by 
$$\Theta_{\max}^\Gamma \in H_*(\C^-(G, \Lambda)^{L_1, \tt|_{W_2 \cup W_3 }}). $$

The following proposition says that one can decide whether $\Phi_{X, \tt} \in \ff$ is zero or one from  information coming from a grid presentation $\Gamma$:

\begin {proposition}
\label {prop:mixed}
Let $X$ be a closed, oriented four-manifold $X$ with $b_2^+(X) \geq 2$, with a $\spc$ structure $\tt$ with $d(\tt) \geq 0$ even. Let $\Gamma$ be a grid presentation for $X$. Then $\Phi_{X, \tt} = 1$ if and only if  $U^{d(\tt)/2} \cdot F^{\mix}_{\Gamma, \tt} (\Theta_{\max}^\Gamma)$ is nonzero.
\end {proposition}
Compare \cite[Theorem 14.10]{LinkSurg}.
 
 \subsection{The link surgeries spectral sequence}
 \label {sec:spectral} 
 We recall the construction from \cite[Section 4]{BrDCov}. Let $M =
 M_1 \cup \dots \cup M_\ell$ be a framed $\ell$-component link in a
 3-manifold $Y$. For each $\eps = (\eps_1, \dots, \eps_\ell) \in
 \E_\ell = \{0,1\}^\ell$, we let $Y(\eps)$ be the $3$-manifold
 obtained from $Y$ by doing $\eps_i$-framed surgery on $M_i$ for $i=1,
 \dots, \ell$.

When $\eps'$ is an immediate successor to $\eps$ (that is, when $\eps < \eps'$ and $\|\eps' - \eps\| = 1$), the two-handle addition from $Y(\eps)$ to $Y(\eps')$ induces a map on Heegaard Floer homology
$$ F^-_{\eps < \eps'} : \HFm(Y(\eps)) \longrightarrow \HFm (Y(\eps')). $$

The following is the link surgery spectral sequence (Theorem 4.1 in \cite{BrDCov}, but phrased here in terms of $\HFm$ rather than $\widehat{\iHF}$ or $\iHF^+$):

\begin {theorem}[Ozsv\'ath-Szab\'o]
\label {thm:OSspectral}
There is a spectral sequence whose $E^1$ term is $\bigoplus_{\eps \in
  \E_\ell} \HFm(Y(\eps))$, whose $d_1$ differential is obtained by
adding the maps $F^-_{\eps < \eps'}$ (for $\eps'$ an immediate
successor of $\eps$), and which converges to
$E^{\infty} \cong \HFm(Y)$.
\end {theorem}

\begin{remark}
  A special case of Theorem~\ref{thm:OSspectral} gives a spectral
  sequence relating the Khovanov homology of a link and the Heegaard
  Floer homology of its branched double cover.  See
  \cite[Theorem~1.1]{BrDCov}.
\end{remark}

The spectral sequence in Theorem~\ref{thm:OSspectral} can be
understood in terms of grid diagrams as follows.

We represent $Y(0,\dots, 0)$ itself as surgery on a framed link $(L', \Lambda')$ inside $S^3$. Let $L'_1, \dots, L'_{\ell'}$ be the components of $L'$. There is another framed link $(L=L_1 \cup \dots \cup L_\ell, \Lambda)$ in $S^3$, disjoint from $L'$, such that surgery on each component $L_i$ (with the given framing) corresponds exactly to the 2-handle addition from $Y(0, \dots, 0)$ to $Y(0, \dots, 0, 1, 0, \dots, 0)$, where the $1$ is in position $i$. For $\eps \in \E_\ell$, we denote by $L^\eps$ the sublink of $L$ consisting of those components $L_i$ such that $\eps_i = 1$.

Let $G$ be a toroidal grid diagram representing the link $L' \cup L
\subset S^3$. As mentioned in Section~\ref{sec:Maps}, inside the
surgery complex $\C^-(G, \Lambda' \cup \Lambda)$ (which is an
$(\ell'+\ell)$-dimensional hypercube of chain complexes) we have
various subcomplexes which compute the Heegaard Floer homology of
surgery on the sublinks on $L' \cup L$. We will restrict our attention
to those sublinks that contain $L'$, and use the respective
subcomplexes to construct a new, $\ell$-dimensional hypercube of chain
complexes $\C^-(G, \Lambda' \cup \Lambda \hey L)$ as follows.

At a vertex $\eps \in \E_\ell$ we put the complex 
$$ \C^-(G, \Lambda' \cup \Lambda \hey L)^\eps = \C^-(G^{L' \cup L^\eps}, \Lambda' \cup \Lambda|_{L^\eps}).$$

Consider now an edge from $\eps$ to $\eps' = \eps +\tau_i$ in the hypercube $\E_\ell$. The corresponding complex $\C^-(G^{L' \cup L^\eps}, \Lambda' \cup \Lambda|_{L^\eps})$ decomposes as a direct product over all $\spc$ structures $\s$ on $Y(\eps) = S^3(L' \cup L^\eps, \Lambda' \cup \Lambda|_{L^\eps})$. As explained in Section~\ref{sec:Maps}, each factor $\C^-(G^{L' \cup L^\eps}, \Lambda' \cup \Lambda|_{L^\eps}, \s)$  admits an inclusion into $\C^-(G^{L' \cup L^{\eps'}}, \Lambda' \cup \Lambda|_{L^{\eps'}})$ as a subcomplex. In fact, there are several such inclusion maps, one for each $\spc$ structure $\t$ on the 2-handle cobordism from $Y(\eps)$ to $Y(\eps')$ such that $\t$ restricts to $\s$ on $Y(\eps)$. Adding up all the inclusion maps on each factor, one obtains a combined map
$$G^-_{\eps < \eps'} : \C^-(G, \Lambda' \cup \Lambda \hey L)^\eps  \longrightarrow \C^-(G, \Lambda' \cup \Lambda \hey L)^{\eps'}.$$

We take $G^-_{\eps < \eps'}$ to be the edge map in the hypercube of chain complexes $\C^-(G, \Lambda' \cup \Lambda \hey L)$. Since the edge maps are just sums of inclusions of subcomplexes, they commute on the nose along each face of the hypercube. Therefore, in the hypercube $\C^-(G, \Lambda' \cup \Lambda \hey L)$ we can take the diagonal maps to be zero, along all faces of dimension at least two. 

This completes the construction of $\C^-(G, \Lambda' \cup \Lambda \hey L)$. As an $\ell$-dimensional hypercube of chain complexes, its total complex admits a filtration by $-\|\eps\|$, which induces a spectral sequence; we refer to the filtration by $-\|\eps\|$ as the {\em depth filtration} on 
$\C^-(G, \Lambda' \cup \Lambda \hey L)$.

The following is \cite[Theorem 14.12]{LinkSurg}, adapted here to the setting of grid diagrams:

\begin {theorem}
\label{thm:SpectralSequence}
Fix a grid diagram $G$ with at least one free marking, such that $G$ represents an oriented link $\orL' \cup \orL$ in $S^3$. Fix also framings $\Lambda$ for $L$ and $\Lambda'$ for $L'$.  Suppose  that $L$ has $\ell$ components $L_1, \dots, L_\ell$. Let $Y(0,\dots,0) = S^3_{\Lambda'}(L')$, and let $Y(\eps)$ be obtained from $Y(0,\dots,0)$ by surgery on the components $L_i \subseteq L$ with $\eps_i = 1$ (for any $\eps \in \E_\ell$). Then, there is an isomorphism between the link surgeries spectral sequence from Theorem~\ref{thm:OSspectral} and the spectral sequence associated to the depth filtration on $\C^-(G, \Lambda' \cup \Lambda \hey L)$.   
\end{theorem}

\section {Enhanced domains of holomorphic polygons}
\label {sec:enhanced}

\subsection {Domains and shadows}
\label {sec:shadows}

In the construction of the complex $\C^-(G, \Lambda)$ in Section~\ref{sec:surgery},
the only non-combinatorial ingredients were the holomorphic polygon counts in the definition of the descent maps $\hat\De_{p^{\orM}(\s)}^{\orM}$. 

As a warm-up to understanding the descent maps, let us consider the simpler polygon maps 
$$ \De^{\zed} : \CFm(\Hyper_G^{\orM}(\zed_0)) \to  \CFm(\Hyper_G^{\orM}(\zed_0 \cup \zed)).$$
from Section~\ref{sec:PolygonMaps}. (We use here the notation from Remark~\ref{rem:note}.) Here, $\zed_0$ and $\zed$ are disjoint collections of markings such that $\zed_0 \cup \zed$ is consistent. Further, $\orM$ is an oriented sublink containing $\orL(\zed_0 \cup \zed)$, such that the orientations are compatible (but do not necessarily agree with the orientation induced from $\orL$).

According to Equation~\eqref{eq:dezeds}, 
the maps $\De^\zed$ are in turn summations of maps of the form $\De^{(\zed)}$, where $(\zed)$ denotes an ordering of a consistent collection of markings $\zed$.

Let $(\zed) = (Z_1, \dots, Z_k)$. The maps
$$ \De^{(\zed)} :  \CFm(\Hyper_G^{\orM}(\zed_0)) \to  \CFm(\Hyper_G^{\orM}(\zed_0 \cup \zed))$$
correspond to handleslides (in the given order) at a set of markings $\zed$, starting
with a diagram $G$ already destabilized at a base set of markings
$\zed_0$. Our goal is to get as close as possible to a combinatorial description of these
maps. In light of the isomorphisms~\eqref{eq:psiz0}, we can assume
without loss of generality that $\zed_0 = \emptyset$. (Indeed, the polygon maps commute with these isomorphisms, by the results on quasi-stabilizations from \cite[Section 6]{LinkSurg}.) 

The difficulty lies in understanding the counts of pseudo-holomorphic polygons $\# \M(\phi)$, where
$$  \phi \in \pi_2(\x, \Theta^{\can}_{\emptyset, \{Z_1\}}, \Theta^{\can}_{\{Z_1\}, \{Z_1, Z_2\}}, \dots, \Theta^{\can}_{\{Z_1, \dots, Z_{k-1}\}, \{Z_1, \dots, Z_{k}\} } , \y)$$ 
is a homotopy class of $(2+k)$-gons with edges on $$\Ta, \Tb, \Tb^{\{Z_1\}}, \Tb^{\{Z_1, Z_2\}}, \dots, \Tb^{\{Z_1, \dots, Z_k\}},$$ in this cyclic order. The Maslov index $\mu(\phi)$ is required to be $1-k$.

The same polygon counts appear in the definition of the descent maps. There, we keep track of the counts in more complicated ways, using $U$ and $V$ variables, with exponents that involve the multiplicities $O_{i,j}(\phi), X_{i,j}(\phi)$ and $O_j(\phi)$. 

As in \cite[Definition 2.13]{HolDisk}, every homotopy class $\phi$ has
an associated {\em domain} $D(\phi)$ on the surface $\TT$. The domain
is a linear combination of regions, i.e., connected components of the
complement in $\TT$ of all the curves $\alphas, \betas, \betas^{\{Z_1\}}, \betas^{\{Z_1, Z_2\}}, \dots, \betas^{\{Z_1, \dots, Z_k\}}$. If $\phi$ admits a
holomorphic representative for some almost complex structure, then $D(\phi)$ is a linear combination of regions, all appearing with nonnegative coefficients. (Here and in the rest of the paper, we implicitly assume that the almost complex structures are nearly symmetric in the sense of \cite[Definition 3.1]{HolDisk}; that is, we pick one basepoint $z_i$ in each region, and ask that the structures be standard in the neighborhood of each divisor $\{z_i\} \times \Sym^{n-1}(\TT)$, and also standard in the neighborhood of the diagonal. Given this, positivity of the domain follows from \cite[Lemma 3.2]{HolDisk}.) 

Let us mark an asterisk in each square of the grid diagram $G$. When
we construct the new beta curves $\betas^{\{Z_1\}}, \betas^{\{Z_1, Z_2\}}, \dots, \betas^{\{Z_1, \dots, Z_k\}}$ (all obtained from the original $\betas$ curves by handleslides), we make sure that each beta curve
encircling a marking does not include an asterisk, and also that
whenever we isotope a beta curve to obtain a new beta curve
intersecting it in two points, these isotopies do not cross the
asterisks.  Then the regions on $\TT$ fall naturally into two types:
{\em small} regions, which do not contain asterisks, and {\em large}
regions, which do. We define the {\em shadow} $\sh(R)$ of a large
region $R$ to be the square in $G$ containing the same asterisk as
$R$; and the shadow of a small region to be the empty set.

If 
$$D = \sum a_i R_i$$
is a linear combination of regions, with $a_i \in \zz$, we define its
shadow to be
$$ \sh(D) = \sum a_i \sh(R_i).$$

The {\em shadow} $\sh(\phi)$ of a homotopy class $\phi$ is defined as $\sh(D(\phi))$. See Figure~\ref{fig:shadow}.

\begin{figure}
\begin{center}
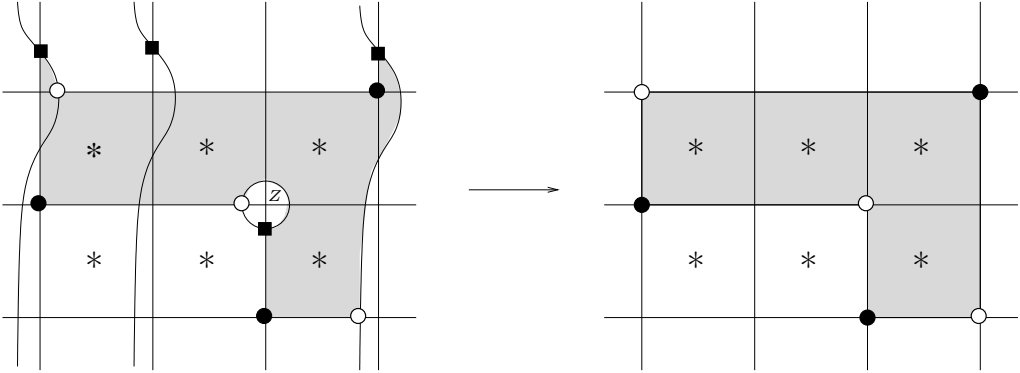
\end{center}
\caption {{\bf The domain of a triangle and its shadow.}  On the left,
  we show the domain $D(\phi)$ of a homotopy class $\phi$ of triangles
  in $\Sym^n(\TT)$. The alpha and beta curves are the horizontal and
  vertical straight lines, respectively, while the curved arcs
  (including the one encircling $Z$) are part of the $\beta^{\{Z\}}$
  curves. The black squares mark components of $\Theta^{\can}_{\emptyset, \{Z\}} \in \Tb
  \cap \Tb^{\{Z\}}$. On the right, we show the shadow $\sh(\phi)$.}
\label{fig:shadow}
\end{figure}

We will study homotopy classes $\phi$ by looking at their shadows,
together with two additional pieces of data as follows.

First, each homotopy class $\phi$ corresponds to a $(2+k)$-gon where
one vertex is an intersection point $\y \in \Ta \cap \Tb^{\zed}$. In particular, $\y$ contains exactly one of the two
intersection points between the alpha curve just below the marking
$Z_i$ and the beta curve encircling $Z_i$, for $i=1, \dots, k$. Set
$\epsilon_i = 0$ if that point is the left one, and $\epsilon_j = 1$
if it is the right one. (This corresponds to the fact that a generator $\y$ containing the right point has Maslov grading one higher than if it had the left point.)

Second, for $i=1, \dots, k$, let $\rho_i \in \Z$ be the multiplicity of the domain $D(\phi)$ at the marking $Z_i$. 

Then, we define the {\em enhanced domain} $E(\phi)$ associated to
$\phi$ to be the triple $(\sh(\phi), \epsilon(\phi), \rho(\phi))$
consisting of the shadow $\sh(\phi)$, the collection
$\epsilon(\phi)=(\epsilon_1, \dots, \epsilon_k)$ corresponding to
$\y$, and the set of multiplicities $\rho(\phi) = (\rho_1, \dots,
\rho_k)$.

\subsection {Grid diagrams marked for destabilization}
\label {sec:uleft}

We now turn to studying enhanced domains on their own (rather than associated to some $\phi$). Note that in the previous section we considered polygon maps where we only keep track of the markings determined by the sublink $\orM$; that is, the markings in $\Os^{\orM}$ and $\Xs^{\orM}$. Further, we assumed $\zed_0=\emptyset$, and the set $\zed=\{Z_1, \dots, Z_k\}$ where we did the handleslides was a subset of $\Os^{\orM}$, coming from some linked $O$ and $X$ markings on the original grid $G$. We can view $\zed$ as a set of free markings on the reduction $r_{\orM}(G)$. In fact, by replacing $G$ with $r_{\orM}(G)$, we can focus on the following setting.

Consider a toroidal grid diagram $G$ of grid number $n$, with $q \geq 1$ free $O$ markings. (Our $G$ here corresponds to $r_{\orM}(G)$ from the previous subsection, and $q$ corresponds to the previous $q$ plus the number of $O$ basepoints on $M$.) We consider the $n$ $O$'s with associated variables $U_1,
\dots, U_n$, the last $q$ of which are free. We also have $n-q$ linked $X$ markings, with associated variables $\tU_1, \dots, \tU_{n-q}$.

Some subset
$$ \zed = \{O_{i_1}, \dots, O_{i_k} \} $$
of the free $O$'s, corresponding to indices $i_1, \dots, i_k
\in \{n-q+1, \dots, n\}$, is {\em marked for destabilization}. (These are the
markings $Z_1, \dots, Z_k$ from the previous subsection.) We will assume\footnote{This assumption is always satisfied in our setting, because in the previous subsection we had a grid with at least one free basepoint, and only did handleslides over some linked basepoints.} that $k < n$. The
corresponding points on the lower left of each $O_{i_j}$ (the intersections of the $\alpha$ curve just below $O_{i_j}$ with the $\beta$ curve just to its left) are denoted
$p_j$ and called {\em destabilization points}.

Recall that in Section~\ref{subsec:ags} we defined a Floer complex $C^-(G)=\CFm(\Ta, \Tb)$, generated (as a module over the ring $\Ring$) by the set $\S(G)=\Ta \cap \Tb$. We also had a curved   link Floer complex $\CFLc(G)$, generated by the same $\S(G)$ but over the bigger ring $\tR$, which includes $\tU$ variables to keep track of the $X$ markings.

Let $G^\zed$ be the destabilized diagram, and let $\Ring^\zed$ and $\tR^\zed$ be the corresponding ground rings, so that, for example,
$$ \Ring = \Ring^\zed[[\{U_Z\}_{Z \in \zed}]].$$
We have a Floer complex $C^-(G^\zed)$ over $\Ring^\zed$ and a curved Floer complex $\CFLc(G^\zed)$ over $\tR^\zed$. We identify $\S(G^\zed)$ with a subset of $\S(G)$ by adjoining to an element of $\S(G^\zed)$ the destabilization points.

The set of {\em enhanced generators} $\ES(G, \zed)$ consists of pairs $(\y,
\epsilon)$, where $\y\in \S(G^\zed)$ and $\epsilon = (\epsilon_1, \dots,
\epsilon_k)$ is a collection of markings ($\epsilon_j = 0$ or $1$) at each
destabilization point.  We will also denote these markings as $L$ for
a left marking ($\epsilon_j = 0$) or $R$ for a right marking
($\epsilon_j = 1$).
The homological grading of an
enhanced generator is
\[ M(\y, \epsilon) = M(\y) + \sum_{j=1}^k \epsilon_j. \]

We define an enhanced destabilized complex $\EC^-(G, \zed)$ whose generators are $\ES(G, \zed)$. It is formed as the tensor product (over $\Ring$) of
$$ C^-(G^\zed) \otimes_{\Ring^\zed} \Ring = C^-(G^\zed)[[\{U_Z\}_{Z \in \zed}]]$$
and $k$ mapping cones 
\begin {equation}
\label {mapcone}
\Ring \xrightarrow{U_{i_j} - U_{i'_j}} \Ring
\end {equation}
given by maps from $L$ to $R$, where
$U_{i'_j}$ corresponds to the $O$ in the row directly below the
destabilization point $p_j$. For a suitable choice of almost-complex
structure on the symmetric product, there is a natural isomorphism as in \eqref{eq:psiz0},
$$\EC^-(G, \zed) \cong \CFm(\Hyper_G(\zed)).$$  

Here, the diagram $\Hyper_G(\zed)$ is the one obtained from $G$ by handlesliding $\beta$ curves over the markings in $\zed$.\footnote{We first used the notation $\Hyper_G(\zed)$ in Section~\ref{sec:DestabSeveral}, where $\zed$ was a set of linked markings. In our current setting, since the reduction $r_{\orM}G$ plays the role of $G$, our set $\zed$ is made of free markings. Thus, $\Hyper_G(\zed)$ still represents the link $\orL$, just like $G$.} Since the diagram $\Hyper_G(\zed)$ is obtained from $G$ by handleslides, the complexes $\CFm(\Hyper_G(\zed))$ and $\CFm(G)=C^-(G)$ are chain homotopy equivalent (via maps counting pseudo-holomorphic triangles). Therefore, $\EC^-(G, \zed)$ and $C^-(G)$ are chain homotopy equivalent and hence quasi-isomorphic.

A combinatorial version of this fact appeared in \cite[Section~3.2]{MOST}, in the case $k=1$. Then, there is a quasi-isomorphism 
\begin{equation}
  \label{eq:destab-chain-map}
  F: C^-(G) \to \EC^-(G, \zed)
\end{equation}
given by counting snail-like domains as in Figure~\ref{fig:ULeft}. For future reference, let us give a formal definition of these domains, following \cite[Definition  3.4]{MOST}.

\begin{definition}
\label{def:snail}
Let $G$ be a grid diagram with a single $O=O_{i_1}$ marked for destabilization, and let $p=p_1$ be the destabilization point. A domain $D$ from $\x \in \S(G)$ to $\y \in \S(G^{\zed}) \subset S(\G)$ is said to be {\em snail-like (of type L or R)} if it satisfies the following conditions:
\begin{itemize}
\item $D$ has only non-negative local multiplicities;
\item For each $c \in \x \cup \y$, other than $p$, at least three of the four adjoining squares have vanishing local multiplicities;
\item In a neighborhood of $p$, the local multiplicities in three of the adjoining squares are the same number $m$. When $D$ has type $L$, the lower left corner has local multiplicity $m-1$, while for $D$ of type $R$, the lower right corner has multiplicity $m+1$.
\item The boundary $\del D$ is connected.
\end{itemize}
\end{definition}

It is easy to see that the shape of a snail-like domain is highly constrained. For each type ($L$ or $R$), there is an infinite sequence of possible shapes, as shown in Figure~\ref{fig:ULeft}.

Let us now go back to the case of a grid diagram $G$ with an arbitrary number $k$ of $O$'s marked for destabilization. If we want to keep track of the $X$ markings as well, we construct an enhanced curved complex $\ECFLc(G, \zed)$ as the tensor product of 
$$ \CFLc(G^\zed)  \otimes_{\tR^\zed} \Ring =\CFLc(G^\zed)[[\{U_Z\}_{Z \in \zed}]] $$
and $k$ mapping cones of the form
$$\tR \xrightarrow{U_{i_j} - U_{i'_j}} \tR,$$
or
$$\tR \xrightarrow{U_{i_j} - U_{i'_j}\tU_{i'_j}} \tR.$$
Here, we use the latter type of map $U_{i_j} - U_{i'_j} \tU_{i'_j}$ in case the $O$ in the row directly below $p_j$ is a linked marking; we then let $\tU_{i_j}$ be the variable corresponding to the $X$ in the same row. 

We have an isomorphism
$$\ECFLc(G, \zed) \cong \CFLc(\Hyper_G(\zed)),$$
and a chain homotopy equivalence between these curved complexes and $\CFLc(G)$.

\begin{figure}
\begin{center}
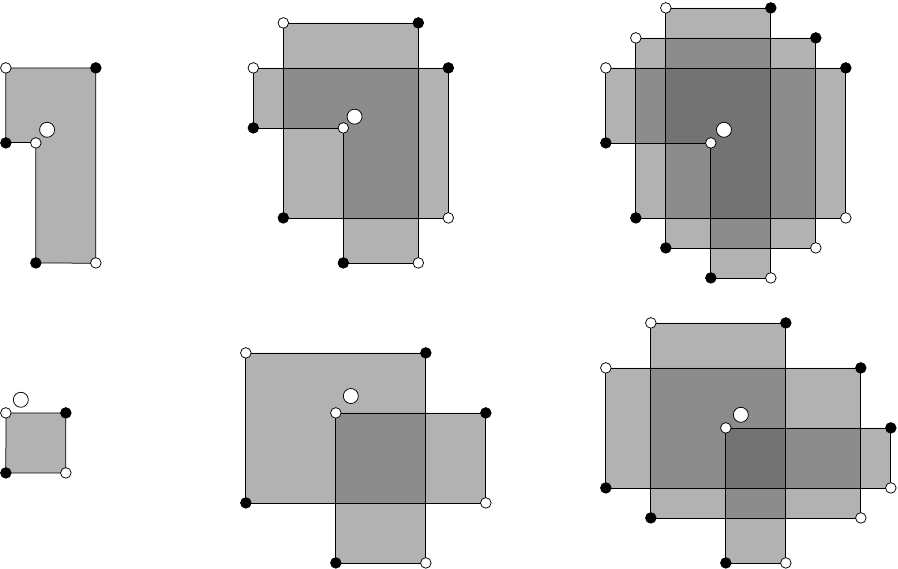
\end{center}
\caption {{\bf Snail-like domains for destabilization at one point.}
We label initial points by dark circles, and terminal points by empty circles.
The top row lists domains of type $L$ (i.e., ending in an enhanced generator with an $L$ marking), while the second row lists
some of type $R$. Darker shading corresponds to
higher local multiplicities.  The domains in each row (top or bottom) are part of an infinite series, corresponding to increasing complexities. The series in the first row also contains the trivial domain of
type~$L$, not shown here.}
\label{fig:ULeft}
\end{figure}

\begin{definition}
  \label{def:EnhancedDomain}
Given a domain $D$ on the $\alpha$-$\beta$ grid (that is, a linear combination of squares), we let $O_i(D)$ be the multiplicity of $D$ at $O_i$. We define an {\em enhanced domain} $(D, \epsilon, \rho)$ to consist of:

\begin {itemize}
\item A domain $D$ on the grid between points $\x \in \S(G)$ and $\y \in \S(G^\zed)$ (in particular, the final configuration contains all destabilization points);

\item A set of markings $\epsilon = (\epsilon_1, \dots, \epsilon_k)$
  at each destabilization point (so that $(\y, \epsilon)$ is an
  enhanced generator); and

\item A set of integers $\rho=(\rho_1, \dots, \rho_k)$, one for each destabilization point.
\end {itemize}
\end{definition}

We call $\rho_j$ the {\em real multiplicity} at $O_{i_j}$. The number
$t_j = O_{i_j}(D)$ is called the {\em total multiplicity}, and the
quantity $f_j = t_j - \rho_j$ is called the {\em fake multiplicity} at
$O_{i_j}$. The reason for this
terminology is that, if $D$ is the shadow of a holomorphic
$(k+2)$-gon~$D'$, then the real multiplicity $\rho_j$ is the multiplicity
of $D'$ at $O_{i_j}$. On the other hand, the fake multiplicity is the difference between the multiplicity that appears in the shadow $D$ (namely, $t_j$) and the one in the polygon $D'$.  For example, in the domain pictured in Figure~\ref{fig:shadow}, the fake multiplicity at the $Z$ marking is one.

Consider the full collection of real multiplicities 
$$(N_1, \dots,
N_n, \tN_1, \dots, \tN_{n-q}),$$ where $N_{i_j} = \rho_j$ for $j \in \{1, \dots, k\}$ and $N_i
= O_i(D)$ when $O_i$ is not one of the $O$'s used for
destabilization; also, $\tN_i=X_i(D)$. We say that the enhanced domain $(D, \epsilon, \rho)$
goes from the generator $\x$ to $U_1^{N_1} \cdots U_n^{N_n} \tU_1^{\tN_1} \dots \tU_{n-q}^{\tN_{n-q}} \cdot
(\y, \epsilon)$. These are called the initial and final point of the
enhanced domain, respectively.

We define the \emph{index} of an enhanced domain to be
\begin{equation}
  \label{eq:enhanced-index}
\begin {aligned}
I(D, \epsilon, \rho) &= M(\x) - M(U_1^{N_1} \cdots U_n^{N_n} \tU_1^{\tN_1} \dots \tU_{n-q}^{\tN_{n-q}} \cdot (\y, \epsilon)) \\ 
&= M(\x) - M(\y) - \sum_{j=1}^k \epsilon_j  + 2\sum_{j=1}^n N_i \\ 
&= M(\x) - M(U_1^{O_1(D)}\cdots U_n^{O_n(D)}\cdot \y) - \sum_{j=1}^k (\epsilon_j + 2 f_j) \\
&= I(D) -  \sum_{j=1}^k (\epsilon_j + 2 f_j).
\end {aligned}
\end{equation}
Here $I(D)$ is the ordinary Maslov index of $D$, given by Lipshitz's formula \cite[Corollary 4.3]{LipshitzCyl}:
\begin {equation}
\label {eq:lipshitz}
 I(D) = \sum_{x \in \x} n_x(D) + \sum_{y \in \y} n_y(D),
 \end {equation}
 where $n_p(D)$ denotes the average multiplicity of $D$ in the four
 quadrants around the point $p$. (Lipshitz's formula in the reference
 has an extra term, the Euler measure of $D$, but this is zero in our
 case because $D$ can be decomposed into rectangles.)

For example, consider the snail-like domains from Definition~\ref{def:snail}, as shown in Figure~\ref{fig:ULeft}. We turn them into enhanced domains by adding the respective marking ($L$ or
$R$) as well as choosing the real multiplicity at the destabilization
point to be zero. Then all those domains have index zero, regardless
of how many (non-destabilized) $O$'s they contain and with what
multiplicities.

\begin {lemma}
  Suppose that $(D, \epsilon, \rho)$ is the enhanced domain associated
  to a homotopy class $\phi$ of $(k+2)$-gons as in
  Section~\ref{sec:uleft}. Then the Maslov indices match: $\mu(\phi) = I(D,
  \epsilon, \rho)$.
\end {lemma}

\begin {proof}
  Note that $I(D,\epsilon, \rho)$ is the difference in Maslov grading 
  between $\x$ (in the original grid diagram) and $\tilde \y =
  U_1^{N_1} \cdots U_n^{N_n} \tU_1^{\tN_1} \dots \tU_{n-q}^{\tN_{n-q}} \cdot (\y, \epsilon)$ (in the destabilized
  diagram), cf.\ Equation~\eqref{eq:enhanced-index}. Because
  $\mu(\phi)$ is additive with respect to pre- or post-composition
  with rectangles, we must have $\mu(\phi) = I(D, \epsilon, \rho)+C$,
  where $C$ is a constant that only depends on the grid. To compute
  $C$, consider the trivial enhanced domain with $D=0, \epsilon = (0,
  \dots, 0), \rho =(0, \dots, 0)$. This is associated to a class
  $\phi$ whose support is a disjoint union of $n$ polygons, all of
  whom are $(k+2)$-gons with three acute angles and $k-1$ obtuse
  angles. See Figure~\ref{fig:local} for an example of a quadrilateral
  of this type. It is easy to check that $\mu(\phi)=0$, which implies
  $C=0$.
\end {proof}

Given an enhanced domain $(D, \epsilon, \rho)$, we denote by $a_j,
b_j, c_j, d_j$ the multiplicities of $D$ in the four squares around
$p_j$, as in Figure~\ref{fig:abcd}. In particular, $b_j = t_j$ is the
total multiplicity there. Note that, if $p_j \not \in \x$, then
\begin{align}
\label {eq:abcd1}
 a_j + d_j &= b_j + c_j +1,\\
\shortintertext{while if $p_j \in \x$ then}
a_j + d_j &= b_j + c_j.\label {eq:abcd0}
\end{align}

\begin{figure}
\begin{center}
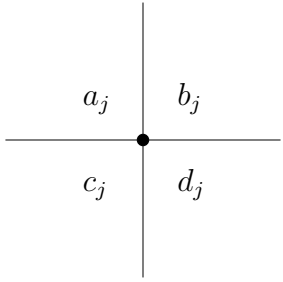
\end{center}
\caption {{\bf Local multiplicities around a destabilization point.}
The marked point in the middle is the destabilization point ($p_j$ for left destabilization and $q_j$ for right destabilization). 
}
\label{fig:abcd}
\end{figure}

\begin{definition}
  \label{def:PositiveEnhancedDomain}
  We say that the enhanced domain $(D, \epsilon, \rho)$ is {\em
    positive} if $D$ has nonnegative multiplicities everywhere in the
  grid, and, furthermore, for every $j \in \{1, \dots, k \}$, we have
\begin {equation}
\label {ineq:abcd1}
\begin{aligned}
a_j &\geq f_j,&b_j &\geq f_j,\\
c_j &\geq f_j + \epsilon_j -1,\qquad &d_j &\geq f_j + 
\epsilon_j.
\end{aligned}
\end {equation}
\end{definition}

Observe that the second of the above inequalities implies $b_j - f_j = \rho_j 
\geq 0$. Thus, in a positive enhanced domain, all real multiplicities $\rho_j$ are nonnegative.

\begin {lemma}
\label{lemma:Der}
  Suppose that $(D, \epsilon, \rho)$ is the enhanced domain associated
  to a homotopy class $\phi$ of $(k+2)$-gons as in
  Section~\ref{sec:uleft}, and that the homotopy class $\phi$ admits
  at least one pseudo-holomorphic representative, for some (nearly symmetric) almost complex structure. Then $(D, \epsilon, \rho)$
  is positive, in the sense of Definition~\ref{def:PositiveEnhancedDomain}.
\end {lemma}

\begin {proof}
  If a homotopy class $\phi$ has at least one holomorphic
  representative, then it has positive multiplicities (for suitable choices
  of almost-complex structure on
  the symmetric product, see for example Lemma~3.2 of~\cite{HolDisk}).
  It follows now that $D$ has nonnegative multiplicities everywhere in
  the grid and that $\rho_j = b_j - f_j \geq 0$. It remains to check
  the three other relations.

  For concreteness, let us consider the case $k=2$, with $\epsilon_j =
  0$, cf.\ Figure~\ref{fig:local}. The two circles encircling the destabilization point in the
  middle are of the type $\beta^{\{Z_1\}}$ (the right one) and
  $\beta^{\{Z_1,Z_2\}}$ (the left one).

\begin{figure}
\begin{center}
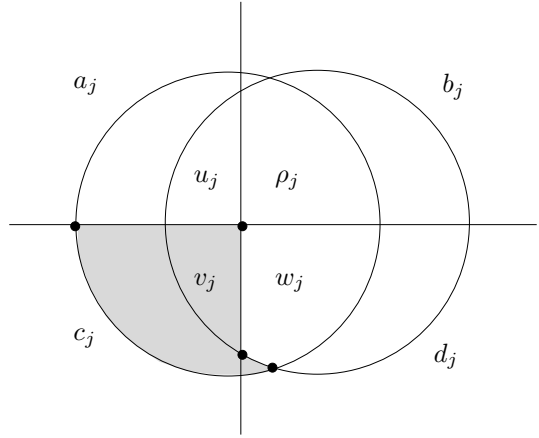
\end{center}
\caption {{\bf Local multiplicities in detail.}  This is a more
  detailed version of Figure~\ref{fig:abcd}, in which we show two
  other types of beta curves: $\beta^{\{Z_1\}}$ and
  $\beta^{\{Z_1, Z_2\}}$. We mark the local multiplicity in each region. The shaded region is the domain of a quadrilateral of index zero.}
\label{fig:local}
\end{figure}

If we are given the domain of a holomorphic quadrilateral (for
instance, the shaded one in Figure~\ref{fig:local}), at most
of the intersection points the alternating sum of nearby multiplicities is
zero.  The only exceptions are the four bulleted points in the figure,
which correspond to the vertices of the quadrilateral, and where the
alternating sum of nearby multiplicities is $\pm 1$.  
(At the central intersection between an alpha and an original beta curve, there may not be a vertex and the alternating sum of the multiplicities may be~$0$.)
If $u_j$, $v_j$, and $w_j$ are the multiplicities in the regions indicated in
Figure~\ref{fig:local}, by adding up
suitable such relations between local multiplicities, we obtain
\begin{alignat*}{3}
a_j-f_j      &= a_j + \rho_j - b_j    &&= u_j &&\!\geq 0\\
c_j -f_j + 1 &= c_j +\rho_j - b_j + 1 &&= v_j &&\!\geq 0\\
d_j -f_j     &= d_j + \rho_j - b_j    &&= w_j &&\!\geq 0,
\end{alignat*}
as desired. The proof in the $\epsilon_j = 1$ case, or for other values of $k$, is similar.  
\end {proof}

\begin{definition}
\label{def:EquivEnhanced}
We say that two enhanced domains $E=(D, \epsilon, \rho)$ and $E'=(D', \epsilon', \rho')$ are {\em equivalent} if $\epsilon=\epsilon'$, $\rho=\rho'$, and the difference $D-D'$ is a linear combination of periodic domains of the following three types:
 \begin {enumerate}[(a)]
\item A column minus the row containing the same free $O_i$;
\item The sum of the columns supporting a link component $L_i \subseteq L$, minus the sum of the rows supporting the same $L_i$;
\item A column containing one of the $O_{i_j}$'s used for destabilization.
\end {enumerate}
\end{definition}

\begin{remark}
Adding a periodic domain of type (c) (a column containing some  $O_{i_j}$) makes the total
multiplicity $t_j$ (and hence also the fake multiplicity $f_j$) increase by $1$. The real multiplicity $\rho_j$ does not change.
\end{remark}

\begin{lemma}
\label{lem:EquivEnhanced}
The equivalence class of an enhanced domain $E=(D, \epsilon, \rho)$ is determined by
its initial point~$\x$ and final point $\tilde \y = U_1^{N_1} \cdots
U_n^{N_n} \tU_1^{\tN_1} \dots \tU_{n-q}^{\tN_{n-q}}\cdot (\y, \epsilon)$.
\end{lemma}

\begin{proof} Clearly, $\x$ and $\tilde \y$ do not change if we add periodic domains to $E$ as in Definition~\ref{def:EquivEnhanced}. 

Conversely, we need to show that if two enhanced domains $E=(D, \epsilon, \rho)$ and $E'=(D', \epsilon', \rho')$ go between the same $\x$ and $\tilde \y$, then they are equivalent. Since $\tilde \y$ is the same, we have $\epsilon=\epsilon'$ and the values of $N_1, \dots, N_n, \tN_1, \dots, \tN_{n-q}$ are the same for $E$ and $E'$. In particular, we have 
$$ \rho_j = N_{i_j} = \rho'_j,$$ 
so $\rho=\rho'$. It remains to study the difference $$\Delta:=D-D'.$$ Because $D$ and $D'$ go between the same $\x$ and $\y$, we see that $\Delta$ is a domain on the grid whose boundary is a linear combination of alpha and beta curves; that is, $\Delta$ is a linear combination of rows and columns. Furthermore, the multiplicities of $\Delta$ have to be zero at all $X$'s, and zero at all $O$'s that were not used for destabilization.

We claim that $\Delta$ is a linear combination of the kinds of periodic domains listed in Definition~\ref{def:EquivEnhanced}. After subtracting from $\Delta$ a suitable number of columns through each destabilization point $O_{i_j}$ (which are type (c) periodic domains), we can assume that the multiplicity of the new $\Delta$ at all $O_{i_j}$, and hence at all $O$ and $X$ markings, is zero. Thus, when writing the new $\Delta$ as a linear combination of rows and columns, the coefficients of the row and the column through the same marking must be equal. It follows that $\Delta$ is a linear combination of periodic domains of types (a) and (b), as desired.
\end{proof}

\begin{definition}
  \label{def:PositivePair}
We say that the pair $[\x, \tilde \y]$ is {\em positive} if there
exists a positive enhanced domain with $\x$ and $\tilde \y$ as its
initial and final points.
\end{definition}

\subsection {Positive domains of negative index}

Figure~\ref{fig:ULeft2} shows some examples of positive enhanced
domains of negative index, in a twice destabilized grid. We see that
the index can get as negative as we want, even if we fix the number of
destabilizations at two (but allow the size of the grid to change).

\begin{figure}
\begin{center}
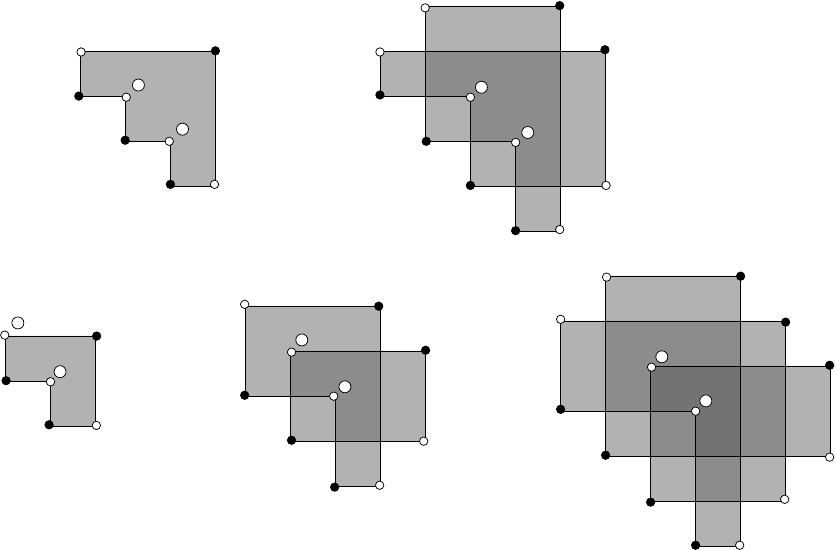
\end{center}
\caption {{\bf Positive domains of negative index.} We destabilize at
  two $O$'s, marked by the two larger circles in the picture. The top
  row shows two positive domains of type $LL$, the first of index
  $-1$, the second of index $-2$. The second row shows three positive
  domains of type $RL$, of indices $-1, -2$ and $-3$,
  respectively. Darker shading corresponds to higher local
  multiplicities.  In each case, the real multiplicity $\rho_j$ is zero.}
\label{fig:ULeft2}
\end{figure}

However, this phenomenon is impossible for one destabilization:
\begin{proposition}
\label{prop:oneO} 
Let $G$ be a toroidal grid diagram with only one $O$ marked for
destabilization. Then any positive enhanced domain has nonnegative
index. Furthermore, if the enhanced domain is positive and has index
zero, then it is snail-like, in the sense of Definition~\ref{def:snail}.
\end{proposition}

\begin {proof}
  Let us first consider a positive enhanced domain $(D, \epsilon,
  \rho)$ going between the generators $\x$ and $\tilde \y = U_1^{N_1}
  \cdots U_n^{N_n}\tU_1^{\tN_1} \dots \tU_{n-q}^{\tN_{n-q}} \cdot (\y, \epsilon)$. With the notations of
  Section~\ref{sec:uleft}, we have $k=1$, $\epsilon_1 \in \{0,1\}$
  (corresponding to the type of the domain, $L$ or $R$), $a_1$, $b_1 =
  t_1$, $c_1$, $d_1$ are the local multiplicities around the
  destabilization point, and $\rho_1$ and $f_1 = t_1 - \rho_1$ are the
  real and fake multiplicities there. Note that $t_1 \geq f_1$.

  Without loss of generality we can assume that every row or column
  contains at least one square where the multiplicity of $D$ is zero;
  otherwise we could subtract that row or column and obtain
  another positive domain, whose index is no larger. Indeed, if the row or column does not contain the destabilization point $p_1$ on its boundary, then subtracting it decreases the index by $2$. If we subtract the row just to the left of $p_1$, or the column just below $p_1$, by simultaneously increasing $\rho_1$ by one we can preserve the positivity conditions and leave the index the same. If we subtract the row or column through the $O$ marked for destabilization, while leaving $\rho$ unchanged, the positivity and the index are again unaffected.

  We say that an alpha or a beta curve is \emph{special} if $\x$ and $\y$ intersect
  it in different points. Let $m$ be the number of special alpha curves; it is
  the same as the number of special beta curves. 
  
  We claim that, as we move from a square to another across 
  an alpha or a beta curve, the multiplicity of the domain $D$ can
  only change by $\pm 1$, and it can do that only if the row (or
  column) is special. Indeed, if the curve is not special, then the difference $\delta$ in multiplicity between the two sides  is the same all along the curve. By our assumption, we have zeros in every row and column, and multiplicities are always nonnegative. By considering a square with multiplicity zero on one side of the curve, we get $\delta \leq 0$. By considering one on the other side, we get $\delta \geq 0$, and we conclude that $\delta$ can only be zero. If the curve is special, its intersections with $\x$ and $\y$ split it into two intervals; the difference in multiplicity of the two sides is some $\delta \in \Z$ along one interval, and $\delta+1$ along the other. By considering squares with zero multiplicity again, we find that $\delta \in \{0, -1\}$, and the claim is proved.
 
  Let us now look at the column containing the destabilized $O$. The
  domain has multiplicity $d_1$ in the spot right below  $O$, and
  it has multiplicity zero on some other spot on that column. We can
  move from $O$ to the multiplicity zero spot either by going up or
  down. As we go in either direction, we must encounter at least $d_1$
  special rows. This means that the total number of special rows is at
  least $2d_1$. Using the fourth inequality in~\eqref{ineq:abcd1}, we
  get
\begin {equation}
\label {eq:mt}
m \geq 2d_1 \geq 2(f_1+ \epsilon_1).
\end {equation} 

The ordinary index of the domain $D$ is given by Equation~\eqref{eq:lipshitz}, involving the sums of the average local
multiplicities of $D$ at the points of $\x$ and $\y$. One such point is the destabilization point $p_1$ which
is part of $\y$. Using the relations in Equation~\eqref{ineq:abcd1}, we find that
the average vertex multiplicity there is
\begin {equation}
\label {eq:f14}
\frac{a_1+b_1 + c_1 + d_1}{4} \geq f_1 + \frac{2\epsilon_1 - 1}{4}. 
\end {equation}

On the other hand, apart from the destabilization point, $\x$ and $\y$
together have either $2m-1$ (if $p_1 \not \in \x$) or $2m$ (if $p_1
\in \x$) corner vertices, where the average multiplicity has to be at
least $1/4$. Together with Equations~\eqref{eq:mt} and \eqref{eq:f14}, this
implies that
 \begin {equation}
\label {eq:idf}
I(D) \geq \frac{2m-1}{4} + f_1 + \frac{2\epsilon_1 - 1}{4}.
\end {equation} 

Using the formula for the index of an enhanced domain, together with 
Equations~\eqref{eq:mt} and  \eqref{eq:idf}, we obtain
\begin {equation}
\label {equation}
I(D, \epsilon, \rho) = I(D) - \epsilon_1 - 2f_1 \geq \frac{m}{2} - f_1 - \frac{\epsilon_1}{2} - \frac{1}{2} \geq \frac{\epsilon_1-1}{2}.  
\end {equation}

Since the index is an integer, we must have $I(D, \epsilon, \rho) \geq
0$. Equality happens only when $D$ has average vertex multiplicity
$1/4$ at all its corners other than the destabilization point. An easy
analysis shows that the domain must be snail-like.
\end {proof}

\subsection {Holomorphic triangles on grid diagrams}
\label {sec:triangles}

\begin{lemma}
\label{lemma:snail}
Fix a snail-like domain $D$ (as in Figure~\ref{fig:ULeft}) for
destabilization at one point.  Then the count of holomorphic triangles
in $\Sym^n(\TT)$ with $D$ as their shadow is one mod $2$.
\end{lemma}

\begin {proof}
The trivial domain $D$ of type $L$ corresponds to a homotopy class $\phi$ whose support is a disjoint union of triangles with $90^\circ$ angles. Hence, the corresponding holomorphic count is one.

To establish the claim in general, note that the handleslide map
at one point (given by counting holomorphic triangles) is a chain map
(in fact, a quasi-isomorphism) from $C^-(G)$ to $\EC^-(G, \zed)$, compare
\eqref{eq:destab-chain-map}.  By Proposition~\ref{prop:oneO}, the only
counts involved in this map are the ones corresponding to snail-like
domains. If to each snail-like domain we assign the coefficient one
when counting it in the map, the result is the chain map
\eqref{eq:destab-chain-map}. 

We claim that, given our assignment for the trivial domain of type $L$, we do indeed need to assign one to
each snail-like domain in order for the result to be a chain map. This is proved by induction on the complexity of the domain, where complexity refers to the number of rows (or columns) in which the initial generator of the domain differs from the final generator. The base case is the trivial domain of type $L$, of complexity zero, which we know to give a count of one. 

For the inductive step, we consider domains for destabilization at one point and of index one. An example is shown in Figure~\ref{fig:StabCancel}. This domain can be decomposed as the juxtaposition $D * r$ of a type $L$ snail-like domain $D$ of index zero and complexity $3$, and a rectangle $r$. (Both $D$ and $r$ contain the marking used for destabilization; note that $r$ becomes an empty rectangle in the destabilized diagram.) The domain in Figure~\ref{fig:StabCancel} can also be decomposed along the dotted lines, as $D' * r'$, where $D'$ is a type $L$ snail-like domain of index zero and complexity $5$, and $r'$ is an empty rectangle. Further, these are the only two ways of decomposing the domain. Therefore, when we consider the (one-dimensional) moduli space of holomorphic representatives for $D * r= D' * r'$, its boundary consists of points coming from each decomposition. Suppose that, by the inductive hypothesis, the holomorphic count for $D$ is one (mod $2$). Rectangles also contribute counts of one. Since the boundary of a one-dimensional compact manifold has an even number of points, it follows that the holomorphic count for $D'$ is one as well (mod $2$). This proves the claim for the domain $D'$ of higher complexity. 

There are several other types of domains of index one to consider, e.g., some relating domains of type $L$ to those of type $R$. All possible cases are discussed in the proof of \cite[Lemma 3.5]{MOST}, where it was established that the differential given by counting snail-like domains squares to zero. The same cases can be used to establish the inductive step in our setting.
\end {proof}

\begin{figure}
\begin{center}
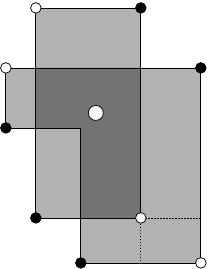
\end{center}
\caption {{\bf Counting holomorphic representatives.} We show here a domain of index $1$, which has two decompositions into a snail-like domain of index $0$ and a rectangle.}
\label{fig:StabCancel}
\end{figure}

Proposition~\ref{prop:oneO} and Lemma~\ref{lemma:snail} imply that one
can count combinatorially (mod $2$) all the index zero holomorphic
triangles in a grid diagram (with one point marked for
destabilization) with fixed shadow. Indeed, if the shadow is a
snail-like domain, then the count is one, and otherwise it is zero.

\section {Formal complex structures and surgery}
\label {sec:comps}

Let $\orL \subset S^3$ be a link with framing $\Lambda$, and let $\ux$ be a  $\spc$ structure on the surgered manifold $S^3_{\Lambda}(L)$. Our goal in this section is
to explain a combinatorial procedure for calculating the ranks of the
groups $\HFm(S^3_{\Lambda}(L), \ux)$. The procedure will be based on
Theorem~\ref{thm:Surgery}.  The algorithm is made more
complicated because Proposition~\ref{prop:oneO} and therefore
Lemma~\ref{lemma:snail} are false if there is more than one
destabilization point; we therefore have to make an appropriate choice
of domains to count, a choice we call a ``formal complex structure''.

\subsection {The complex of positive pairs}
\label{sec:complex}
Let $G$ be a grid diagram (of size $n$) marked for destabilization at a collection $\zed$ of some
free $O$'s, with $| \zed|  = k< n$, and let $q$ be the total number of free markings, as in Section~\ref{sec:uleft}. In that section we defined the
(homotopy equivalent) curved complexes $\CFLc(G)$ and $\ECFLc(G, \zed)$. Let us consider
the Hom complex
$$ \Hom_{\tR} (\CFLc(G), \ECFLc(G, \zed)),$$ 
where $\tR=\ff[[U_1, \dots, U_n, \tU_1, \dots, \tU_{n-q}]]$. Since $\CFLc(G)$ is a free $\tR$-module, we can naturally identify it with its dual using the basis given by $\S(G)$. Thus, if we view the $\Hom$ complex as an $\ff$-vector space, its generators (in the sense of direct products) are pairs $[\x, \tilde \y]$, where $\tilde \y$ is an
enhanced generator possibly multiplied by some powers of $U$, i.e.,
$$\tilde \y = U_1^{N_1} \cdots U_n^{N_n} \tU_1^{\tN_1} \dots \tU_{n-q}^{\tN_{n-q}} \cdot (\y, \epsilon).$$

The homological degree of a generator in the Hom complex is $M(\tilde
\y) - M(\x)$. However, we would like to think of the generators
as being enhanced domains (up to the addition of periodic domains), so in
order to be consistent with the formula for the index of domains we
set
$$ I([\x, \tilde \y]) = M(\x) - M(\tilde \y)$$
and view the Hom complex as a cochain complex, with a differential $d$ that increases the grading. It has the structure of an $\tR$-module, where multiplication by a variable $U_i$ increases the grading by two: $U_i [\x, \tilde \y] = [\x, U_i \tilde \y]$. Also, multiplication by a variable $\tU_i$ preserves the grading. Note that the $\Hom$ complex is bounded from below with respect to the Maslov grading. 

There is a differential on the complex given by
$$ d [\x, \tilde \y] = [\del^* \x, \tilde \y] + [\x, \del \tilde \y].$$
Thus, taking the differential of a domain consists in summing over the ways of pre- and post-composing the domain with a rectangle. Note that $d^2=0$, i.e, the Hom complex is a true complex (not a curved one). Indeed, the differentials on both $\CFLc(G)$ and $\ECFLc(G, \zed)$ square to the same multiple of the identity, as in \eqref{eq:delcurved}. These multiples cancel each other out when we compute $d^2$ for the Hom complex.

If a pair $[\x, \tilde \y]$ is positive (as in
Definition~\ref{def:PositivePair}), then by definition it represents a
positive domain; adding a rectangle to it keeps it positive. Indeed,
note that if the rectangle crosses an $O$ used for destabilization,
then the real and total multiplicities increase by $1$, but the fake
multiplicity stays the same, so the inequalities \eqref{ineq:abcd1}
are still satisfied.

Therefore, the positive pairs $[\x, \tilde \y]$ generate a subcomplex
$\CP^*(G, \zed)$ of the Hom complex. For the moment, let us ignore its
structure as an $\Ring$-module, and simply consider it as a cochain
complex over $\ff$. We denote its cohomology by $\HP^*(G, \zed)$.

We make the following:
\begin {conjecture}
\label {conj}
Let $G$ be a toroidal grid diagram of size $n$ with a collection $\zed$ of free $O$'s marked for destabilization, such that $|\zed| < n$. Then $\HP^{d}(G, \zed) = 0$ for $d < 0$.
\end {conjecture}

Proposition~\ref{prop:oneO} implies Conjecture~\ref{conj} in
the case when only one $O$ is marked for destabilization. Indeed, in
that case we have $\CP^d(G, \zed) = 0$ for $d < 0$, so the homology is
also zero.

In the case of several destabilization points, we can prove only a
weaker form of the conjecture, namely Theorem~\ref{thm:sparse}
below. However, this will suffice for our application.

\begin {definition}
  Let $G$ be a toroidal grid diagram of size $n > 1$, with a collection $\zed$ of some free $O$'s marked for
  destabilization. If none of the $O$'s marked for destabilization sit
  in adjacent rows or adjacent columns, we say that the pair $(G, \zed)$ is {\em
    sparse}. (Note that if $(G, \zed)$ is sparse, we must have $|\zed| \leq (n+1)/2 < n$.)
\end {definition}

Recall that in the Introduction we gave a similar definition, which applies to grid diagrams $G$ (with free markings) representing links in $S^3$. Precisely, we said that $G$ is {\em sparse} if none of the linked markings sit in adjacent rows or in adjacent columns. Observe that, if $G$ is sparse, then the pair $(G, \zed)$ is sparse for any collection $\zed$ of linked markings.

\begin {theorem}
\label {thm:sparse}
If $(G, \zed)$ is a sparse toroidal grid diagram with $k$ $O$'s marked for destabilization, then $\HP^d(G, \zed) = 0$ for $d < \min \{0, 2-k\}$.
\end {theorem}

The proof of Theorem~\ref{thm:sparse} will be given in Section~\ref{sec:sparse}.

\subsection {An extended complex of positive pairs}
\label {sec:ext}
Let us now return to the set-up of Sections~\ref{sec:DestabSeveral} and \ref{sec:shadows}, where we destabilize at linked markings. Thus, $\orL$ is an oriented link with a grid presentation $G$ of grid number $n$, and with $q \geq 1$ free markings. Let $\zed_0$ and $\zed$ be two disjoint sets of linked markings on $G$ such that $\zed_0 \cup \zed$ is consistent. 

There is a curved link Floer complex $\CFLc(\Hyper_G({\zed_0 \cup \zed}))$ over a base ring $\tR(\zed_0 \cup \zed)$, with one $U$ variable for each marking in $\Os^{\orL(\zed_0 \cup \zed)}$, and one $\tU$ variable for each marking in $\Xs^{\orL(\zed_0 \cup \zed)}$. There is also a curved link Floer complex $\CFLc(\Hyper_G(\zed_0))$, over a different base ring $\tR(\zed_0)$. We have a natural injective ring map
\begin{equation}
\label{eq:res} 
\tR(\zed_0 \cup \zed) \to \tR(\zed_0),
\end{equation}
defined as follows. Consider the $X$ markings on components of $\orL(\zed_0 \cup \zed) - \orL(\zed_0)$ that are oriented oppositely to the orientation in $\orL$, and thus become $O$'s in $\Os^{\orL(\zed_0 \cup \zed)}$. For those markings, we map the corresponding $U$ variables in $\tR(\zed_0 \cup \zed)$ into $\tU$ variables in $\tR(\zed_0)$. All the other $U$ and $\tU$ variables in $\tR(\zed_0 \cup \zed)$ are mapped to the same variables in $\tR(\zed_0)$.

By restriction of scalars using the map \eqref{eq:res}, we can turn $\CFLc(\Hyper_G(\zed_0))$ into a curved complex over $\tR(\zed_0 \cup \zed)$. In fact, we obtain the curved Floer complex $\CFLc(r_{\orL(\zed_0 \cup \zed) - \orL(\zed_0)}(\Hyper_G(\zed_0)))$, which we simply denote by $r_{\zed} \CFLc(\Hyper_G(\zed_0))$.

We can define a cochain complex
$$ \CP^*(G; \zed_0, \zed),$$
which is the subcomplex of 
\begin {equation}
\label {eq:Hom}
 \Hom_{\tR(\zed_0 \cup \zed)} \bigl(r_{\zed} \CFLc(\Hyper_G(\zed_0)), \CFLc(\Hyper_G({\zed_0 \cup \zed}))\bigr )
 \end {equation}
spanned by positive pairs. Here positivity of pairs has the same 
meaning as before: it is defined in terms of enhanced domains, by
looking at the grid diagram $G^{\zed_0}$ destabilized at the points of
$\zed$. 

Let $G^{\zed_0, \zed}$ be the reduction of the grid diagram $G^{\zed_0}$ at $\orL(\zed_0 \cup \zed) -\orL(\zed_0)$. (Alternatively, we can obtain $G^{\zed_0, \zed}$ from $G$ by reducing at $\orL(\zed_0 \cup \zed)$ and then deleting the rows and columns that contain markings in $\zed_0$.) 
Further, on $G^{\zed_0, \zed}$ we mark for destabilization the markings in $\zed$. Then $(G^{\zed_0, \zed}, \zed)$ is a grid diagram with some free $O$'s marked for destabilization, as in Sections~\ref{sec:uleft} and \ref{sec:complex}. As such, it has a complex of positive pairs $\CP^*(G^{\zed_0, \zed}, \zed)$, cf.\ Section~\ref{sec:complex}. This is defined over a base ring $\tR(\zed_0, \zed)$, with   one $U$ variable for each marking in $\Os^{\orL(\zed_0 \cup \zed)}$ that is not in $\zed_0$, and one $\tU$ variable for each marking in $\Xs^{\orL(\zed_0 \cup \zed)}$.

It is clear from the definitions that 
\begin {equation}
\label {eq:gze}
\CP^*(G; \zed_0, \zed) \cong \CP^*(G^{\zed_0, \zed}, \zed)[[\{U_Z\}_{Z \in \zed_0}]].
\end {equation}

Let $\zed_0, \zed, \zed'$ be disjoint sets of linked markings  on $G$, such that $\zed_0 \cup \zed \cup \zed'$ is consistent. There are natural composition maps
\begin {equation}
\label {eq:compose}
\circ : \CP^i(G; \zed_0, \zed) \otimes \CP^j(G; \zed_0 \cup \zed, \zed')  \to \CP^{i+j}(G; \zed_0, \zed \cup \zed'),
\end {equation}
obtained from the respective $\Hom$ complexes by restriction. Implicitly, in \eqref{eq:compose} we first turn the factor in
$$\CP^i(G; \zed_0, \zed)  \subset  \Hom_{\tR(\zed_0 \cup \zed)} \bigl(r_{\zed} \CFLc(\Hyper_G(\zed_0)), \CFLc(\Hyper_G({\zed_0 \cup \zed}))\bigr )$$
into  one in
$$ \Hom_{\tR(\zed_0 \cup \zed \cup \zed')} \bigl(r_{\zed \cup \zed'} \CFLc(\Hyper_G(\zed_0)), r_{\zed'}\CFLc(\Hyper_G({\zed_0 \cup \zed}))\bigr ),$$
by restriction of scalars using the natural ring map $\tR(\zed_0 \cup \zed \cup \zed') \to \tR(\zed_0 \cup \zed)$.

We define the {\em extended complex of positive pairs} associated to $G$ to be
$$ \CE^*(G) = \bigoplus_{\zed_0, \zed} \CP^*(G; \zed_0, \zed), $$
where the direct sum is over all collections $\zed_0, \zed$ such that $\zed_0 \cup \zed$ is consistent.

This breaks into a direct sum
$$ \CE^*(G) = \bigoplus_{k = 0}^{n-q} \CE^*(G; k),$$
according to the cardinality $k$ of $\zed$, i.e., the number of points
marked for destabilization.  Putting together the maps
\eqref{eq:compose}, we obtain global composition maps:
\begin {equation}
\label {eq:circ}
\circ : \CE^i(G; k) \otimes \CE^j(G; l) \to \CE^{i+j}(G; k+l),
\end {equation}
where the compositions are set to be zero when not a priori
well-defined on the respective summands. These composition maps
satisfy a Leibniz rule for the differential.

The complex $\CE^*(G)$ was mentioned in the Introduction. There we stated Conjecture~\ref{conj:extended}, which says that for any toroidal grid diagram $G$, we have
$$ \HE^d(G) = 0 \text{ when } d <0.$$

Observe that Conjecture~\ref{conj:extended} would be a direct consequence of Conjecture~\ref{conj}, because of Equation~\eqref{eq:gze}.

We prove a weaker version of Conjecture~\ref{conj:extended}, which applies only to sparse grid diagrams (as defined in the Introduction).

\begin {theorem}
\label {thm:esparse}
Let $G$ be a sparse toroidal grid diagram representing a link $L$. Then $\HE^d(G; k) = 0$ whenever $d < \min \{0, 2-k\}$.
\end {theorem}

\begin {proof} This follows from Theorem~\ref{thm:sparse}, using
  Equation~\eqref{eq:gze}. The key observation is that all
  destabilizations of a sparse diagram at linked markings are also sparse.
\end {proof}

\subsection{Formal complex structures}

Let $G$ be a grid presentation for the link $L$, such that $G$ has $q \geq 1$ free markings. 
Our goal is to find an algorithm for computing $\HFm_{*}(S^3_\Lambda(L), \ux)$, where $\Lambda$ is a framing of $L$ and  $\ux$ a  $\spc$ structure on $S^3_{\Lambda}(L)$. We will first describe how to do so assuming that Conjecture~\ref{conj:extended} is true, and using Theorem~\ref{thm:Surgery}. 

By Theorem~\ref{thm:Surgery}, we need to compute the homology of the
complex $\C^-(G, \Lambda, \ux) \subseteq \C^-(G, \Lambda)$, with its differential $\D^-$. In the definition
of $\D^-$ we use the maps $\Phi^{\orM}_\s$, which in turn are based on
descent maps $\hat \De^{\orM}_{p^{\orM}(\s)}$ of the kind constructed in
Sections~\ref{sec:descent2}-\ref{sec:descent3}. In turn, the maps $\hat \De^{\orM}_{p^{\orM}(\s)}$ are obtained by compression of a hyperbox, and thus are sums of compositions of certain polygon maps of the form 
\begin{equation}
\label{eq:dezz}
 D_{\eps}^{\eps'-\eps}: C^{\eps} \to C^{\eps'},
 \end{equation}
in the notation of Section~\ref{sec:descent3}.   

We seek to understand the maps $D_{\eps}^{\eps'-\eps}$ combinatorially. Note that they involve counting $(k+2)$-gons of index $1-k$ in $\Sym^n(\TT)$, for $k \geq 0$.

In the case $k=0$, we know that holomorphic bigons are the same as empty rectangles on $\TT$,
cf.~\cite{MOS}. For $k=1$, one can still count holomorphic triangles
explicitly, cf.\ Section~\ref{sec:triangles}.

Unfortunately, for $k \geq 2$, the count of holomorphic $(k+2)$-gons
seems to depend on the almost complex structure on $\Sym^n(\TT)$. The
best we can hope for is not to calculate the maps $D_{\eps}^{\eps'-\eps}$
explicitly, but to calculate them up to chain homotopy. In turn, this will
give an algorithm for computing the chain complex $\C^-(G, \Lambda, \ux)$ up
to chain homotopy equivalence, and this is enough for knowing its homology.

Recall that a complex structure $j$ on the torus $\TT$ gives rise to a
complex structure $\Sym^n(j)$ on the symmetric product
$\Sym^n(\TT)$. In \cite{HolDisk}, in order to define Floer homology
the authors used a certain class of perturbations of $\Sym^n(j)$,
which are (time dependent) almost complex structures on $\Sym^n(\TT)$.
For each almost complex structure $J$ in this class, one can count
$J$-holomorphic polygons for various maps. 
Specifically, for any disjoint sets $\zed_0, \zed$ of linked markings on $G$ such that $\zed_0 \cup \zed$ is consistent, and for any 
$$ \x \in \Ta \cap \Tb^{\zed_0}, \ \ \ \y \in \Ta \cap \Tb^{\zed_0 \cup \zed}, \ \tilde \y = U_1^{N_1} \cdots U_n^{N_n} \tU_1^{\tN_1} \dots \tU_{n-q}^{\tN_{n-q}} \cdot (\y,\epsilon)$$
with
\begin{equation}
\label{eq:Ni}
 \tN_i = 0 \ \text{ when } \ X_i \in L(\zed_0 \cup \zed),
 \end{equation}
we denote by
$$ n_J^{\zed_0, \zed}(\x, \tilde \y) \in \ff $$
the count of
$J$-holomorphic $(k+2)$-gons between $\x$ and $\tilde \y$, in all
possible homotopy classes $\phi$ with $\mu(\x, \y) =\mu(\phi) = 1-k$,
and coming from all possible orderings of the elements of
$\zed$. Here, the exponents $N_i$ and $\tN_i$ keep track of the multiplicities of $\phi$ at the basepoints in $\Os^{\orL(\zed_0 \cup \zed)}$ and $\Xs^{\orL(\zed_0 \cup \zed)}$. The condition \eqref{eq:Ni} comes from the fact that $\Xs^{\orL(\zed_0 \cup \zed)}$ contains exactly those $X$'s that are not in $L(\zed_0 \cup \zed)$.

According to \eqref{eq:bigformula}, the map $ D_{\eps}^{\eps'-\eps}$ is constructed by counting holomorphic polygons in multi-diagrams where we use $\Os^{\orM-M_{\eps}}$, $\Xs^{\orM-M_{\eps}}$ for basepoints, and
$$ \zed_0 = \zed(\orM)^{\eps^>}, \ \ \ \zed_0 \cup \zed = \zed(\orM)^{(\eps^>)'}.$$
Thus, the values $n_J^{\zed_0, \zed}(\x, \tilde \y)$ determine $ D_{\eps}^{\eps'-\eps}$. Indeed, we can write
$$ D_{\eps}^{\eps'-\eps}( \x) = \sum_{\tilde \y} n_J^{\zed_0, \zed}(\x, \tilde \y) \cdot  \langle \tilde \y \rangle, $$
where $ \langle \tilde \y \rangle$ is obtained from $\tilde \y$ by changing the coefficient of $(\y, \epsilon)$ to the expression \eqref{eq:multiU}. Although that expression is complicated, it is based on the multiplicities of $\phi$ at the markings in $\Os^{\orM-M_{\eps}} \cup \Xs^{\orM-M_{\eps}}$, which are a subset of the markings in $\Os^{\orL(\zed_0 \cup \zed)} \cup \Xs^{\orL(\zed_0 \cup \zed)}$. These multiplicities are part of the information in the specific $\tilde \y$.

We are left to study the values $n_J^{\zed_0, \zed}(\x, \tilde \y)$.

Observe that, in order for $n_J^{\zed_0, \zed}(\x, \tilde \y)$ to be
nonzero, the pair $[\x, \tilde \y]$ has to be positive. Hence, the set
of values $n_J^{\zed_0, \zed}(\x, \tilde \y)$ produces well-defined
elements in the extended complex of positive pairs on $G$:
$$ c_k(J) = \sum n_J^{\zed_0, \zed}(\x, \tilde \y) \cdot  [\x, \tilde \y] \in \CP^{1-k}(G; \zed_0, \zed) \subseteq \CE^{1-k}(G; k), \ \ k \geq 1.$$

Lemma~\ref{lemma:d2} implies that the elements $c_k(J)$ satisfy the
following compatibility conditions, with respect to the composition
product~\eqref{eq:circ}:
$$ dc_k(J) = \sum_{i=1}^{k-1} c_i(J) \circ c_{k-i}(J).$$

In particular, $dc_1(J) = 0$. Note that $c_1(J)$ is given by the count of snail-like domains, and therefore is independent of $J$. We denote it by 
$$c_1^{\snail} \in \CE^0(G; 1).$$

\begin {definition}
\label{def:FormalComplexStructure}
A {\em formal complex structure} $\cx$ on the grid diagram $G$ (of grid number $n$, and with $q \geq 1$ free markings) consists of a family of elements 
$$c_k \in  \CE^{1-k}(G; k),\quad k = 1, \dots, n-q,$$
satisfying $c_1 = c_1^{\snail}$ and the compatibility conditions:
\begin {equation}
 \label {eq:compat}
 dc_k = \sum_{i=1}^{k-1} c_i \circ c_{k-i}.
 \end {equation}
\end {definition}

In particular, an (admissible) almost complex structure $J$ on $\Sym^n(\TT)$ induces a formal complex structure $\cx(J)$ on $G$.

\begin {remark}
If we let $\cx=(c_1, c_2, \dots) \in \CE^*(G)$, the relation~\eqref{eq:compat} 
is summarized by the equation
\begin {equation}
\label {eq:master}
 d\cx = \cx \circ \cx.
 \end {equation}
\end {remark}

\begin {definition}
\label {def:homot}
  Two formal complex structures $\cx = (c_1, c_2, \dots),
  \cx'=(c_1', c_2', \dots)$ on a grid diagram $G$ (of grid number $n$, with $q \geq 1$ free markings) are called {\em
    homotopic} if there exists a sequence of elements
$$ h_k \in \CE^{-k}(G; k),\quad k = 1, \dots, n-q $$
satisfying $h_1 = 0$ and 
\begin {equation}
\label {eq:homs}
c_k - c_k' = dh_k + \sum_{i=1}^{k-1} \bigl( c'_i \circ h_{k-i} + h_i \circ c_{k-i} \bigr ).
\end {equation}
\end {definition}

Observe that, if $J$ and $J'$ are (admissible) almost complex
structures on $\Sym^n(\TT)$, one can interpolate between them by a
family of almost complex structures. The resulting counts of
holomorphic $(2+k)$-gons of index $-k$ induce a homotopy between
$\cx(J)$ and $\cx(J')$.  There is therefore a canonical homotopy class
of formal complex structures that come from actual almost complex
structures.

\begin {lemma}
\label {lemma:1k}
Assume $\HE^{1-k}(G; k) = 0$ for any $k = 2, \dots, n-q$. Then any two formal complex structures on $G$ are homotopic.
\end {lemma} 

\begin {proof}
Let $\cx = (c_1, c_2, \dots), \cx'=(c_1', c_2', \dots)$ be two formal complex structures on $G$. We prove the existence of the elements $h_k$ by induction on $k$. When $k=1$, we have $c_1 = c_1' = c_1^{\snail}$ so we can take $h_1 =0$.  

Assume we have constructed $h_i$ for $i < k$ satisfying~\eqref{eq:homs},
and we need $h_k$. Since by hypothesis the cohomology group $\CE(G;
k)$ is zero in
degree $1-k$, it suffices to show that
\begin{equation}
  \label{eq:ck-cycle}
  c_k - c_k' - \sum_{i=1}^{k-1} \bigl( c'_i \circ h_{k-i} + h_i \circ c_{k-i} \bigr )
\end{equation}
is a cocycle. Indeed, we have
\begin {multline*}
d \Bigl( c_k - c_k' - \sum_{i=1}^{k-1} \bigl( c'_i \circ h_{k-i} + h_i
\circ c_{k-i} \bigr ) \Bigr ) \\
\begin{aligned}
&= \sum_{i=1}^{k-1} c_i \circ c_{k-i} - \sum_{i=1}^{k-1} c'_i \circ c'_{k-i} - \sum_{i=1}^{k-1} \bigl( dc'_i \circ h_{k-i} + c'_i \circ dh_{k-i} + dh_i \circ c_{k-i} + h_i \circ dc_{k-i} \bigr) \\
&= 
\sum_{i=1}^{k-1} (c_i - c_i' - dh_i) \circ c_{k-i} + \sum_{i=1}^{k-1} c_i' \circ (c_{k-i} - c'_{k-i} -dh_{k-i}) - \sum_{i=1}^{k-1} dc'_i \circ h_{k-i} - \sum_{i=1}^{k-1} h_i \circ dc_{k-i} \\
&= \sum(c_{\alpha}'h_{\beta} c_{\gamma} + h_{\alpha}c_{\beta}
c_{\gamma}) +  \sum(c_{\alpha}'h_{\beta} c_{\gamma} +
c'_{\alpha}c'_{\beta} h_{\gamma}) - \sum  c'_{\alpha}c'_{\beta}
h_{\gamma} - \sum  h_{\alpha}c_{\beta} c_{\gamma}\\
&= 0.
\end{aligned}
\end {multline*}
In the second-to-last line the summations are over $\alpha, \beta, \gamma \geq
1$ with $\alpha + \beta + \gamma = k$, and we suppressed the
composition symbols for simplicity.
\end {proof}

\subsection {Combinatorial descriptions}
Consider a formal complex structure $\cx$ on a grid $G$ (of grid number $n$, with $q \geq 1$ free markings),  a framing
$\Lambda$ for $L$, and an equivalence class $\ux  \in \bigl(\H(L)/H(L, \Lambda)\bigr)$ $\cong \spc(S^3_\Lambda(L))$. We seek to define a complex $\C^-(G, \Lambda, \ux, \cx)$ analogous to the complex $\C^-(G, \Lambda, \ux)$ from Section~\ref{sec:surgery}, but defined using the elements $c_k$ instead
of the holomorphic polygon counts. (In particular, if $\cx = \cx(J)$
for an actual almost complex structure $J$, we want to recover the
complex $\C^-(G, \Lambda, \ux)$ from Section~\ref{sec:surgery}.)

Let us explain the construction of $\C^-(G, \Lambda, \ux,
\cx)$. Recall that in Section~\ref{sec:surgery} the complex $\C^-(G,
\Lambda, \ux)$ was built from projection-inclusion maps, descent maps, and isomorphisms. The descent maps are obtained by compression of hyperboxes that involve the polygon maps $ D_{\eps}^{\eps'-\eps}$. We can define analogous maps $ D_{\eps}^{\eps'-\eps}[\cx]$, by counting $(k+2)$-gons according to the coefficients of enhanced domains that appear in $\cx_k$. A homotopy between $\cx$ and $\cx'$ gives homotopy equivalences between the hyperboxes based on $ D_{\eps}^{\eps'-\eps}[\cx]$ and $ D_{\eps}^{\eps'-\eps}[\cx']$. 

Taking into account the naturality properties of compression discussed at the end of Section~\ref{sec:hyper}, we obtain the following lemma:

\begin {lemma}
\label {lemma:xx}
A homotopy between formal complex structures $\cx, \cx'$ on $G$
induces a chain homotopy equivalence between the complexes $\C^-(G, \Lambda, \ux, \cx)$
and $\C^-(G, \Lambda, \ux, \cx')$.
\end {lemma}

With this in mind, we are ready to prove the two theorems advertised in the Introduction.

\begin{proof}[Proof of Theorem~\ref{thm:Three}]
The algorithm to compute $\HFm$ of an integral surgery on a link $\orL$
goes as follows. First, choose a sparse grid diagram $G$ for $\orL$ (for example, by taking the sparse double of an ordinary grid diagram, as in Figure~\ref{fig:sparse}). Then, choose any formal complex structure on $G$, construct the complex
$\C^-(G, \Lambda, \ux, \cx)$, and take its homology. 

Let us explain why this algorithm is finite and gives the desired answer. Observe that
$\CE^*(G)$ is finite in each degree, and the number $k$ of destabilization
points is bounded above by $n-q$, so the direct sum $\oplus_{k \geq 1} \CE^{1-k}(G; k)$ is a finite set. Further, we know that a 
formal complex structure exists, because it could be induced by some
almost complex structure $J$. Thus, we can find a formal complex structure $\cx$ on $G$ by a brute
force approach: go over all the (necessarily finite) sequences $ \cx = (c_1=c_1^{\snail}, c_2, c_3, \dots)  \in \oplus_{k \geq 1} \CE^{1-k}(G; k)$, and pick the first one that satisfies Equation~\eqref{eq:master}. Then, by Theorem~\ref{thm:esparse} and Lemma~\ref{lemma:1k} we know that all
possible $\cx$'s are homotopic. Lemma~\ref{lemma:xx}, together
with Theorem~\ref{thm:Surgery}, tells us that the homology of $\C^-(G, \Lambda, \ux, \cx)$ is indeed the right answer.

(We could alternately take a somewhat more efficient, step-by-step approach to finding $\cx$:
start with $c_1=c_1^{\snail}$ and inductively find each $c_k$ for $k
\ge 2$.  Since Formula~\eqref{eq:ck-cycle} represents a cycle, the
obstruction to extending a partial formal complex structure to the next step
vanishes.)

Note that although this description is combinatorial in nature, the complex $\C^-(G, \Lambda, \ux, 
\cx)$ is still infinite dimensional, being an infinite direct product of modules over a ring of power series. However, we can replace it by a quasi-isomorphic, finite dimensional complex over $\ff=\zz/2\zz$ using the vertical and horizontal truncation procedures from \cite[Section 10]{LinkSurg}. Taking the homology of the truncated complex is clearly an algorithmic task.

One can calculate the other versions of Heegaard Floer homology ($\HFinf, \iHF^+$) in a similar way, using Theorem~\ref{thm:AllVersions}. 
\end {proof}

\begin{proof}[Proof of Theorem~\ref{thm:Four}]
One can calculate the maps induced by two-handle additions using Theorems~\ref{thm:Cobordisms} and \ref{thm:AllVersions}. Furthermore, we can calculate the mixed invariants of closed four-manifolds using Proposition~\ref{prop:mixed}. In all cases, one proceeds by choosing an arbitrary formal complex structure $\cx$ on the grid diagram $G$, and computing the respective groups or maps using polygon counts prescribed by $\cx$. 
\end {proof}

We also have the following:
 
\begin {theorem}
\label {thm:SpectralComb}
Fix a sparse grid diagram $G$ for an oriented link $\orL' \cup \orL$ in $S^3$. Fix also framings $\Lambda$ for $L$ and $\Lambda'$ for $L'$.  Suppose that $L$ has $\ell$ components $L_1, \dots, L_\ell$. Let $Y(0,\dots,0) = S^3_{\Lambda'}(L')$, and, for any $\eps \in \E_\ell$, let $Y(\eps)$ be obtained from $Y(0,\dots,0)$ by surgery on the components $L_i \subseteq L$ with $\eps_i = 1$. Then, all the pages of the link surgeries spectral sequence from Theorem~\ref{thm:OSspectral} (with $E^1 = \oplus_{\eps \in \E_\ell} \CFm(Y(\eps))$ and coefficients in $\ff =\zz/2\zz$) are algorithmically computable.
\end {theorem} 
 
\begin {proof}
Use the equivalent description of the spectral sequence given in Theorem~\ref{thm:SpectralSequence}. Choose a formal complex structure $\cx$ on the grid diagram $G$, and construct a complex  $\C^-(G, \Lambda' \cup \Lambda, \cx \hey L)$ analogous to $\C^-(G, \Lambda' \cup \Lambda \hey L)$, but using the polygon counts given by $\cx$. Then compute the spectral sequence associated to the depth filtration on $\C^-(G, \Lambda' \cup \Lambda, \cx \hey L)$.   
\end {proof}

\begin {remark}
Suppose a link $L$ has grid number $m$, that is, $m$ is the lowest number such that $G$ admits a grid presentation of that size. Our algorithms above are based on a sparse grid diagram for $L$, and such a diagram must have grid number at least $2m$. If Conjecture~\ref{conj:extended} were true, we would obtain more efficient algorithms, because we could start with a diagram of grid number only $m+1$ (by adding only one free marking to the minimal grid).
\end {remark}

\begin {remark}
Our present techniques do not give a combinatorial procedure for
finding the map $F^-_{W,\t}$ associated to an arbitrary cobordism map  (in a given degree)
or even its rank. However, suppose $W$ is a cobordism between
connected three-manifolds $Y_1$ and $Y_2$ such that the induced maps
$H_1(Y_1; \zz)/\Tors \to H_1(W; \zz)/\Tors$ and $H_1(Y_2; \zz)/\Tors
\to H_1(W; \zz)/\Tors$ are surjective. Then the ranks of $F^-_{W,
  \tt}$ in fixed degrees can be computed using the same arguments as
in \cite[Section 4]{LMW}. Indeed, they are determined by the ranks of
the map induced by the two-handle additions which are part of the
cobordism $W$.
\end {remark}

\section{Sparse grid diagrams}
\label {sec:sparse}

This section is devoted to the proof of Theorem~\ref{thm:sparse}. Let
$(G, \zed)$ be a sparse toroidal grid diagram with some free $O$'s marked for
destabilization, as in Section~\ref{sec:uleft}. We define a filtration
$\F$ on $\CP^*(G, \zed)$ as follows. Let us mark one $T$ in each square of
the grid with the property that neither its row nor its column
contains an $O$ marked for destabilization.  See
Figure~\ref{fig:ys}, where the squares
marked by a $T$ are shown shaded.

Given a pair $[\x, \tilde \y]$, let $(D, \epsilon, \rho)$ be an
enhanced domain from $\x$ to $\tilde \y$. Let $T(D)$ be the number of
$T$'s inside $D$, counted with multiplicity. Define
\begin {equation}
\label {eq:fex}
\F([\x, \tilde \y]) = -T(D) - \sum_{j=1}^k \rho_{j}.
\end {equation}

It is easy to see that $T(D)$ does not change under the notion of equivalence from Definition~\ref{def:EquivEnhanced}. The same is true for the real multiplicities
$\rho_{j}$. Therefore, in view of Lemma~\ref{lem:EquivEnhanced}, the value $\F([\x, \tilde \y])$  is well-defined. Furthermore, pre- or post-composing
with a rectangle can only decrease $\F$, so $\F$ is indeed a
filtration on $\CP^*(G, \zed)$.

To show that $\HP^d(G, \zed) = 0 $ (for $d > 0$) it suffices to check that
the homology of the associated graded groups to $\F$ are zero. In the
associated graded complex $\grcp$, the differential only involves
composing with rectangles $r$ such that $r$ is supported in a row or
column going through some $O_{i_j} \in \zed$ marked for destabilization, but
$r$ does not contain $O_{i_j}$. We cannot post-compose with such a
rectangle, because it would move the destabilization corners in $\y$,
and that is not allowed. Thus, the differential of $\grcp$ only
involves pre-composing with rectangles as above.

For each positive pair $\p = [\x, \tilde \y]$, we let $C\F^*(G,\zed, \p)$ be
the subcomplex of $\grcp$ generated by $\p$ and those pairs related to
$\p$ by a sequence of nonzero differentials in $\grcp$. More precisely, for
each complex over $\ff$ freely generated by a set $S$, we can
form an associated graph whose set of vertices is $S$ and
with an edge from $\x$ to $\y$ ($\x, \y\in S$) whenever the
coefficient of $\y$ in $d\x$ is one.  Then the graph of $C\F^*(G,\zed,\p)$
is the connected component containing~$\p$ of the graph of $\grcp$
(with respect to the standard basis).

\begin {definition}
  Let $(G, \zed)$ be a toroidal grid diagram with $\zed= \{O_{i_j}| j=1, \dots, k\}$ marked for
  destabilization.  The four corners of the square
  containing $O_{i_j}$ are called \emph{inner corners} at $O_{i_j}$.
  An element $\x \in \S(G)$ is
  called {\em inner} if, for each $j=1, \dots, k$, at least one of the
  inner corners at $O_{i_j}$ is part of $\x$.
  The element $\x$ is called {\em outer} otherwise.
\end {definition}

\begin {lemma}
\label {lemma:inner}
Let $(G, \zed)$ be a sparse toroidal grid diagram with some free $O$'s marked for
destabilization, where $|\zed| = k \geq 2$. Let $\p = [\x, \tilde \y]$ be a
positive pair such that $\x$ is inner. Then the index $I(\p)$ is at
least $2-k$.
\end {lemma}

\begin {proof}
Let $(D, \epsilon, \rho)$ be a positive enhanced domain going between
$\x$ and $\tilde \y$. Then
$$ I(\p) = I(D, \epsilon, \rho) = I(D) - \sum_{j=1}^k (\epsilon_j + 2f_j)$$
by Equation~(\ref{eq:enhanced-index}).

Just as in the proof of Proposition~\ref{prop:oneO}, without loss of
generality we can assume that every row or column contains at least
one square where the multiplicity of $D$ is zero; hence, as we move
to an adjacent row (or column), the multiplicity of the domain~$D$ can only
change by $0$ or $\pm 1$.

According to Equation~\eqref{eq:lipshitz}, the usual index $I(D)$ is given by the sum of
the average multiplicities of $D$ at the corners. For each $j =1,
\dots, k$, define $n_j$ to be the sum of the average
multiplicities at the corners situated on one of the following four
lines: the two vertical lines bordering the column of $O_{i_j}$ and
the two horizontal lines bordering the row of $O_{i_j}$; with the
caveat that, if such a corner $c$ (with average multiplicity $a$
around it) appears on one of the four lines for $O_{i_j}$ and also on
one of the four lines for $O_{i_l}$, $l \neq j$, then we let $c$
contribute $a/2$ to $n_j$ and $a/2$ to $n_l$. For an example, see
Figure~\ref{fig:inner}. Note that, since the diagram is sparse, the
average multiplicities at inner corners are only counted in one~$n_l$.

\begin{figure}
\begin{center}
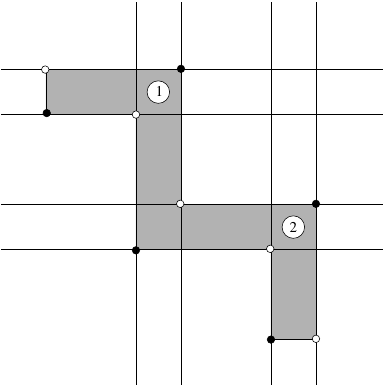
\end{center}
\caption {{\bf An example of a positive enhanced domain with $\x$ inner.} We
  have $k=2$ and the two $O$'s marked for destabilization ($O_{i_j},
  j=1,2$) are shown in the figure with the value of $j$ written
  inside. We view this as a domain of type $LL$, meaning that
  $\epsilon_1 = \epsilon_2 = 0$, and with the real multiplicities at the
  destabilization points equal to zero, so that $f_1 = f_2 = 1$.  The
  domain has index $0$, and the quantities $n_1$ and $n_2$ both equal
  $3\cdot (1/4) + (3/4) + (1/4 + 3/4)/2 = 2$. }
\label{fig:inner}
\end{figure}

We get
\begin {equation}
\label {eq:index1}
I(\p) \geq \sum_{j=1}^k (n_j - \epsilon_j - 2f_j).
\end {equation}

We will prove that
\begin{equation}
  \label{eq:nj-bound}
  n_j \geq 2f_j + \epsilon_j/2.
\end{equation}
Indeed, the relations \eqref{ineq:abcd1} imply that the average multiplicity at the destabilization point $p_j$ for $O_{i_j}$ (which is part of $\y$) is
$$\frac{a_j+b_j + c_j + d_j}{4} \geq   f_j + \frac{2\epsilon_j-1}{4}. $$
Since $\x$ is inner, there is also one point of $\x$, call it $x$, in
a corner of the square containing $O_{i_j}$. There are four cases,
according to the position of $x$.  We first consider the case when the
marking at $O_{i_j}$ is $L$ (i.e., $\epsilon_j = 0$).

\begin{enumerate}
\item 
If $x$ is the lower left corner (which is the same as the
destabilization point), we have $a_j + d_j = b_j +c_j \geq 2f_j$
there, so the average multiplicity at $x$ is at least $f_j$; since the
corner counts in both $\x$ and $\y$, we get $n_j \geq 2f_j$, as
desired.
\item If $x$ is the lower right corner, since $b_j, d_j \geq f_j$ and
  the multiplicity of $D$ can change by at most $\pm 1$ as we pass a
  column, we find that the average multiplicity at $x$ is at least
  $f_j - 1/4$. Together with the contribution from $p_j \in \y$, this
  adds up to $2f_j - 1/2$. Note that  \eqref{eq:abcd1} says that $a_j + d_j = b_j + c_j +1$, so the contribution from $p_j$ differs from $(b_j + d_j)/2 + 1/4$ by a multiple of $1/2$. Similarly, the contribution from $x$ differs from $(b_j + d_j)/2 + 1/4$ by a multiple of $1/2$. Thus, the sum of the contributions from $p_j$ and $x$ is a multiple of $1/2$. To see that this sum is at least $2f_j$, it suffices to argue that it cannot be $2f_j - 1/2$, i.e., the two contributions cannot be both exactly $f_j-1/4$. Suppose this were the case. Then if $r$ is a square other than
  $O_j$ in the row through $O_j$, and $s$ is the square directly below
  it, the local multiplicity at $r$ is one greater than it is at
  $s$.  This contradicts the positivity assumption combined with the
  assumption that there is at least one square with multiplicity~$0$
  in the row containing $O_j$.
 \item
The case when $x$ is the upper left corner is similar to lower right,
with the roles of the the row and the column through $O_{i_j}$
swapped.
\item
Finally, if $x$ is the upper right corner, then the average
multiplicity there is at least $f_j - 3/4$. Together with the
contribution from $p_j \in \y$, we get a contribution of at least
$2f_j - 1$. There are
two remaining corners on the vertical lines through $p_j$ and $x$; we
call them $c_1 \in \x$ and $c_2 \in \y$, respectively. We claim that
the contributions of the average multiplicities of $c_1$ and $c_2$ to
$n_j$ sum up to at least $1/2$. Indeed, if at least one of these
average multiplicities is $\geq 3/4$, their sum is $\geq 1$, which
might be halved (because the contribution may be split with another $n_l$)
to get at least $1/2$. If both of the average multiplicities
are $1/4$ (i.e., both $c_1$ and $c_2$ are $90^\circ$ corners), they
must lie on the same horizontal line, and therefore their
contributions are not shared with any of the other $n_l$'s; so they
still add up to $1/2$. A similar argument gives an additional
contribution of at least $1/2$ from the two remaining corners on the
row through $O_{i_j}$. Adding it all up, we get $n_j \geq (2f_j - 1) +
1/2 + 1/2 = 2f_j$.
\end{enumerate}

This completes the proof of Equation~\eqref{eq:nj-bound} when $D$ is
of type~$L$ at $O_{i_j}$.  When $D$ is of type~$R$ there (i.e.,
$\epsilon_j = 1$),
the contribution of~$x$ to $n_j$
is at least $f_j - \frac{3}{4}$. Studying the four possible positions
of $x$, just as in the $L$ case, gives additional contributions to
$n_j$ of at least $1$, which proves Equation~\eqref{eq:nj-bound}. In fact,
the contributions are typically strictly greater than $1$; the only
situation in which we can have equality in~\eqref{eq:nj-bound} when
$D$ is of type~$R$ is when $x$ is
the upper right corner and the local multiplicities around $x$ and
$p_j$ are exactly as in Figure~\ref{fig:equality}.

\begin{figure}
\begin{center}
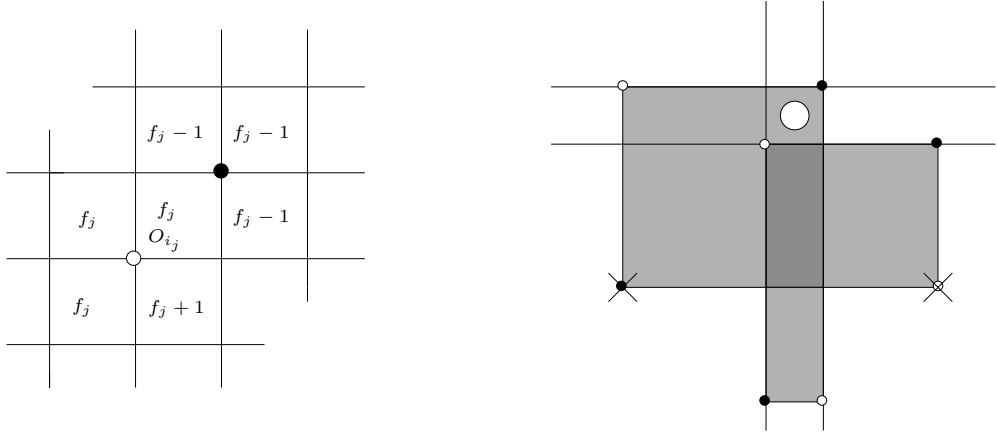
\end{center}
\caption {{\bf Equality in Equation~\eqref{eq:nj-bound}.} On the left, we
  show the local multiplicities around $O_{i_j}$ in the case
  $\epsilon_j =1,n_j = 2f_j + 1/2$. (Note that the real multiplicity
  at $O_{i_j}$ is then zero.) On the right we picture a domain of this
  type, with $k=1$, $f_1 = 1$. There are two corners, marked by $\times$, whose contributions are not counted in $n_1$. As a consequence, the Inequality~\eqref{eq:index1} is strict.}
\label{fig:equality}
\end{figure}

Putting Equation~\eqref{eq:nj-bound} together with Inequality~\eqref{eq:index1},
we obtain
\begin {equation}
\label {eq:index2}
 I(\p) \geq \sum_{j=1}^k (-\epsilon_j/2) \geq -k/2.
\end {equation}

Our goal was to show that $I(\p) \geq 2-k$. This follows directly
from~\eqref{eq:index2} for $k \geq 4$; it also follows when $k = 3$,
by observing that $I(\p)$ is an integer. 

The only remaining case is $k=2$, when we want to show $I(\p) \geq 0$,
but Inequality~\eqref{eq:index2} only gives $I(\p) \geq -1$. However, if $I(\p) =
-1$ we would have equality in all inequalities that were used to
arrive at Inequality~\eqref{eq:index2}. In particular, both destabilizations are
of type $R$, the corresponding $x$'s are the upper right corners of
the respective squares, and the local multiplicities there are as in
Figure~\ref{fig:equality}. Observe that as we move down from the row
above $O_{i_j}$ to the row containing $O_{i_j}$, the local
multiplicity cannot decrease. The same is true as we move down from
the row containing $O_{i_j}$ to the one just below. On the other hand,
looking at the column of $O_{i_1}$, as we go down around the grid from
the square just below $O_{i_1}$ (where the multiplicity is $f_1+1$) to
the one just above $O_{i_1}$ (where the multiplicity is $f_j-1$), we
must encounter at least two horizontal circles where the multiplicity
decreases. By our observation above, neither of these circles can be
one of the two that bound the row through $O_{i_2}$. However, one or
two of them could be the circles of the two remaining corners on the
column through $O_{i_1}$. These corners only contribute to $n_1$, not
to $n_2$, and since we had equality when we counted their contribution
to be at least $1/2$, it must be the case that each of them is a
$90^{\circ}$ corner, with a contribution of $1/4$. This means that
they must lie on the same horizontal circle. Hence, there must be one 
other horizontal circle along which local multiplicities of our
domain decrease as we cross it from above. On this circle
there are some additional corners, with a nontrivial contribution to
$I(D)$ unaccounted for; compare the right hand side of
Figure~\ref{fig:equality}. (Note however that in that figure, we consider
$k=1$, rather than $k=2$.) These vertices contribute extra to the
vertex multiplicity, which means that $I(\p) > -1$.
\end {proof}

\begin {remark}
It may be possible to improve the inequality~\eqref{eq:index2} to
$I(\p) \geq 0$ for every $k > 0$ along the same lines, by doing a more
careful analysis of the contributions to $I(D)$.
\end {remark}

\begin {lemma}
\label {lemma:outer}
Let $(G, \zed)$ be a sparse toroidal grid diagram, and $\p = [\x, \tilde \y]$ a positive pair such that $\x$ is outer. Then the cohomology of $C\F^*(G, \zed, \p)$ is zero.
\end {lemma}

\begin {proof}
  First, observe that, since the differential of $\grcp$ only involves
  pre-compositions, $\tilde \y$ is the same for all generators of
  $C\F^*(G, \zed, \p)$. Furthermore, $\x$ being outer means that there is some destabilization point $O_{i_j}$ such that $\x$ does not contain any of the inner corners at $O_{i_j}$. This property is preserved when we pre-compose with rectangles supported in the rows or columns containing the destabilization points. It follows that all the other
  generators $[\x', \tilde \y]$ of $C\F^*(G, \zed, \p)$ have the property that $\x'$ is outer.

  Now let $j$ be such that no corner of $O_{i_j}$ is in $\x$. Consider a
  new filtration $\G$ on $C\F^*(G, \zed, \p)$ given as follows. Let us mark
  one $Y$ in each square of the grid that lies in a column or row
  through an $O$ marked for destabilization, but does not lie in the
  column going though $O_{i_j}$ (where $j$ was chosen above), nor does
  it lie in the same square as one of the other $O_{i_l}$'s. Further,
  we mark $n-1$ copies of $Y$ in the square directly below the square of
  $O_{i_j}$. (Here $n$ is the size of the grid.) Finally, we mark one
  extra $Y$ in the square directly to the left of each $O_{i_l}$, $l \neq
  j$.  Observe that for every
  periodic domain equal to the row through some $O_{i_l}$ minus the
  column through $O_{i_l}$ (for some $l=1, \dots, k$, including $l=j$),
  the signed count of $Y$'s in that domain is zero.

Consider now the squares  in $G$ that do not lie in any row or column
that goes through an $O$ marked for destabilization. We denote them by
$s_{u,v}$, with $u, v \in \{1, \dots, n-k\}$, where the two indices
$u$ and $v$ keep track of the (renumbered) row and column,
respectively. Note that all these squares are already marked by an
$T$, used to define the filtration $\F$. We will additionally mark them
with several $Y$'s, where the exact number $\alpha_{i,j}$ of $Y$'s in
$s_{i,j}$ is to be specified soon. See Figure~\ref{fig:ys} for an example. 

\begin{figure}
\begin{center}
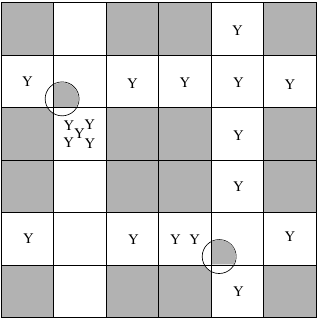
\end{center}
\caption {{\bf The markings defining the filtrations.} We show here a
  grid diagram of grid number $6$, with two $O$'s marked for
  destabilization, namely $O_{i_1}=O_2$ and $O_{i_2}=O_5$. We draw small circles surrounding each of the two destabilization points. Each shaded
  square contains a $T$ and this defines the filtration $\F$. One could also imagine a $T$ marking in the same position as an $O$ marking, to account for the terms $\rho_j$ in Equation~\eqref{eq:fex}.  Given an outer generator $\x$, we need to choose some $j \in \{1,2\}$ such that no corner $O_{i_j}$ is in $\x$. In this figure, $j=1$. We then define a second filtration $\G$ on the components
  of the associated graded of $\F$, using the $Y$ markings as shown, plus five (invisible) $Y$ markings in the lower left quadrants of the small disks (bounded by beta circles) around each $O$ marked for destabilization. We first choose the $Y$ markings in the unshaded squares, then mark the shaded squares so that every periodic domain has a total of zero markings, counted with multiplicities.}
\label{fig:ys}
\end{figure}

Given one of the generators $\p' = [\x', \tilde \y]$ of $C\F^*(G, \zed, \p)$, 
choose an enhanced domain $E=(D, \epsilon, \rho)$ from $\x'$ to $\y$, and let $Y(D)$ be the number of $Y$'s inside $D$ (counted with multiplicity). Set
\begin {equation}
\label {eq:gp}
 \G(\p') = - Y(D)  - (n-1)\sum_{l=1}^k (c_l(E) - f_l(E) - \epsilon_l(E) +1).
\end {equation} 
Here $c_l, f_l, \epsilon_l$ are as in Section~\ref{sec:uleft}. The
second term in the formula above is chosen so that adding to $D$ a
column through some $O_{i_l}$ does not change the value of $\G$; indeed, from the proof of Lemma~\ref{lemma:Der}, we have that $c_l(E) - f_l(E) - \epsilon_l(E) +1=v_l(D')$, where $D'$ is the domain of a polygon inducing the enhanced domain $E$. Introducing this term
$$-(n-1)\sum_{l=1}^k (c_l(E) - f_l(E) - \epsilon_l(E) +1)$$ can be interpreted as marking $n-1$ additional $Y$'s in the lower left quadrant of each disk bounded by the small beta curve around an $O_{i_l}$. 

We require that the quantity $\G$ is well-defined, i.e., it should not depend on $D$, but only on
the pair $\p'$. For this to be true, in view of Lemma~\ref{lem:EquivEnhanced}, we need to ensure that the addition of a periodic domain to $D$ (of the kinds listed in Definition~\ref{def:EquivEnhanced}) 
does not change $\G$. We have already made sure that $\G$ is unaffected by adding:
\begin{itemize}
\item Periodic domains of type (a) corresponding to free $O_{i_l}$ marked for destabilization; that is, the column minus the row through some $O_{i_l}$;
\item Periodic domains of type (c), that is, columns through the $O_{i_l}$'s.
\end{itemize}
We are left with the periodic domains of type (a) that correspond to free $O$'s not marked for destabilization, and periodic domains of type (b). Note that those of type (b) are linear combinations of  domains of the following form: a column minus a row through the same linked $O$ marking. 

By a judicious choice of the quantities $\alpha_{u,v}$, for $u, v \in \{1, \dots, n-k\}$, we will actually arrange for something slightly stronger than what we need. We will make it so that $\G$ is unaffected by adding any domain of the form: a column minus row though some $O$ not marked for destabilization. 

Note that there is a permutation $\sigma$ of $\{1, \dots, n-k\}$ such that the unmarked $O$'s are in the squares $s_{u\sigma(u)}$. There are $n-k$ conditions that we need to impose on $\alpha_{u,v}$, namely
\begin {equation}
\label {eq:linear}
 \sum_{v=1}^{n-k} \alpha_{uv} - \sum_{v=1}^{n-k} \alpha_{v\sigma(u)} = t_u,  
\end {equation}
for $u=1, \dots, n-k$. Here $t_u$ are determined by the number of $Y$'s that we already marked in the respective row and column (as specified above, in squares not marked by $T$), with an extra contribution in the case of the row just below some $O_{i_l}$ and the column just to the left of $O_{i_l}$, to account for the term $c_l(E) - f_l(E) - \epsilon_l(E) +1$ from Equation~\eqref{eq:gp}.

Note that $\sum_{u=1}^{n-k} t_u = 0$. We claim that there exists a solution
(in rational numbers) to the linear system described in
Equation~\eqref{eq:linear}. Indeed, the system is described by a $(n-k)$-by-$(n-k)^2$ matrix $A$, each of whose columns contains one $1$ entry, one $-1$ entry, and the rest just zeros. If a vector $(\beta_u)_{u=1, \dots, n-k}$ is in the kernel of the transpose $A^t$, it must have $\beta_u - \beta_v = 0$ for all $u, v$. In other words, $\im(A)^\perp = \ker(A^t)$ is the span of $(1, \dots, 1)$, so $(t_1, \dots, t_{n-k})$ must be in the image of $A$.

By multiplying all the values in a rational solution of \eqref{eq:linear} by a large integer (and also multiplying the number of $Y$'s initially placed in the rows and columns of
the destabilized $O$'s by the same integer), we can obtain a solution of \eqref{eq:linear} in integers.  By adding a sufficiently large constant to the $\alpha_{uv}$ (but not to the number of $Y$'s initially
placed), we can then obtain a solution in nonnegative integers, which we
take to be our definition of $\alpha_{uv}$.

We have now arranged so that $\G$ is an invariant of $\p'$. Moreover, pre-composing with a rectangle can only decrease $\G$, and it keeps $\G$ the same only when the rectangle (which
a priori has to be supported in one of the rows and columns through
some $O_{i_l}$) is actually supported in the column through $O_{i_j}$,
and does not contain the square right below $O_{i_j}$.

It follows that $\G$ is indeed a filtration on $C\F^*(G, \zed, \p)$. We
denote the connected components of the associated graded complex by $C\G^*(G,\zed, \p')$; it suffices to show that these have zero cohomology. Without
loss of generality, we will focus on $C\G^*(G, \zed, \p)$.

The complex $C\G^*(G, \zed, \p)$ can only contain pairs $[\x', \tilde
\y]$ such that $\x'$ differs from $\x$ by either pre- or
post-composition with a rectangle supported in the column through
$O_{i_j}$. The condition that this rectangle does not contain the
square right below $O_{i_j}$ is automatic, because $\x$ and $\x'$ are
outer.

We find that there can be at most two elements in $C\G^*(G, \zed, \p)$,
namely $\p = [\x, \tilde \y]$ and $\p' = [\x', \tilde \y]$, where
$\x'$ is obtained from $\x$ by switching the horizontal coordinates of
the components of $\x$ in the two vertical circles bordering the
column of $O_{i_j}$. Provided that $\p'$ is positive, the pairs $\p$
and $\p'$ are related by a differential, so the cohomology of
$C\G^*(G, \zed, \p)$ would indeed be zero.

Therefore, the last thing to be checked is that $\p'$ is positive. We
know that $\p$ is positive, so we can choose a positive enhanced
domain $D$ representing $\p$. Recall that the destabilization point
near $O_{i_j}$ is denoted $p_j$ and is part of $\y$. Draw a vertical
segment $S$ going down the vertical circle from $p_j$ to a point of
$\x$. There are three cases:
\begin {enumerate}
\item There is no point of $\x'$ on the segment $S$. Then there exists
  a rectangle going from $\x'$ to $\x$, and $\p'$ appears in $d \p$ by
  pre-composition. Adding the rectangle to $D$ preserves positivity.
\item\label{case:del-rect} There is a point of $\x'$ on $S$, and the multiplicity of $D$
  just to the right of $S$ is positive. Then there is a rectangle,
  just to the right of $S$, going from $\x$ to $\x'$. We get a
  positive representative for $\p'$ by subtracting this rectangle from
  $D$.
\item There is a point of $\x'$ on $S$, and the multiplicity of $D$ is
  zero somewhere just to the right of $S$. Note that as we cross the
  segment $S$ from left to right the drop in multiplicity is
  constant; since $D$ is positive, this drop must be
  nonnegative. In particular, $c_j \geq
  d_j$. Relation~\eqref{eq:abcd1} implies $a_j \geq b_j + 1$. Looking
  at the inequalities in~\eqref{ineq:abcd1}, we see that we can use
  two of them (the ones involving $b_j$ and $d_j$) to improve the
  other two:
  \begin {equation}
    \label {ineq:abcd2}
    a_j \geq b_j + 1 \geq f_j + 1, \qquad c_j \geq d_j \geq f_j + \epsilon_j.
  \end {equation}
  Let us add to $D$ the periodic domain given by the column through
  $O_{i_j}$. This increases $b_j, d_j$, and $f_j$ by $1$ while keeping
  $a_j$ and $c_j$ constant. Nevertheless, Inequality~\eqref{ineq:abcd2} shows
  that the inequalities~\eqref{ineq:abcd1} are still satisfied for
  the new domain $\tilde D$. Thus $\tilde D$ is positive, and its
  multiplicity just to the right of $S$ is everywhere nonnegative. We
  can then subtract a rectangle from $\tilde D$ to obtain a positive
  representative for $\p'$ as in case~\eqref{case:del-rect}.
\end {enumerate}

The three cases are pictured in Figure~\ref{fig:3cases}. \end {proof}

\begin{figure}
\begin{center}
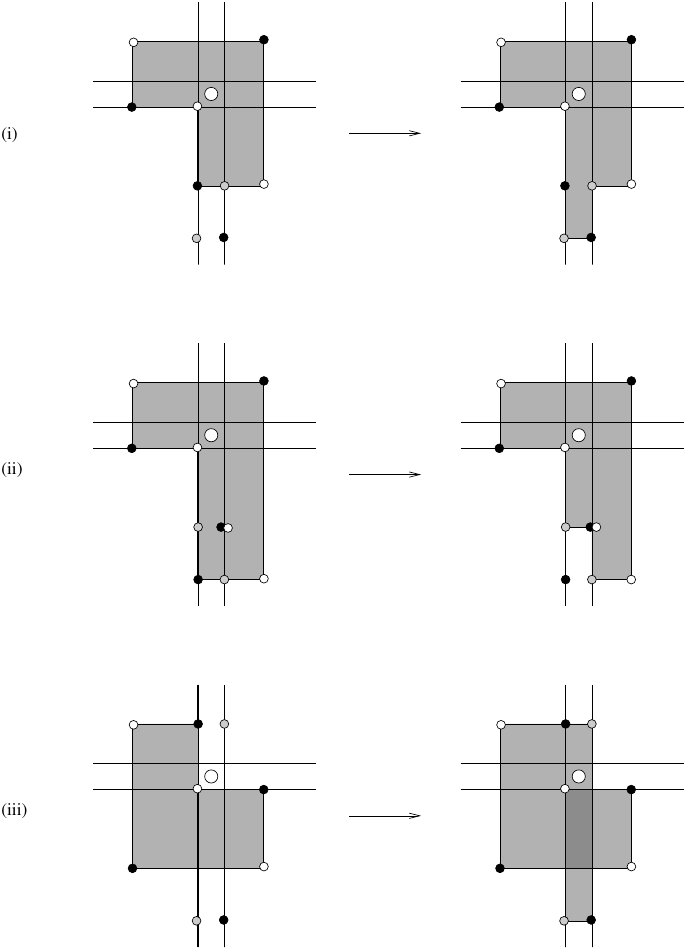
\end{center}
\caption {{\bf The three cases in Lemma~\ref{lemma:outer}.} On the
  left of each picture we show a positive $L$ domain representing
  $\p=[\x, \tilde \y]$. These domains have index $0$, $2$ and $2$,
  respectively (assuming that all the real multiplicities are
  zero). On the right of each picture we show the corresponding
  positive domain representing $\p' = [\x', \tilde \y]$. These domains
  have all index $1$. The components of $\x$ are the black dots, the
  components of $\x'$ the gray dots, and those of $\y$ the white
  dots. The segment $S$ is drawn thicker.}
\label{fig:3cases}
\end{figure}

\begin {proof}[Proof of Theorem~\ref{thm:sparse}]
  The case when $k=0$ is trivial, and the case $k=1$ follows from
  Proposition~\ref{prop:oneO}. For $k \geq 2$, Lemma~\ref{lemma:inner}
  says that for any pair $[\x, \tilde \y]$ of index $< 2-k$, the generator $\x$ is
  outer. Lemma~\ref{lemma:outer} then shows that the homology of
  $\grcp$ is zero in the given range, which implies the same for
  $\HP^*(G, \zed)$.
\end {proof}

\bibliographystyle{custom}
\bibliography{biblio}

\end{document}

%% file: sparse.pdf_t
\begin{picture}(0,0)%
\includegraphics{sparse.pdf}%
\end{picture}%
\setlength{\unitlength}{2983sp}%
\begingroup\makeatletter\ifx\SetFigFont\undefined%
\gdef\SetFigFont#1#2#3#4#5{%
  \reset@font\fontsize{#1}{#2pt}%
  \fontfamily{#3}\fontseries{#4}\fontshape{#5}%
  \selectfont}%
\fi\endgroup%
\begin{picture}(7674,3624)(214,-5248)
\end{picture}%

%% file: betaz.pdf_t
\begin{picture}(0,0)%
\includegraphics{betaz.pdf}%
\end{picture}%
\setlength{\unitlength}{3947sp}%
\begingroup\makeatletter\ifx\SetFigFont\undefined%
\gdef\SetFigFont#1#2#3#4#5{%
  \reset@font\fontsize{#1}{#2pt}%
  \fontfamily{#3}\fontseries{#4}\fontshape{#5}%
  \selectfont}%
\fi\endgroup%
\begin{picture}(4824,3859)(1039,-4133)
\put(3071,-2071){\makebox(0,0)[lb]{\smash{{\SetFigFont{12}{14.4}{\rmdefault}{\mddefault}{\updefault}{\color[rgb]{0,0,0}$Z_1$}%
}}}}
\put(4271,-2671){\makebox(0,0)[lb]{\smash{{\SetFigFont{12}{14.4}{\rmdefault}{\mddefault}{\updefault}{\color[rgb]{0,0,0}$Z_2$}%
}}}}
\end{picture}%

%% file: destabs.pdf_t
\begin{picture}(0,0)%
\includegraphics{destabs.pdf}%
\end{picture}%
\setlength{\unitlength}{3564sp}%
\begingroup\makeatletter\ifx\SetFigFont\undefined%
\gdef\SetFigFont#1#2#3#4#5{%
  \reset@font\fontsize{#1}{#2pt}%
  \fontfamily{#3}\fontseries{#4}\fontshape{#5}%
  \selectfont}%
\fi\endgroup%
\begin{picture}(4167,3996)(796,-3829)
\put(2836,-1906){\makebox(0,0)[lb]{\smash{{\SetFigFont{10}{12.0}{\rmdefault}{\mddefault}{\updefault}{\color[rgb]{0,0,0}$\De^{(Z_1, Z_2)} + \De^{(Z_2, Z_1)}$}%
}}}}
\put(3016,-3751){\makebox(0,0)[lb]{\smash{{\SetFigFont{10}{12.0}{\rmdefault}{\mddefault}{\updefault}{\color[rgb]{0,0,0}$\CFm(\Hyper^{\orM}_G(\{Z_1,Z_2\}))$}%
}}}}
\put(2926,-16){\makebox(0,0)[lb]{\smash{{\SetFigFont{10}{12.0}{\rmdefault}{\mddefault}{\updefault}{\color[rgb]{0,0,0}$\CFm(\Hyper^{\orM}_G(\emptyset))$}%
}}}}
\put(4906,-1861){\makebox(0,0)[lb]{\smash{{\SetFigFont{10}{12.0}{\rmdefault}{\mddefault}{\updefault}{\color[rgb]{0,0,0}$\CFm(\Hyper^{\orM}_G(\{Z_2\}))$}%
}}}}
\put(4411,-2851){\makebox(0,0)[lb]{\smash{{\SetFigFont{10}{12.0}{\rmdefault}{\mddefault}{\updefault}{\color[rgb]{0,0,0}$\De^{(Z_1)}$}%
}}}}
\put(4276,-781){\makebox(0,0)[lb]{\smash{{\SetFigFont{10}{12.0}{\rmdefault}{\mddefault}{\updefault}{\color[rgb]{0,0,0}$\De^{(Z_2)}$}%
}}}}
\put(811,-1861){\makebox(0,0)[lb]{\smash{{\SetFigFont{10}{12.0}{\rmdefault}{\mddefault}{\updefault}{\color[rgb]{0,0,0}$\CFm(\Hyper^{\orM}_G(\{Z_1\}))$}%
}}}}
\put(2836,-1591){\makebox(0,0)[lb]{\smash{{\SetFigFont{10}{12.0}{\rmdefault}{\mddefault}{\updefault}{\color[rgb]{0,0,0}$\De^{\{Z_1, Z_2\}} = $}%
}}}}
\put(2026,-781){\makebox(0,0)[lb]{\smash{{\SetFigFont{10}{12.0}{\rmdefault}{\mddefault}{\updefault}{\color[rgb]{0,0,0}$\De^{(Z_1)}$}%
}}}}
\put(1801,-2806){\makebox(0,0)[lb]{\smash{{\SetFigFont{10}{12.0}{\rmdefault}{\mddefault}{\updefault}{\color[rgb]{0,0,0}$\De^{(Z_2)}$}%
}}}}
\end{picture}%

%% file: annuli.pdf_t
\begin{picture}(0,0)%
\includegraphics{annuli.pdf}%
\end{picture}%
\setlength{\unitlength}{3646sp}%
\begingroup\makeatletter\ifx\SetFigFont\undefined%
\gdef\SetFigFont#1#2#3#4#5{%
  \reset@font\fontsize{#1}{#2pt}%
  \fontfamily{#3}\fontseries{#4}\fontshape{#5}%
  \selectfont}%
\fi\endgroup%
\begin{picture}(4254,2094)(3049,-4033)
\put(6110,-3089){\makebox(0,0)[lb]{\smash{{\SetFigFont{9}{10.8}{\rmdefault}{\mddefault}{\updefault}{\color[rgb]{0,0,0}$Z_1$}%
}}}}
\put(3833,-3546){\makebox(0,0)[lb]{\smash{{\SetFigFont{9}{10.8}{\rmdefault}{\mddefault}{\updefault}{\color[rgb]{0,0,0}$Z_2$}%
}}}}
\end{picture}%

%% file: shadow.pdf_t
\begin{picture}(0,0)%
\includegraphics{shadow.pdf}%
\end{picture}%
\setlength{\unitlength}{1579sp}%
\begingroup\makeatletter\ifx\SetFigFont\undefined%
\gdef\SetFigFont#1#2#3#4#5{%
  \reset@font\fontsize{#1}{#2pt}%
  \fontfamily{#3}\fontseries{#4}\fontshape{#5}%
  \selectfont}%
\fi\endgroup%
\begin{picture}(12780,4654)(225,-4229)
\put(3571,-2091){\makebox(0,0)[lb]{\smash{{\SetFigFont{5}{6.0}{\rmdefault}{\mddefault}{\updefault}{\color[rgb]{0,0,0}$Z$}%
}}}}
\put(1276,-3084){\makebox(0,0)[lb]{\smash{{\SetFigFont{12}{14.4}{\rmdefault}{\mddefault}{\updefault}{\color[rgb]{0,0,0}*}%
}}}}
\put(2693,-1666){\makebox(0,0)[lb]{\smash{{\SetFigFont{12}{14.4}{\rmdefault}{\mddefault}{\updefault}{\color[rgb]{0,0,0}*}%
}}}}
\put(2693,-3084){\makebox(0,0)[lb]{\smash{{\SetFigFont{12}{14.4}{\rmdefault}{\mddefault}{\updefault}{\color[rgb]{0,0,0}*}%
}}}}
\put(4111,-3084){\makebox(0,0)[lb]{\smash{{\SetFigFont{12}{14.4}{\rmdefault}{\mddefault}{\updefault}{\color[rgb]{0,0,0}*}%
}}}}
\put(4111,-1666){\makebox(0,0)[lb]{\smash{{\SetFigFont{12}{14.4}{\rmdefault}{\mddefault}{\updefault}{\color[rgb]{0,0,0}*}%
}}}}
\put(8835,-3084){\makebox(0,0)[lb]{\smash{{\SetFigFont{12}{14.4}{\rmdefault}{\mddefault}{\updefault}{\color[rgb]{0,0,0}*}%
}}}}
\put(8835,-1666){\makebox(0,0)[lb]{\smash{{\SetFigFont{12}{14.4}{\rmdefault}{\mddefault}{\updefault}{\color[rgb]{0,0,0}*}%
}}}}
\put(10252,-1666){\makebox(0,0)[lb]{\smash{{\SetFigFont{12}{14.4}{\rmdefault}{\mddefault}{\updefault}{\color[rgb]{0,0,0}*}%
}}}}
\put(10252,-3084){\makebox(0,0)[lb]{\smash{{\SetFigFont{12}{14.4}{\rmdefault}{\mddefault}{\updefault}{\color[rgb]{0,0,0}*}%
}}}}
\put(11670,-3084){\makebox(0,0)[lb]{\smash{{\SetFigFont{12}{14.4}{\rmdefault}{\mddefault}{\updefault}{\color[rgb]{0,0,0}*}%
}}}}
\put(11670,-1666){\makebox(0,0)[lb]{\smash{{\SetFigFont{12}{14.4}{\rmdefault}{\mddefault}{\updefault}{\color[rgb]{0,0,0}*}%
}}}}
\end{picture}%

%% file: ULeft.pdf_t
\begin{picture}(0,0)%
\includegraphics{ULeft.pdf}%
\end{picture}%
\setlength{\unitlength}{789sp}%
\begingroup\makeatletter\ifx\SetFigFont\undefined%
\gdef\SetFigFont#1#2#3#4#5{%
  \reset@font\fontsize{#1}{#2pt}%
  \fontfamily{#3}\fontseries{#4}\fontshape{#5}%
  \selectfont}%
\fi\endgroup%
\begin{picture}(35870,22668)(5766,-18595)
\end{picture}%

%% file: abcd.pdf_t
\begin{picture}(0,0)%
\includegraphics{abcd.pdf}%
\end{picture}%
\setlength{\unitlength}{2842sp}%
\begingroup\makeatletter\ifx\SetFigFont\undefined%
\gdef\SetFigFont#1#2#3#4#5{%
  \reset@font\fontsize{#1}{#2pt}%
  \fontfamily{#3}\fontseries{#4}\fontshape{#5}%
  \selectfont}%
\fi\endgroup%
\begin{picture}(2424,2424)(889,-2773)
\put(1576,-1261){\makebox(0,0)[lb]{\smash{{\SetFigFont{12}{14.4}{\rmdefault}{\mddefault}{\updefault}{\color[rgb]{0,0,0}$a_j$}%
}}}}
\put(1576,-2011){\makebox(0,0)[lb]{\smash{{\SetFigFont{12}{14.4}{\rmdefault}{\mddefault}{\updefault}{\color[rgb]{0,0,0}$c_j$}%
}}}}
\put(2401,-1261){\makebox(0,0)[lb]{\smash{{\SetFigFont{12}{14.4}{\rmdefault}{\mddefault}{\updefault}{\color[rgb]{0,0,0}$b_j$}%
}}}}
\put(2401,-2011){\makebox(0,0)[lb]{\smash{{\SetFigFont{12}{14.4}{\rmdefault}{\mddefault}{\updefault}{\color[rgb]{0,0,0}$d_j$}%
}}}}
\end{picture}%

%% file: local.pdf_t
\begin{picture}(0,0)%
\includegraphics{local.pdf}%
\end{picture}%
\setlength{\unitlength}{2403sp}%
\begingroup\makeatletter\ifx\SetFigFont\undefined%
\gdef\SetFigFont#1#2#3#4#5{%
  \reset@font\fontsize{#1}{#2pt}%
  \fontfamily{#3}\fontseries{#4}\fontshape{#5}%
  \selectfont}%
\fi\endgroup%
\begin{picture}(5559,4569)(709,-5248)
\put(3511,-2581){\makebox(0,0)[lb]{\smash{{\SetFigFont{10}{12.0}{\rmdefault}{\mddefault}{\updefault}{\color[rgb]{0,0,0}$\rho_j$}%
}}}}
\put(2656,-2581){\makebox(0,0)[lb]{\smash{{\SetFigFont{10}{12.0}{\rmdefault}{\mddefault}{\updefault}{\color[rgb]{0,0,0}$u_j$}%
}}}}
\put(5266,-1636){\makebox(0,0)[lb]{\smash{{\SetFigFont{10}{12.0}{\rmdefault}{\mddefault}{\updefault}{\color[rgb]{0,0,0}$b_j$}%
}}}}
\put(5176,-4471){\makebox(0,0)[lb]{\smash{{\SetFigFont{10}{12.0}{\rmdefault}{\mddefault}{\updefault}{\color[rgb]{0,0,0}$d_j$}%
}}}}
\put(1396,-1591){\makebox(0,0)[lb]{\smash{{\SetFigFont{10}{12.0}{\rmdefault}{\mddefault}{\updefault}{\color[rgb]{0,0,0}$a_j$}%
}}}}
\put(1396,-4246){\makebox(0,0)[lb]{\smash{{\SetFigFont{10}{12.0}{\rmdefault}{\mddefault}{\updefault}{\color[rgb]{0,0,0}$c_j$}%
}}}}
\put(3511,-3661){\makebox(0,0)[lb]{\smash{{\SetFigFont{10}{12.0}{\rmdefault}{\mddefault}{\updefault}{\color[rgb]{0,0,0}$w_j$}%
}}}}
\put(2656,-3661){\makebox(0,0)[lb]{\smash{{\SetFigFont{10}{12.0}{\rmdefault}{\mddefault}{\updefault}{\color[rgb]{0,0,0}$v_j$}%
}}}}
\end{picture}%

%% file: ULeft2.pdf_t
\begin{picture}(0,0)%
\includegraphics{ULeft2.pdf}%
\end{picture}%
\setlength{\unitlength}{789sp}%
\begingroup\makeatletter\ifx\SetFigFont\undefined%
\gdef\SetFigFont#1#2#3#4#5{%
  \reset@font\fontsize{#1}{#2pt}%
  \fontfamily{#3}\fontseries{#4}\fontshape{#5}%
  \selectfont}%
\fi\endgroup%
\begin{picture}(33379,21964)(1583,-16728)
\end{picture}%

%% file: StabCancel.pdf_t
\begin{picture}(0,0)%
\includegraphics{StabCancel.pdf}%
\end{picture}%
\setlength{\unitlength}{789sp}%
\begingroup\makeatletter\ifx\SetFigFont\undefined%
\gdef\SetFigFont#1#2#3#4#5{%
  \reset@font\fontsize{#1}{#2pt}%
  \fontfamily{#3}\fontseries{#4}\fontshape{#5}%
  \selectfont}%
\fi\endgroup%
\begin{picture}(8270,10668)(-234,-6595)
\end{picture}%

%% file: inner.pdf_t
\begin{picture}(0,0)%
\includegraphics{inner.pdf}%
\end{picture}%
\setlength{\unitlength}{2368sp}%
\begingroup\makeatletter\ifx\SetFigFont\undefined%
\gdef\SetFigFont#1#2#3#4#5{%
  \reset@font\fontsize{#1}{#2pt}%
  \fontfamily{#3}\fontseries{#4}\fontshape{#5}%
  \selectfont}%
\fi\endgroup%
\begin{picture}(5124,5124)(1189,-5173)
\end{picture}%

%% file: equality.pdf_t
\begin{picture}(0,0)%
\includegraphics{equality.pdf}%
\end{picture}%
\setlength{\unitlength}{2368sp}%
\begingroup\makeatletter\ifx\SetFigFont\undefined%
\gdef\SetFigFont#1#2#3#4#5{%
  \reset@font\fontsize{#1}{#2pt}%
  \fontfamily{#3}\fontseries{#4}\fontshape{#5}%
  \selectfont}%
\fi\endgroup%
\begin{picture}(10374,4524)(439,-4273)
\put(1916,-1196){\makebox(0,0)[lb]{\smash{{\SetFigFont{7}{8.4}{\rmdefault}{\mddefault}{\updefault}{\color[rgb]{0,0,0}$f_j - 1$}%
}}}}
\put(1931,-2291){\makebox(0,0)[lb]{\smash{{\SetFigFont{7}{8.4}{\rmdefault}{\mddefault}{\updefault}{\color[rgb]{0,0,0}$O_{i_j}$}%
}}}}
\put(1196,-2096){\makebox(0,0)[lb]{\smash{{\SetFigFont{7}{8.4}{\rmdefault}{\mddefault}{\updefault}{\color[rgb]{0,0,0}$f_j$}%
}}}}
\put(2021,-2006){\makebox(0,0)[lb]{\smash{{\SetFigFont{7}{8.4}{\rmdefault}{\mddefault}{\updefault}{\color[rgb]{0,0,0}$f_j$}%
}}}}
\put(1136,-3011){\makebox(0,0)[lb]{\smash{{\SetFigFont{7}{8.4}{\rmdefault}{\mddefault}{\updefault}{\color[rgb]{0,0,0}$f_j$}%
}}}}
\put(1931,-3026){\makebox(0,0)[lb]{\smash{{\SetFigFont{7}{8.4}{\rmdefault}{\mddefault}{\updefault}{\color[rgb]{0,0,0}$f_j + 1$}%
}}}}
\put(2816,-2081){\makebox(0,0)[lb]{\smash{{\SetFigFont{7}{8.4}{\rmdefault}{\mddefault}{\updefault}{\color[rgb]{0,0,0}$f_j - 1$}%
}}}}
\put(2816,-1196){\makebox(0,0)[lb]{\smash{{\SetFigFont{7}{8.4}{\rmdefault}{\mddefault}{\updefault}{\color[rgb]{0,0,0}$f_j - 1$}%
}}}}
\end{picture}%

%% file: ys.pdf_t
\begin{picture}(0,0)%
\includegraphics{ys.pdf}%
\end{picture}%
\setlength{\unitlength}{2763sp}%
\begingroup\makeatletter\ifx\SetFigFont\undefined%
\gdef\SetFigFont#1#2#3#4#5{%
  \reset@font\fontsize{#1}{#2pt}%
  \fontfamily{#3}\fontseries{#4}\fontshape{#5}%
  \selectfont}%
\fi\endgroup%
\begin{picture}(3624,3624)(889,-3373)
\put(4080,-3114){\makebox(0,0)[lb]{\smash{{\SetFigFont{7}{8.4}{\rmdefault}{\mddefault}{\updefault}{\color[rgb]{0,0,0}$O_6$}%
}}}}
\put(2267,-1319){\makebox(0,0)[lb]{\smash{{\SetFigFont{7}{8.4}{\rmdefault}{\mddefault}{\updefault}{\color[rgb]{0,0,0}$O_3$}%
}}}}
\put(1091,-116){\makebox(0,0)[lb]{\smash{{\SetFigFont{7}{8.4}{\rmdefault}{\mddefault}{\updefault}{\color[rgb]{0,0,0}$O_1$}%
}}}}
\put(2853,-1905){\makebox(0,0)[lb]{\smash{{\SetFigFont{7}{8.4}{\rmdefault}{\mddefault}{\updefault}{\color[rgb]{0,0,0}$O_4$}%
}}}}
\put(1531,-850){\makebox(0,0)[lb]{\smash{{\SetFigFont{7}{8.4}{\rmdefault}{\mddefault}{\updefault}{\color[rgb]{0,0,0}$O_2$}%
}}}}
\put(3327,-2659){\makebox(0,0)[lb]{\smash{{\SetFigFont{7}{8.4}{\rmdefault}{\mddefault}{\updefault}{\color[rgb]{0,0,0}$O_5$}%
}}}}
\end{picture}%

%% file: 3cases.pdf_t
\begin{picture}(0,0)%
\includegraphics{3cases.pdf}%
\end{picture}%
\setlength{\unitlength}{2763sp}%
\begingroup\makeatletter\ifx\SetFigFont\undefined%
\gdef\SetFigFont#1#2#3#4#5{%
  \reset@font\fontsize{#1}{#2pt}%
  \fontfamily{#3}\fontseries{#4}\fontshape{#5}%
  \selectfont}%
\fi\endgroup%
\begin{picture}(7827,10845)(-464,-10894)
\end{picture}%